\documentclass{amsart}
\usepackage[dvipsnames]{xcolor}

\usepackage{amsmath, amsfonts, amssymb, amsbsy, bigstrut, graphicx, enumerate,  upref, longtable, comment, bbm}
\usepackage{amsmath}
\usepackage{pstricks,csquotes}
\allowdisplaybreaks
\usepackage{dsfont,mathtools}
\usepackage{bbm}
\usepackage{mathalpha}
\usepackage[colorlinks=true,
            linkcolor=blue,
            citecolor=blue,
            urlcolor=blue]{hyperref}
\usepackage[numbers, sort&compress]{natbib} 
\usepackage{euscript}

\usepackage[noabbrev,capitalize]{cleveref}
\usepackage[top=1in, bottom=1in, left=1in, right=1in]{geometry}

\newcommand{\mv}{\mathfrak{v}}

\newcommand{\sm}{s_{\theta,p}}

\newcommand{\E}{\mathbb{E}}

\newcommand{\reihaneh}[1]{\textcolor{purple}{(Reihaneh comment: #1)}}
\newtheorem{thm}{Theorem}
\newtheorem{lmm}[thm]{Lemma}

\newtheorem{corr}[thm]{Corollary}
\newtheorem{prop}[thm]{Proposition}
\newtheorem{defn}{Definition}

\theoremstyle{definition}
\newtheorem{remark}{Remark}
\newtheorem{ex}{Example}

\numberwithin{thm}{section}
\numberwithin{remark}{section}
\numberwithin{ex}{section}
\numberwithin{assm}{section}
\numberwithin{defn}{section}
\newcommand{\argmin}{\operatorname{argmin}}

\newcommand{\R}{\mathbb{R}}

\renewcommand{\P}{\mathbb{P}}

\DeclareMathOperator*{\arginf}{arg\,inf}

\usepackage{mathrsfs}

\numberwithin{equation}{section}

\allowdisplaybreaks
\begin{document}

\title{LDP for tensor forms}
\author[Malekian]{Reihaneh Malekian}
\address{R. \ Malekian\hfill\break
	Department of Statistics\\ Columbia University\\ New York, NY 10027, USA.}
\email{rm3942@columbia.edu}
\author[Bhattacharya]{Sohom Bhattacharya}
\address{S. \ Bhattacharya\hfill\break
	Department of Statistics\\ University of Florida\\ Gainesville, FL 32603, USA.}
\email{bhattacharya.s@ufl.edu}
\author[Deb]{Nabarun Deb}
\address{N. \ Deb\hfill\break
	Econometrics and Statistics\\ University of Chicago Booth School of Business\\ Chicago, IL 60637, USA.}
\email{nabarun.deb@chicagobooth.edu}

\author[Mukherjee]{Sumit Mukherjee}
\address{S. \ Mukherjee\hfill\break
	Department of Statistics\\ Columbia University\\ New York, NY 10027, USA.}
\email{sm3949@columbia.edu }

\thanks{SM's research is partially supported by NSF grant DMS-2113414.}

\begin{abstract}
In this paper, we study the large deviation principle (LDP) for a tensor-weighted functional of i.i.d. random variables, when the sequence of tensors converges under a variant of the \enquote{bad} cut norm. Using the LDP, we analyze a Gibbs measure with a tensor-valued Hamiltonian, and characterize the optimizers of the limiting variational problem in terms of a functional fixed point equation. As applications, we focus on several concrete examples, which include monochromatic subgraph counts in sparse random graphs, Erd\H{o}s-R\'enyi hypergraphs, and a generalized Potts statistic of order $v\ge 2$. Studying the optimization problem, we give sufficient conditions for uniqueness of the optimizer, as well as for existence of constant optimizers (replica symmetry). Our results demonstrate universal weak laws for a large class of tensor Gibbs models with approximately regular tensors.

% MORE SUGGESTIONS FOR A NAME:\newline1.LDP for Tensor forms
% %Generalized Potts Model.
% \newline 2.LDP for Tensor Interaction Models: The Generalized Potts Case \newline 3.LDP for Tensor Interaction Models(talking about potts on abstract)\newline 4.LDP for Hypergraphon Interaction Models(talking about potts on abstract)

\end{abstract}

% \keywords{Graph limits, magnetization, phase transition, replica-symmetry, tensor Ising model}

\maketitle

\section{Introduction}

Let ${\bf X}=(X_1,\ldots,X_n)$ be i.i.d. random variables from a probability measure $\mu$ on $[c]:=\{1,\ldots,c\}$, where $c,n\ge 2$ are positive integers, and $\mu_r:=\mathbb P(X_1=r)>0$ for all $r\in[c]$.  %${\bf X}, \boldsymbol{\alpha}, \boldsymbol{f}, \ldots$.
For a positive integer $v\ge 2$, let $Q_n:[n]^v \to \mathbb{R}$ be a $v$-tensor which satisfies the following two properties:
\begin{enumerate}
    \item[(i)] ({Symmetric}) 
$Q_n(i_1,\dots,i_v)=Q_n(i_{\sigma(1)},\ldots,i_{\sigma(v)})$, for all $\sigma\in S_v$, where $S_v$ denotes the set of all permutations of $[v]$.
\item[(ii)] ({Zero-diagonal}) $Q_n(i_1,\ldots,i_v)=0$ if $i_1,\cdots,i_v$ are not all distinct.
\end{enumerate}
Let $\phi:[c]^v\to \R$ be an arbitrary function. %The function $\phi$ determines how the states interact across $v$ coordinates. In particular, different choices of $\phi$ recover a wide range of observables, including weighted subgraph counts and energy functionals ({\color{red} Again highlight with examples. Use the example environment to highlight maybe 2 or 3 choices where you vary $Q_n$ and also $\phi$. In each example, have maybe 2 citations. The $c=2$ example can be taken from the Gibbs paper).}
In this paper we study the large deviation behavior of the tensor form
\begin{align}\label{hamiltoniandef}
    N_n(Q_n,\phi,{{\bf X}}):=\frac{1}{n^v}\sum_{i_1,\cdots,i_v=1}^nQ_n(i_1,\ldots,i_v)\phi(X_{i_1},\ldots, X_{i_v}).
\end{align}

In the very special case when $v=1$ and $Q_n(i)=1$  for all $i\in [n]$, the random variable $N_n(Q_n,\phi,{\bf X})$ reduces to the empirical average $$\frac{1}{n}\sum_{i=1}^n\phi(X_i),$$
the LDP for which follows from the classical Cramér's theorem. 
% it specifies which collections of indices interact and with what strength. 
%\cite{Basak2017,wu1982potts,potts1952some,eichelsbacher2015rates,dembo2014replica,costeniuc2005complete,blanchard2008thermodynamic,gandolfo2010limit}). In this regime, one can exploit the structure of matrices together with tools from graphon/graph limit theory to control the interaction functional.
More generally,  when $Q_n\equiv 1$ and $v\ge 1$, the random variable \eqref{hamiltoniandef} reduces to a U-statistic (V-statistic), large deviation principles for which have been studied extensively in the literature  \cite{arcones1992large, eichelsbacher1995large, eichelsbacher2002large}. In the most general case, the tensor $Q_n$ encodes the underlying interaction structure. Such tensors naturally arise, for example, as adjacency tensors of random graphs and hypergraphs, or as interaction kernels in statistical physics models.  
The LDP for multilinear forms, and more generally for nonlinear functionals of i.i.d.~random variables, has been studied extensively in the literature 
%in the case where the tensor $Q_n$ is induced by a matrix, corresponding to pairwise interactions 
(see \cite{chatterjee2016nonlinear,yan2020nonlinear, augeri2020nonlinearlargedeviationbounds,Basak2017,jain2018meanfield,eldan2018gaussian,austin2019structure,lacker2022mean} and references therein). Our approach differs from the nonlinear large-deviation framework, in that we study a proper large deviation with a limiting rate function, whereas the existing literature mostly focuses on nonasymptotic finite-sample bounds. LDPs for similar multilinear forms have recently been studied in the literature (see ~\cite{bhattacharya2023gibbs,bhattacharya2024ldp}). The main difference is that instead of working with a rank-1 tensor characterized by a two-dimensional graphon and leveraging the counting lemma for two-dimensional $L^p$ graphons studied in   \cite[Theorem 2.18, Theorem 2.20]{BorgsLPII},
we work directly in the higher-order tensor setting under the topology of $\text{Cut}^*$ convergence (see \cref{defcutstar}), which is equivalent to the \enquote{bad} cut norm for tensors~\cite{zhao2015hypergraph}. It is well known that the bad cut norm does not admit a counting lemma for sub-hypergraph counts. Nevertheless, by proving a counting lemma for random graphs (see \cref{lem:uniform-core-discrepancy-ER}), we obtain an LDP for monochromatic subgraph counts under much weaker sparsity conditions in \cref{LDPforINSER}, compared to the existing literature  (see \cite[Corollary 1.7, Section 1.3]{bhattacharya2023gibbs}).
The key ingredient is \cref{lem:uniform-core-reduction-ER}, which allows us to estimate a subgraph count by the corresponding estimate for its 2-core graph. \cref{lem:uniform-core-discrepancy-ER} then invokes sharp large deviation bounds for subgraph counts in random graphs (see~\cite{BasakKarmakar2025UpperTailIrregular,BasakBasu2023UpperTailRegularLocalized,KozmaSamotijLowerTailsRelativeEntropy}) to verify the regularity lemma for random graphs. This allows us to get results for much wider sparsity regimes.

\noindent We now give some examples of random variables of the form \eqref{hamiltoniandef}, which are of interest.

\begin{ex}

    Assume $v=2$ and  $\phi(x,y)=\mathbf 1_{\{x=y\}}$. Then \eqref{hamiltoniandef} becomes 
    $$\frac{1}{n^2}\sum_{i,j} Q_n(i,j)\mathbf 1_{\{X_i=X_j\}},$$ which is the Hamiltonian/sufﬁcient statistic for the Potts model (a generalization of the Ising model) with coupling matrix $Q_n$ (see \cite{ising1925, potts1952some, ellis1990limit, Bresler2019,Rados2019}).

    \end{ex}

        \begin{ex}\label{TRI}

    Assume $v=3$, $\phi(x,y,z)=\mathbf 1_{\{x=y=z\}}$, and  $Q_n(i,j,k)=G_n(i,j)G_n(j,k) G_n(k,i)$, where $G_n$ is the (scaled) adjacency matrix of a graph $G_n$. Then \[\frac{1}{n^3}\sum_{i,j,k} G_n(i,j)G_n(j,k) G_n(k,i)\mathbf 1_{\{X_i=X_j=X_k\}}\] counts the (scaled) number of monochromatic triangles in
$G_n$, where the vertices of $G_n$ receive colors from $[c]$ according to the law $\mu$, independent of other vertices. Limiting distributions of such random variables have been studied in the literature (see \cite{ bhattacharya2019monochromatic,bhattacharya2022normal, fang2015universal, Bhattacharya2017,bhattacharya2020second} and references therein). 

In particular, if ${G}_n$ is an Erd\H{o}s-R\'enyi random graph with parameter $p$ on $n$ vertices, then the LDP for the above random variable has been studied in the literature when the Erd\H{o}s-R\'enyi parameter $p$ is fixed, in \cite{bhattacharya2024ldp}. Our results allow us to study the case when  $(n\log n)^{-1/2}\ll p\ll 1$.   In fact, we study the same problem for a general subgraph $H$ (instead of triangles), and allow $p=p_n$ to go to $0$ at some rate, depending on $H$. %However, our paper is able to address that partially.
%More generally, one can study the number of monochromatic subgraphs in a large graph, whose vertices are colored using a probability distribution $\mu$, independently across vertices.
    \end{ex}

% \begin{ex}
% In the setting of \cref{TRI}, suppose ${G}_n$ is a Erd\H{o}s-R\'enyi random graph with parameter $p$ on $n$ vertices. If $p=p_n\to 0$, then the LDP for ${n^{-3}}\sum_{i,j,k} G_n(i,j)G_n(j,k) G_n(k,i)\mathbf 1_{\{X_i=X_j=X_k\}}$ cannot be covered by existing literature; e.g.  \cite{chatterjee2011large,bhattacharya2024ldp}. However, our paper is able to address that partially.
% \end{ex}

  \begin{ex}\label{Hypex} Assume $\phi(x_1,\ldots,x_v)=\mathbf 1_{\{x_1= \ldots=x_v\}}$ and let ${G}^{(v)}_n$ be a $v$-uniform hypergraph over $n$ vertices. Let $Q_n(i_1,\ldots,i_v)$ be 1 if and only if $\{i_1,\ldots,i_v\}$ form a $v$-hyperedge in ${G}^{(v)}_n$ and otherwise zero. Then 
  \[\frac{1}{n^v}\sum_{i_1,\cdots,i_v=1}^nQ_n(i_1,\ldots,i_v)\mathbf 1_{\{X_{i_1}= \ldots=X_{i_v}\}}\]
counts the (scaled) number of monochromatic hyperedges in ${G}^{(v)}_n$ (see \cite{SomabhaBhattacharya2020}).
\end{ex}
    \begin{ex}[Rainbow triangles]
    Suppose we are in the setting of \cref{TRI}, but with $\phi(x,y,z)$ changed to $\mathbf 1_{\{x,y,z \,\text{are pairwise distinct}\}}$. Then \[\frac{1}{n^3}\sum_{i,j,k} G_n(i,j)G_n(j,k) G_n(k,i)\mathbf 1_{\{X_i, X_j, X_k \,\text{are pairwise distinct}\}}\] counts the (scaled) number of rainbow triangles (where each vertex has a different color) in
$G_n$. More generally, one can study the number
of rainbow subgraphs.
\end{ex}

\paragraph{\textbf{Main Contributions}}
Our main contributions can be summarized as follows:
\begin{itemize}
    \item[(i)] We establish an LDP for the tensor functional $N_n(Q_n,\phi, \mathbf{X})$ defined by~\eqref{hamiltoniandef} (\cref{NewLDP}) with speed $n$, under a convergence condition on the associated hypergraphons. The rate function admits an explicit variational form, and we also obtain convergence of the corresponding log-partition function for a Gibbs measure with Hamiltonian $N_n(Q_n,\phi, \mathbf{X})$, in terms of a variational problem. Our results go beyond existing LDP results, and cover a wide range of examples, including pairwise interactions and general $v$-tensors.
    \item[(ii)] To establish this result, we introduce a notion of convergence for $v$-hypergraphons via cut$^*$ norm (\cref{defcutstar}), and show the equivalence of this norm with the ``bad cut" norm (see \cite{zhao2015hypergraph}).
    
    %that convergence in this norm is sufficient to control the interaction functional and thus, obtain the exponential equivalence needed for the LDP.
      \item[(iii)] We derive a fixed-point equation in~\cref{prop:propotts2} that characterizes optimizers of the limiting variational problem from part (i). Utilizing this, we provide various necessary and (separately) sufficient conditions for the optimizers to be constant (replica symmetry).

    \item[(iv)] As an application of our framework, we derive the LDPs for colored sparse Erd\H{o}s-R\'enyi hypergraphs (\cref{LDPforERHyp}) and colored sparse Erd\H{o}s-R\'enyi graphs (\cref{LDPforINSER}).
  
    % \item[(iv)] As an application of our framework, we provide a general conditional LDP (\cref{LDPforINSER}) for random tensors that converge in cut$^*$, and derive LDPs for sparse Erd\H{o}s-R\'enyi hypergraph (\cref{LDPforERHyp}) and sparse inhomogeneous Erd\H{o}s-R\'enyi graph (\cref{LDPforINSER}).
    \item[(v)] As a further application, we consider a generalized Potts statistic with tensor interactions. We study an LDP for the generalized Potts statistic, as well as a Gibbs measure with the Potts statistic as the Hamiltonian. We derive a variational characterization of the rate function, and the limiting free energy, in terms of constrained and unconstrained optimization problems (\cref{prop:ldp+zn}). 
    We further characterize the optimizers of the unconstrained and the constrained optimization problems in \cref{prop:propotts} and \cref{OPT}, respectively. %analyzes the associated variational problem in special cases and derive explicit rate functions and optimizers, including regimes with multiple or non-constant solutions.
\end{itemize}

\section{Main Results}

% LDP for multilinear forms have been studied in the literature in the case where the tensor $Q_n$ is induced by a matrix, corresponding to pairwise interactions (see \cite{Cha2016,cook2023regularity, augeri2020nonlinearlargedeviationbounds}). In this regime, one can exploit the structure of matrices together with tools from graphon/graph limit theory to control the interaction functional.
% \reihaneh{Some papers to mention here?} {\color{red} Perhaps take a look at Somabha's papers and see the papers they cite.} 

% However, extending these results to general $v$-tensors introduces fundamentally new challenges. In particular, there is no canonical analogue of the graphon framework that simultaneously guarantees compactness of the underlying objects and continuity of the associated multilinear functionals \cite{lovasz2004limitsdensegraphsequences}. A key contribution of this work is to overcome this difficulty by introducing a framework that allows one to analyze large deviations for general tensor interactions.
% To capture the large-scale structure of the tensors $Q_n$, we embed them into a continuous framework analogous to graph limits. This allows us to study asymptotic properties of the interaction functional through suitable limit objects.
\begin{defn}[$v$-Hypergraphon]\label{vhypergraphon}
Let $\mathcal{W}_{v}$ be the space of all real-valued symmetric functions in $L^1([0,1]^v)$. We will refer to $\mathcal{W}_v$ as the space of \emph{``$v$-hypergraphon"}s, or simply, the space of hypergraphons. 
\\

For a symmetric and zero-diagonal $v$-tensor $Q_n$, we define the corresponding hypergraphon $W_{Q_n}\in \mathcal{W}_v$ by setting
\begin{align}\label{eq:funext}
W_{Q_n}(x_1,x_2,\ldots,x_v):=Q_n(i_1,i_2,\ldots,i_v)\quad\text{ if }\,\max\{\lceil nx_r\rceil ,1\}=i_r\text{ for all }r\in   [v]. 
    \end{align}
%where If there is no ambiguity, we write $W_{Q_n}$ as $W_n$.   
\end{defn}

    % Suppose $\mathcal{A}_{n}$ is the space of all real-valued functions $A$ on $[n]^v$ such that 

%This construction allows us to view the discrete tensor $Q_n$ as a continuous object, making it possible to analyze convergence and stability properties using tools from analysis.

% To establish our main results, we introduce a notion of convergence for sequences of $v$-tensors that is sufficiently strong to control the interaction functional. The associated cut$^*$ norm provides a natural extension of classical cut norms to higher-order interactions and may be of independent interest.

Throughout the paper, we work with the associated $v$-hypergraphons $\{W_{Q_n}\}_{n\ge 1}$ and assume that they converge in the cut$^*$ distance defined below (see also~\cite{BorgsLPII, borgs2018p, borgsdense1, borgsdense2, Lovasz2012} regarding the theory of 2-hypergraphons or graphons).

% Cut distance/cut metric that has been introduced in the Combinatorics literature to study limits of graphs and matrices (see~\cite{FriezeKannan1999}), and has received significant attention in the recent literature (\cite{bc_lpi,borgs2018p,borgsdense1,borgsdense2}). For more details on cut metric and its manifold applications, we refer the interested reader to \cite{Lovasz2012}.

\begin{defn}[Cut$^*$ Norm for $v$-Hypergraphons]\label{defcutstar}
 For $W\in \mathcal{W}_v$, define the Cut$^*$ norm by setting  
    \[
\|W\|_{\square^*}:=\sup _{S \subseteq[0,1]} \left|\int_{  S^v} W\left(x_1,\ldots, x_v\right) \prod_{i=1}^v d x_i\right|.\]
\end{defn}
Notice that for all $W\in \mathcal{W}_v$, we have $\|W\|_{\square^*}\le \|W\|_1$. It is straightforward to verify $\|\cdot\|_{\square^*}$ is a norm.
The following lemma shows that this norm is equivalent to the
``bad'' cut norm, which is an extension of the usual cut norm in two dimensions (see  \cite{zhao2015hypergraph}):
  \[\|W\|_{\square}:=\sup _{S_1,\ldots,S_v \subseteq[0,1]} \left|\int_{  S_1\times\ldots\times S_v} W\left(x_1,\ldots, x_v\right) \prod_{i=1}^v d x_i\right|.\]
\begin{lmm}\label{equivalence of two cut norms}
    The Cut$^*$ norm is equivalent to the bad cut norm, i.e. there exists a constant $C_v>0$, depending only on $v$, such that
    \[C_v\|W\|_{\square}\le\|W\|_{\square^*}\le  \|W\|_{\square},\]
    for any $W\in \mathcal{W}_v$.
\end{lmm}
The bad cut norm was referred to as \textit{bad}, since one cannot obtain a hypergraph version of the counting lemma for general $v$ with respect to (w.r.t.) this norm \cite[Section 3]{zhao2015hypergraph}. However, as we show below, it is enough to determine the large deviation behavior for the random variable $N_n(Q_n,\phi,{{\bf X}})$. 

% {\color{red} A natural question would be why not work with the bad cut directly? And add cut (star) topology as a tool in proofs?}

We now introduce the limiting functional that will govern both the large deviation rate function and the associated variational problem.
\begin{defn}\label{def:g2}
Let $\mathcal{F}_{c}$ denote the set of all measurable functions ${\boldsymbol f}=(f_1,\ldots,f_c):[0,1]\to [0,1]^c$ such that $\sum_{r=1}^c f_r(x)=1$ for all $x\in [0,1]$ (modulo almost everywhere equality).  For any $W\in \mathcal{W}_v$, define the functional ${G}_{W,\phi}(.): \mathcal{F}_c\to \R$ by 
\[
G_{W,\phi}({\boldsymbol f})
:=
\int_{[0,1]^v}
W(x_1,\ldots,x_v)\,\Phi_{\boldsymbol f}(x_1,\ldots,x_v)\,dx_1\cdots dx_v,
\]
% \[
% \mathscr{G}_{W,\phi}({\boldsymbol f})
% :=
% \int_{[0,1]^v}
% W(x_1,\ldots,x_v)\,\Phi_{\boldsymbol f}(x_1,\ldots,x_v)\,dx_1\cdots dx_v,
% \]
% \[
% {\bf{\mathscr{G}}}_{W,\phi}({\boldsymbol f})
% :=
% \int_{[0,1]^v}
% W(x_1,\ldots,x_v)\,\Phi_{\boldsymbol f}(x_1,\ldots,x_v)\,dx_1\cdots dx_v,
% \]
% \[\symcal{G}\]

\noindent where for ${\boldsymbol f}\in\mathcal F_c$,
\[
\Phi_{\boldsymbol f}(x_1,\ldots,x_v)
:=
\sum_{r_1,\ldots,r_v\in[c]}
\phi(r_1,\ldots,r_v)\prod_{a=1}^v f_{r_a}(x_a).
\]
\end{defn}
One can think of $G_{W,\phi}({\boldsymbol f})$ as the continuum analogue of the statistic $N_n(Q_n, \phi,{\bf X})$ defined in \eqref{hamiltoniandef} evaluated at $\boldsymbol{f}$.
%The large deviation behavior of $N_n(Q_n,\phi,{\bf X})$ is closely related to the asymptotics of the
We also define an associated Gibbs measure, with Hamiltonian $N_n(Q_n, \phi,{\bf X})$. %normalization constant encodes the macroscopic behavior of the system.
% ({\color{red} Depending on how we add references in the intro, we might want to add a couple references where some special case of the Gibbs measure below was studied}).

\begin{defn}
With $N_n(Q_n,\phi,{{\bf X}})$ as in \eqref{hamiltoniandef} and $\beta \in \R$, define the Gibbs measure
\begin{align}\label{eq:gibbs_potts}
	\frac{d\R_{Q_n,{ \beta},\mu}}{d\mu^{\otimes n}}(\mathbf{X}):=\exp\Big(n\beta N_n(Q_n,\phi,{{\bf X}})-nZ_{Q_n}(\beta,\mu)\Big),
\end{align}
where \begin{align}\label{eq:gibbs_pottsloppartition}
    Z_{Q_n}(\beta,\mu):=\frac{1}{n}\log \E_{\mu^{\otimes n}} \exp\Big(n \beta N_n(Q_n,\phi,{{\bf X}})\Big)
\end{align}
is the scaled log normalization constant.
\end{defn}
We now state our main result, which establishes an LDP for $N_n(Q_n,\phi,{\bf X})$ and derives a variational characterization of the limiting free energy.
%\begin{remark}
%    Notice that absolute homogeneity, and triangle inequality trivially hold for $\|\cdot\|_{\square^*}$. Moreover, positive definiteness is easy to verify, noticing that the Borel $\sigma$-algebra in $[0,1]^v$ is generated by product sets and then using \cref{equivalence of two cut norms}. Therefore, $\|\cdot\|_{\square^*}$ is indeed a norm.
%\end{remark}

\begin{thm}%[LDP and Variational Principle for Tensor Interaction Models]
\label{NewLDP}
%Suppose $\mu$ is a probability measure on $[c]$ with ${\bf{X}}=(X_1,\ldots,X_n)$, where $X_i \overset{\text{i.i.d.}}{\sim} \mu$ with $\P(X_1=r)=\mu_r$, 
Suppose $\{Q_n\}_{n\ge 1}$ is a sequence of symmetric, zero-diagonal $v$-tensors, such that   \begin{align}\label{eq:cut}\|W_{Q_n}- W_\infty\|_{\square^*}\xrightarrow{n\to\infty}0,
\end{align}
 for some $W_{\infty}\in \mathcal{W}_v$.
% and 
%  \begin{align}\label{eq:l1}
%    \limsup_{n\to \infty} \left\|
% W_{Q_n}\right\|_1<\infty.
% \end{align} 
Then the following conclusions hold:
 \begin{itemize}
     \item[(i)] Suppose ${\bf{X}}=(X_1,\ldots,X_n)$ where $X_i \overset{\text{i.i.d.}}{\sim} \mu$. Then $N_n(Q_n,\phi,{{\bf X}})$ satisfies an LDP with speed $n$, and the good rate function
 \begin{equation*}
J_{W_\infty}(t):=\inf_{{\boldsymbol f}\in \mathcal{F}_c:\ { G}_{{W_\infty},\phi}({\boldsymbol f})={t}}  \left\{ \int_0^1\sum_{r=1}^c f_{r}(u)\log \frac{f_{r}(u)}{\mu_r}du\right\},
\end{equation*}
with $G_{W,\phi}(\cdot)$ as in \cref{def:g2}.
\item[(ii)] For any $\beta\in\R$ with $Z_{Q_n}(\beta,\mu)$ as in \eqref{eq:gibbs_pottsloppartition}, we have
\begin{align}\label{eq:potts_unconstrained_opt2}
\lim\limits_{n\rightarrow \infty} Z_{Q_n}(\beta,\mu)=\sup_{{\boldsymbol f}\in \mathcal{F}_c}\left\{ \beta{G}_{{W_\infty},\phi}({\boldsymbol f})
-\int_0^1\sum_{r=1}^c f_r(u)\log \frac{f_r(u)}{\mu_r}\,du\right\}.
		\end{align}
        %where for $W\in \mathcal{W}_v$, $\Xi_{W}:\mathcal F_c\to\mathbb R$ denotes the functional\begin{equation}
%\label{eq:Phi_def2}
%\Xi_{W}({\boldsymbol f})
%:=\beta{G}_{{W},\phi}({\boldsymbol f})
%-\int_0^1\sum_{r=1}^c f_r(u)\log \frac{f_r(u)}%{\mu_r}\,du.
%\end{equation}
%+\sum_{r=1}^c h_r\int_0^1 f_r(u)\,du
 Moreover, the maximizers of the above optimization problem are attained.
    \end{itemize}
\end{thm}
% \cref{NewLDP} shows that the large deviation behavior of the interaction functional is governed by a variational problem balancing entropy and interaction energy. This is analogous to classical results in statistical mechanics, where the limiting free energy is characterized by a Gibbs variational principle (see \cite{jain2018meanfield}).
% \reihaneh{more classic citations???}
% \\

A natural question arising from the variational characterization in \eqref{eq:potts_unconstrained_opt2} is the structure of its maximizers. In particular, we are interested in understanding when the system admits only constant optimizers (\emph{replica symmetry}).
Proceeding to study this,  we introduce the following definition.
\begin{defn}\label{def:tr}
Given a hypergraphon $W\in \mathcal{W}_v $, and functions $\phi:[c]^v\to \R$ and ${\boldsymbol{f}}\in \mathcal{F}_c$, we define
 \begin{align}\label{eq:def_T_general_m}
\notag\mathcal T_r^{(m)}[W,\phi,{\boldsymbol f}](x)
:=&
\int_{[0,1]^{v-1}}
W(x_1,\ldots,x_{m-1},x,x_{m+1}, \ldots, x_v)\\
&\sum_{\substack{s_1,\ldots,s_{m-1},s_{m+1},\ldots,s_v\in[c]}}
\phi(s_1,\ldots,s_{m-1},r,s_{m+1},\ldots,s_v)
\prod_{\substack{a=1\\ a\neq m}}^v f_{s_a}(x_a)\prod_{\substack{a=1\\ a\neq m}}^vdx_a,
\end{align}
for $m\in[v], r\in [c]$.

\end{defn}
% Our next result deals with the natural question when does the set of optimizers of \eqref{eq:potts_unconstrained_opt2} consists of only constant functions. Equivalently, borrowing terminology from statistical physics, we want to understand the
% “replica-symmetry” phase of the optimization problem in \eqref{eq:potts_unconstrained_opt2}.
Heuristically, $\mathcal T_r^{(m)}[W,\phi,{\boldsymbol f}](x)$ represents the contribution of the interaction energy when the $m$-th coordinate is fixed at position $x$ and color $r$, while the remaining coordinates are averaged according to ${\boldsymbol f}$.
The following result characterizes the maximizers of the variational problem in \eqref{eq:potts_unconstrained_opt2} through a fixed point equation.

 %This can be interpreted as a mean-field condition, where the local distribution at each point must be consistent with the global interaction it induces.
\begin{prop}\label{prop:propotts2}
Under the assumptions of \cref{NewLDP}, the following conclusions hold:
\begin{itemize}
\item[(i)](Characterization of Optimizers)
If ${\boldsymbol f}=(f_1,\ldots,f_c)$ is a maximizer of  the optimization problem \eqref{eq:potts_unconstrained_opt2}, it satisfies
\begin{equation}\label{eq:general_fp_onlyWsym}
f_r(x)\stackrel{\lambda-a.e.}{=}
\frac{\mu_r\exp\!\big(\beta\,\mathcal T_r[W_{\infty},\phi,{\boldsymbol f}](x)\big)}
{\sum_{s=1}^c\mu_s\exp\!\big(\beta\,\mathcal T_s[W_{\infty},\phi,{\boldsymbol f}](x)\big)},
\qquad r\in[c],
\end{equation}
where
\begin{equation}\label{eq:def_T_general_sum}
\mathcal T_r[W,\phi,{\boldsymbol f}](x)
:=
\sum_{m=1}^v \mathcal T_r^{(m)}[W,\phi,{\boldsymbol f}](x),\qquad r\in[c],
\end{equation}
and we write $\lambda$-a.e. to mean almost everywhere w.r.t. Lebesgue measure.
\item[(ii)](Replica symmetry breaking) Suppose there exist $r,s\in[c]$ such that \[\Gamma_r({(\mu_1,\ldots,\mu_c)})\neq\Gamma_s((\mu_1,\ldots,\mu_c)),\] where
\begin{align}\label{eq:def_Gamma_general}
\Gamma_r({\boldsymbol y})
:=
\sum_{m=1}^v
\sum_{\substack{s_1,\ldots,s_{m-1},s_{m+1},\ldots,s_v\in[c]}}
\phi(s_1,\ldots,s_{m-1},r,s_{m+1},\ldots,s_v)
\prod_{\substack{a=1\\ a\neq m}}^v y_{s_a}.
\end{align}
Moreover, assume $\beta\neq0$ and $\mathcal{V}[W_{\infty}](\cdot)$ (as in~\eqref{deft1}) is not constant $\lambda$-a.e. Then none of the maximizers of \eqref{eq:potts_unconstrained_opt2} are constant $\lambda$-a.e.

% \item[(iii)](Replica symmetry)
% Suppose that \(W_{\infty}\) is constant $\lambda$-a.e.
% Then every maximizer of \eqref{eq:potts_unconstrained_opt2}
% is constant $\lambda$-a.e.

% Moreover, if \({\boldsymbol f}(x)={\boldsymbol y}\) a.e., then
% \({\boldsymbol y}\in\Delta_{c-1}\) satisfies
% \begin{equation}\label{eq:RS_fp}
% y_r
% =
% \frac{\mu_r\exp\!\big(\beta\kappa\,\Gamma_r({\boldsymbol y})\big)}
% {\sum_{s=1}^c\mu_s\exp\!\big(\beta\kappa\,\Gamma_s({\boldsymbol y})\big)},
% \qquad r\in[c].
% \end{equation}
\end{itemize}
\end{prop}

%This fixed-point equation is analogous to the classical mean-field equations arising in statistical physics models, such as the Curie–Weiss and Potts models, but now incorporates higher-order interactions through the operator $\mathcal T$.

% \reihaneh{CHECK IF I NEED THIS}
% \begin{remark}\label{nonneg}
%     It is obvious from \eqref{eq:general_fp_onlyWsym} that if ${\boldsymbol f}=(f_1,\ldots,f_c)$ is an optimizer of \eqref{eq:potts_unconstrained_opt2}, we have $0<f_r<1$ $\lambda$-a.e. for all $r\in [c]$.
% \end{remark}
\begin{remark}
Even though our setup is for the case when the underlying sequence of tensors $Q_n$ is fixed, there are many applications when $Q_n$ is a random tensor. In this setting, it is natural to study the large deviation behavior of $N_n(Q_n,\phi,{\bf X})$ conditional on $Q_n$, and to ask whether a deterministic rate function emerges in the limit. The following definition makes this precise: %(see \cite[Definition 1.3]{bhattacharya2024ldp}).

\begin{defn}(Conditional LDP)
\label{cor:conditional-ldp}
% Let $Q_n$ be random symmetric zero-diagonal $v$-tensors. Suppose
% \begin{align}
% \|W_{Q_n}-W\|_{\square^*}=o_{\mathbb P}(1)
% \label{eq:rand-cut}
% \end{align}
% for some deterministic $W\in\mathcal W_v$, and
% \begin{align}
% \|W_{Q_n}\|_1=O_{\mathbb P}(1).
% \label{eq:rand-l1}
% \end{align}
Assume $Q_n$ is a sequence of random symmetric zero-diagonal $v$-tensors.
We say that $ N_n(Q_n,\phi,{{\bf X}})$ satisfies a \emph{conditional LDP} with speed $n$ and a deterministic good rate function $J(\cdot):\R \to [0,\infty]$, if the level sets of $J(\cdot)$ are compact, and for any Borel set $F\subseteq \R$ and every $\varepsilon>0$,
\begin{align}\label{conldp}
\mathbb{P}_{Q_n} \Bigg(
 - \inf_{x \in F^\circ} J(x) - \varepsilon 
 \le \frac{1}{n} \log \mathbb{P}\big( N_n(Q_n,\phi,{{\bf X}}) \in F \big|\, Q_n\big)
 \le - \inf_{x \in \overline{F}} J(x) + \varepsilon
\Bigg)\xrightarrow{n\to\infty} 1,
\end{align}
where $\mathbb{P}_{Q_n}$ denotes the probability
over the randomness of $Q_n$. Moreover, $F^\circ$ and $\overline{F}$ denote the interior and the closure of $F$, respectively.
\end{defn}
\end{remark}

% \reihaneh{Should it be a corr? a remark? what? Proof needed or not?} {\color{red} I think Corollary is fine.} 
\begin{prop}\label{condldpdef}Assume $Q_n$ is a sequence of random symmetric zero-diagonal $v$-tensors (independent of $\mathbf{X}$) such that
\begin{align}
\|W_{Q_n}-W_\infty\|_{\square^*}=o_{\mathbb P}(1),
% \quad \|W_{Q_n}\|_1=O_{\mathbb P}(1)
\label{eq:rand-cut}
\end{align}
for some deterministic $W_\infty\in\mathcal W_v$.
% \begin{align}
% \|W_{Q_n}\|_1=O_{\mathbb P}(1).
% \label{eq:rand-l1}
% \end{align} 
%Then using \cref{NewLDP}, Then
Then $N_n(Q_n,\phi,{{\bf X}})$ satisfies a conditional LDP (see~\cref{cor:conditional-ldp}) with speed $n$ and the deterministic good rate function $J_{W_\infty}(\cdot)$ as in \cref{NewLDP} (i).
\end{prop}

We now illustrate our general results, showing how the abstract framework developed above applies to concrete examples.
\subsection{Erd\H{o}s-R\'enyi Hypergraph}\label{ER Hypergraph}

Suppose \(\mathcal G_n^{(v)}\) is a \(v\)-uniform random hypergraph on \([n]\), viewed
as a symmetric zero-diagonal \(v\)-tensor, defined as
\begin{align}\label{DEF:hyptensor}
\mathcal{G}^{(v)}_n(i_1,\ldots,i_v)=\begin{cases}1/p_n, & \mbox{with probability }\,\, p_n\\ 0, & \mbox{with probability }\ 1-p_n\,\end{cases},
\end{align}
independently over unordered \(v\)-sets, and extend this value symmetrically to all
permutations of \((i_1,\ldots,i_v)\). If \(i_1,\ldots,i_v\) are not all distinct, set
\(
\mathcal G_n^{(v)}(i_1,\ldots,i_v)=0.
\) Then $\{\mathcal{G}^{(v)}_n\}_{n\ge1}$ is a sequence of symmetric zero-diagonal random tensors. 

\begin{prop}\label{LDPforERHyp}
Consider the sequence of random tensors $\{\mathcal{G}^{(v)}_n\}_{n\ge1}$ defined in \eqref{DEF:hyptensor} and assume $ n^{1-v}\ll p_n $.
Then the following conclusions hold:
 \begin{itemize}
     \item[(i)] Suppose ${\bf{X}}=(X_1,\ldots,X_n)$ where $X_i \overset{\text{i.i.d.}}{\sim} \mu$. Then $N_n(\mathcal{G}^{(v)}_n,\phi,{{\bf X}})$ satisfies a (conditional) LDP with speed $n$, and the good rate function
 \begin{align*}
J_{\mathbbm{1}}(t)=\inf_{{\boldsymbol f}\in \mathcal{F}_c:\ { G}_{\mathbbm{1},\phi}({\boldsymbol f})={t}}  \left\{ \int_0^1\sum_{r=1}^c f_{r}(u)\log \frac{f_{r}(u)}{\mu_r}du\right\},
\end{align*}
where $\mathbbm{1}$ denotes the constant graphon that equals 1 at all points in $[0,1]^v$.
\item[(ii)] For any $\beta\in\R$ we have
\begin{align*}
 Z_{\mathcal{G}^{(v)}_n}(\beta,\mu)\xrightarrow[n\to\infty]{\mathbb P}\sup_{{\boldsymbol f}\in \mathcal{F}_c}\left\{ \beta\,{G}_{\mathbbm 1,\phi}({\boldsymbol f})
-\int_0^1\sum_{r=1}^c f_r(u)\log \frac{f_r(u)}{\mu_r}\,du\right\}.
		\end{align*}
      
%+\sum_{r=1}^c h_r\int_0^1 f_r(u)\,du
% with \[
% G_{1,\phi}({\boldsymbol f})
% =\int_{[0,1]^v}
% \Phi_{\boldsymbol f}(x_1,\ldots,x_v)\,dx_1\cdots dx_v.
% \]
Moreover, the maximizers of the above optimization problem are attained.
    \end{itemize}
   
\end{prop}

\subsection{Sparse Erd\H{o}s-R\'enyi Random Graphs}\label{SIERgraph}
Let $H:= (V(H), E(H))$ be a finite connected graph with $ |V(H)|=v\geq 2$ vertices labeled $[v]=\{1,2, \ldots, v\}$, $e:=|E(H)|$ edges, and maximum degree $\Delta$.
% Suppose $\psi:[0,1]^2\to\mathbb{R}_{>0}$ is a  symmetric and continuous function, and $\mathcal{G}_n$ is a random graph with $n$ labeled vertices, where for each pair $i \neq j$, $\mathcal{G}_n(i,j)\sim \mathrm{Ber}(p_n\,\psi(\frac{i}{n},\frac{j}{n}))$ independently, and $\mathcal{G}_n(i,i)=0$ for $i\in[n]$. 
Suppose $\mathcal{G}_n$ is a random graph with $n$ labeled vertices, where for each pair $i< j$, $\mathcal{G}_n(i,j)\sim \mathrm{Ber}(p_n)$ independently, and for $i>j$, set $\mathcal{G}_n(i,j)=\mathcal{G}_n(j,i)$, and $\mathcal{G}_n(i,i)=0$ for $i\in[n]$.
%Here, $\mathrm{Ber}(q)$ denotes the Bernoulli random variable with mean $q$.
% \begin{remark}
%     One could easily check that our proofs are valid if we replace the continuity assumption on $\psi$ with a.e continuity and boundedness.
% \end{remark}
\begin{defn}\label{symg}
Setting 
    % \[\mathrm{Sym}[\mathcal{G}_n](i_1,\ldots,i_v):=\frac{1}{v!p_n^e}\sum_{\sigma \in S_v} \prod_{\{a,b\}\in E(H)} \mathcal{G}_n\Big(i_{\sigma(a)}, i_{\sigma(b)}\Big), \]
    % \[{\Psi_n}(i_1,\ldots,i_v):=\frac{1}{v!}\sum_{\sigma \in S_v} \prod_{\{a,b\}\in E(H)} \psi\Big(i_{\sigma(a)}/n, i_{\sigma(b)}/n\Big), \]
    %$Q_n(i,j):=\mathcal{G}_n(i,j)/p$. 
    \[\mathrm{Sym}[\mathcal{G}_n,H](i_1,\ldots,i_v):=\mathrm{Sym}[\mathcal{G}_n](i_1,\ldots,i_v)=\frac{1}{v!p_n^e}\sum_{\sigma \in S_v} \prod_{\{a,b\}\in E(H)} \mathcal{G}_n\Big(i_{\sigma(a)}, i_{\sigma(b)}\Big)\mathbf 1_{\{\text{all}\,i_1,\ldots,i_v\,\text{are distinct}\}}, \]
    note that $\mathrm{Sym}[\mathcal{G}_n]$ is a zero-diagonal symmetric tensor.
\end{defn}
% \begin{defn}\label{wpsi}
% For any symmetric continuous function $\psi:[0,1]^2\to \R$, define
%     % \[\mathcal{U}_\psi(x_1,\ldots, x_v):=\frac{1}{v!}\sum_{\sigma \in S_v} \prod_{\{a,b\}\in E(H)} \psi\Big(x_{\sigma(a)}, x_{\sigma(b)}\Big),\]
%     \[\mathcal{U}_{\psi,H}(x_1,\ldots, x_v):=\mathcal{U}_{\psi}(x_1,\ldots, x_v)=\frac{1}{v!}\sum_{\sigma \in S_v} \prod_{\{a,b\}\in E(H)} \psi\Big(x_{\sigma(a)}, x_{\sigma(b)}\Big).\]
%     and note that $\mathcal{U}_\psi\in \mathcal{W}_v.$
% \end{defn}

\begin{defn}\label{corH}
For the graph $H$, we define $\widetilde{H}:=(V(\widetilde{H}), E(\widetilde{H}))$ to be the 2-core of $H$, i.e., the maximal subgraph of $H$ in which all vertices have degree at least $2$.\end{defn}

$\widetilde{H}$ can be obtained from $H$ as follows:
Remove any leaf present in $H$, and the edge incident on it, to obtain a new graph. If it has no leaves we are done. If not repeat the above process till the resulting graph has no leaves left. Notice that the final graph $\widetilde{H}$ equals the empty graph if and only if $H$ is a tree. In ~\cite[Theorem 1.2]{bhattacharya2024ldp} the authors prove a result which applies to the case where $H$ is a tree. In this paper we will focus on the case when $H$ is not a tree.

\begin{defn}\label{def:notation}
We denote $|V(\widetilde{H})|, |E(\widetilde{H})|$, and the maximum degree of $\widetilde{H}$ by $\widetilde{v}, \widetilde{e}, \widetilde{\Delta}$, respectively. Define $V':= V(H)\setminus V(\widetilde{H})$ and $E':=E(H)\setminus E(\widetilde{H})$. Notice that $|V'|=|E'|=v-\widetilde{v}=e-\widetilde{e}$, where this quantity is denoted by $\omega$. Also, we have $\widetilde{\Delta}\ge2$, by definition of $\widetilde H$. 
\end{defn}
\begin{prop}\label{LDPforINSER}Suppose $H$ is not a tree and \[(n\log n)^{-1/ {\widetilde \Delta}}\ll p_n \ll1 .\]

Then the following conclusions hold:

 \begin{itemize}
     \item[(i)]Suppose ${\bf{X}}=(X_1,\ldots,X_n)$ where $X_i \overset{\text{i.i.d.}}{\sim} \mu$. Then $N_n(\mathrm{Sym}[\mathcal{G}_n],\phi,{{\bf X}})$ satisfies a (conditional) LDP with speed $n$, and the good rate function
     
 \begin{equation*}
J_{\mathbbm{1}}(t)=\inf_{{\boldsymbol f}\in \mathcal{F}_c:\ { G}_{\mathbbm{1},\phi}({\boldsymbol f})={t}}  \left\{ \int_0^1\sum_{r=1}^c f_{r}(u)\log \frac{f_{r}(u)}{\mu_r}du\right\},
\end{equation*}where $\mathbbm{1}$ denotes the constant graphon that equals 1 at all points in $[0,1]^v$.
\item[(ii)] For any $\beta\in\R$, we have
\begin{align*}
Z_{\mathrm{Sym}[\mathcal{G}_n]}(\beta,\mu)\xrightarrow[n\to\infty]{\mathbb P}\sup_{{\boldsymbol f}\in \mathcal{F}_c}\left\{ \beta\,{G}_{\mathbbm{1},\phi}({\boldsymbol f})
-\int_0^1\sum_{r=1}^c f_r(u)\log \frac{f_r(u)}{\mu_r}\,du\right\}.
		\end{align*}
        Moreover, the maximizers of the above optimization problem are attained.
    \end{itemize}

\end{prop}
% \begin{prop}\label{LDPforINSER}Suppose $H$ is not a tree and \[(n\log n)^{-1/ {\widetilde \Delta}}\ll p_n \ll1 .\]

% Then the following conclusions hold:
%  \begin{itemize}
%      \item[(i)]Suppose ${\bf{X}}=(X_1,\ldots,X_n)$ where $X_i \overset{\text{i.i.d.}}{\sim} \mu$. Then $N_n(\mathrm{Sym}[\mathcal{G}_n],\phi,{{\bf X}})$ satisfies a (conditional) LDP with speed $n$, and the good rate function
     
%  \begin{equation*}
% J_{\mathcal{U}_\psi}(t)=\inf_{{\boldsymbol f}\in \mathcal{F}_c:\ { G}_{\mathcal{U}_\psi,\phi}({\boldsymbol f})={t}}  \left\{ \int_0^1\sum_{r=1}^c f_{r}(u)\log \frac{f_{r}(u)}{\mu_r}du\right\}.
% \end{equation*}
% \item[(ii)] For any $\beta\in\R$, we have
% \begin{align*}
% Z_{\mathrm{Sym}[\mathcal{G}_n]}(\beta,\mu)\xrightarrow[n\to\infty]{\mathbb P}\sup_{{\boldsymbol f}\in \mathcal{F}_c}\left\{ \beta\,{G}_{{\mathcal{U}_\psi},\phi}({\boldsymbol f})
% -\int_0^1\sum_{r=1}^c f_r(u)\log \frac{f_r(u)}{\mu_r}\,du\right\}.
% 		\end{align*}
%         Moreover, the maximizers of the above optimization problem are attained.
%     \end{itemize}

% \end{prop}

Finally, we connect our results to statistical physics by studying a generalized Potts statistic with higher-order interactions. In this setting, the LDP translates into precise information about the limiting free energy and macroscopic behavior of the system.

\subsection{A Generalized Potts Statistic}\label{GPM}
With $ N_n(Q_n,\phi,{{\bf X}})$ as in \eqref{hamiltoniandef}, let $N^{(r)}_n(Q_n,{\bf X})$ denote \( N_n(Q_n,\phi,{{\bf X}})\) where $\phi(x_1,\ldots,x_v)=\mathbf{1}_r$. In other words,
\begin{align}\label{defNinGPT}
& N^{(r)}_n(Q_n,{\bf X})=\frac{1}{n^v}\sum_{i_1,\cdots,i_v=1}^n\mathbf{1}_{\{ X_{i_1}=\ldots=X_{i_v}=r\}} Q_n(i_1,\ldots,i_v).
\end{align}
 % $N_n(Q_n,\mathbf{1}_r,{{\bf X}})$
 Define the $c$-length vector 
\begin{align}\label{defNinGPT2}
\boldsymbol N_n(Q_n,\mathbf{X})= (N^{(1)}_n(Q_n,\mathbf{X}), \ldots, N^{(c)}_n(Q_n,\mathbf{X}))^\top.    
\end{align}
For any $\boldsymbol\alpha,\boldsymbol h \in \R^c$, define the Gibbs measure
\begin{align}\label{eq:gibbs_potts2}
	\frac{d\mathcal R_{Q_n,{ \boldsymbol{\alpha}},\boldsymbol h}}{d(\mathscr{U}[c])^{\otimes n}}(\mathbf{X}):=\exp\Big(n\boldsymbol\alpha^\top \boldsymbol N_n(Q_n,\mathbf{X})+\sum_{i=1}^{n}\sum_{a=1}^{c} h_a\mathbf{1}_{\{X_i=a\}}-n\mathcal{Z}_{Q_n}({\boldsymbol{\alpha}},\boldsymbol{h})\Big),
\end{align}
and \begin{align}\label{logpf}
\mathcal Z_{Q_n}({\boldsymbol{\alpha}},\boldsymbol h):=\frac{1}{n}\log \E_{(\mathscr{U}[c])^{\otimes n}} \exp\Big(n \boldsymbol\alpha^\top \boldsymbol N_n(Q_n,\mathbf{X})+\sum_{i=1}^{n}\sum_{a=1}^{c}h_a\mathbf{1}_{\{X_i=a\}}\Big),    
\end{align}
where $\mathscr{U}[c]$ denotes the uniform distribution on $[c]$.
%$\boldsymbol N_n(Q_n,\mathbf{X})$ is supposed to reflect all the interactions between the entries of $\bf X$, so one could write $$\mathcal R_{Q_n,{ \boldsymbol{\alpha}},\boldsymbol h}(\mathbf{X}) \propto \exp\Big(n\boldsymbol\alpha^\top \boldsymbol N_n(Q_n,\mathbf{X})\Big)\mu^{\otimes n} ({\bf X}),$$ 
Let $\mu$ be a probability distribution on $[c]$, defined by \begin{align}\label{mu}
 &\mu_r =\frac{e^{h_r}}{\sum_{s=1}^c e^{h_s}}\qquad r\in[c],
\end{align}

and note that
\begin{align*}%\label{eq:gibbs_potts21}
&\frac{d\mathcal R_{Q_n,\boldsymbol\alpha,\boldsymbol h}}
{d\mu^{\otimes n}}(\mathbf X)= \exp\Big(n\boldsymbol\alpha^\top \boldsymbol N_n(Q_n,\mathbf{X})-n\mathcal{Z}_{Q_n}({\boldsymbol{\alpha}}, \boldsymbol{h})\Big)\left(\sum_{s=1}^c e^{h_s}/c\right)^n.
\end{align*}

%so without loss of generality we can assume our Gibbs measure has the form of \eqref{eq:gibbs_potts21}, and our base measure is $\mathscr{U}[c]$, or the Hamiltonian is $n\boldsymbol\alpha^\top \boldsymbol N_n(Q_n,\mathbf{X})$ and the base measure is $\mu$ where $\mu_r\propto e^{h_r}$. 

For ${\boldsymbol{x}}=(x_1,\dots,x_n)\in[c]^n$, define the empirical color proportion vector
\[
{\boldsymbol m}_n({\boldsymbol{x}})
:=\Big(\frac1n\sum_{i=1}^n\mathbf 1_{\{x_i=1\}},\ldots,
\frac1n\sum_{i=1}^n\mathbf 1_{\{x_i=c\}}\Big).
\]

We will refer to the Gibbs measure in~\eqref{eq:gibbs_potts2}  as the generalized Potts model. In particular, if $v=2$ and $Q_n$ is a matrix (quadratic interaction), we get back the usual Potts model, which is what has been extensively studied in the literature (see~\cite{Basak2017,eichelsbacher2015rates,dembo2014replica,costeniuc2005complete,gandolfo2010limit}~
%\cite{Basak2017,wu1982potts,potts1952some,eichelsbacher2015rates,dembo2014replica,costeniuc2005complete,blanchard2008thermodynamic,gandolfo2010limit}
and the references therein). %This model was introduced by \cite{bhattacharya2024ldp} and corresponds to the usual Potts model with \emph{higher order interactions} in the Hamiltonian.
% \reihaneh{useful citATION???\cite{dembo2010ising,basak2017ferromagnetic,dembo2010gibbs}.}
In our next result, we derive a law of large numbers for the empirical color proportion vector $\boldsymbol m_n({\bf X})$ in terms of the maximizers of the variational problem. For stating the result, for any $\boldsymbol{x}\in\R^c$ and $\mathcal S\subseteq \R^c$, we set
\(
\operatorname{dist}(\boldsymbol{x},\mathcal S)
:=\inf_{\boldsymbol{y}\in\mathcal S}\|\boldsymbol{x}-\boldsymbol{y}\|,
\)
where $\|\cdot\|$ denotes a fixed norm on $\mathbb R^c$ (all such norms being equivalent). 
\begin{prop}\label{prop:ldp+zn}
Suppose $\{Q_n\}_{n\ge 1}$ is a sequence of symmetric, zero-diagonal $v$-tensors, such that \eqref{eq:cut} holds. Then the following conclusions hold:
    \begin{itemize}
        \item[(i)] Suppose ${\bf{X}}=(X_1,\ldots,X_n)$ where $X_i \overset{\text{i.i.d.}}{\sim} \mu$ defined in \eqref{mu}. Then for any ${\boldsymbol\alpha}\in \R^c$, $\boldsymbol{\alpha}^\top\boldsymbol N_n(Q_n, \bf{X})$ (see \eqref{defNinGPT2}) satisfies an LDP  with speed $n$ and the good rate function 
        
\begin{equation}\label{eq:i2}
I_{W_\infty,\boldsymbol{\alpha}}({t}):=\underset{\substack{{\boldsymbol f}\in \mathcal{F}_c: \\ \ \boldsymbol{\alpha}^\top {\boldsymbol G}_{W_{\infty}}({\boldsymbol f})=t}}{\inf} \Big\{ \int_0^1\sum_{r=1}^c f_{r}(u)\log f_{r}(u)du-\sum_{r=1}^{c}h_r \int_0^1 f_r(u)du\Big\}+\log \sum_{s=1}^c e^{h_s},
\end{equation}
where ${\boldsymbol G}_{W_{\infty}}({\boldsymbol f}):= \Big(G_{W_{\infty},\mathbf{1}_1}(\boldsymbol f), \ldots, G_{{W_{\infty}},\mathbf{1}_c}(\boldsymbol f)\Big)$, see \cref{def:g2}. 

\item[(ii)] For the scaled log partition function $\mathcal{Z}_{Q_n}({\boldsymbol\alpha},\boldsymbol h)$ as in \eqref{logpf}, we have
\begin{align}\label{eq:potts_unconstrained_opt}
\lim\limits_{n\rightarrow \infty} \mathcal{Z}_{Q_n}({\boldsymbol\alpha},\boldsymbol h)= \sup_{{\boldsymbol f}\in \mathcal{F}_c}\left\{{\boldsymbol\alpha}^\top {\boldsymbol G}_{W_\infty}({\boldsymbol f})
+\sum_{r=1}^c h_r\int_0^1 f_r(u)\,du
-\int_0^1\sum_{r=1}^c f_r(u)\log f_r(u)\,du\right\}-\log c.
		\end{align}   
        Moreover, the maximizers of the above optimization problem are attained.
        %where  $\mathfrak{F}:\mathcal F_c\to\mathbb R$ denotes the variational functional\begin{equation}
%\label{eq:potts_unconstrained_opt}
% \mathfrak{F}({\boldsymbol f})
% :={\boldsymbol\alpha}^\top {\boldsymbol G}_{W}({\boldsymbol f})
% +\sum_{r=1}^c h_r\int_0^1 f_r(u)\,du
% -\int_0^1\sum_{r=1}^c f_r(u)\log f_r(u)\,du.
% \end{equation}

\item [(iii)] Let $\mathcal{A}\subseteq \mathcal{F}_c$ denote the set of optimizers in \eqref{eq:potts_unconstrained_opt}. Then, for every $\varepsilon>0$ we have 
\begin{equation}
\label{eq:LLN_set}
\mathcal R_{Q_n,{ {\boldsymbol\alpha}},\boldsymbol h}\left(
\operatorname{dist}({\boldsymbol m}_n({\bf X}),\mathcal S)>\varepsilon
\right)
\xrightarrow{n\to\infty}0,
\end{equation}
%and in fact the convergence holds with an exponential rate, 
% Let $\mathcal M:=\arg\max_{{\boldsymbol f}\in\mathcal F_c}\Phi({\boldsymbol f})$
% be the set of maximizers of~\eqref{eq:potts_unconstrained_opt}.
where
\[
\mathcal S
:=\Big\{\Big(\int_0^1 f_1(u)\,du,\ldots,\int_0^1 f_c(u)\,du\Big):\, {\boldsymbol f}=(f_1,\ldots,f_c)\in\mathcal{A}\Big\}.
\]
In particular, if $\mathcal{A}=\{{\boldsymbol f}^*\}$, then
\[
{\boldsymbol m}_n({\bf X})\xrightarrow[n\to\infty]{\mathcal R_{Q_n,{ {\boldsymbol\alpha}},\boldsymbol h}\text{-prob.}}
\Big(\int_0^1 f_1^*(u)\,du,\ldots,\int_0^1 f_c^*(u)\,du\Big).
\]
    \end{itemize}
\end{prop}
 %Thus, the LDP not only describes rare events, but also determines the typical macroscopic behavior of the system. 
 When the optimizer in the RHS of \eqref{eq:potts_unconstrained_opt} is unique (i.e.~$\mathcal{A}$ is a singleton), this leads to a true law of large numbers for the empirical color proportions, using part (iii).

\subsubsection{Monochromatic subgraphs of random graphs}
Let $H$ be a finite connected graph with $v\geq 2$ labeled vertices, $e$ edges, and maximum degree $\Delta$. Let $G_n$ be a sequence of simple graphs with $n$ vertices labeled by $[n]$, each containing at least one edge.
%${\bf{X}}=(X_1,\ldots,X_n)$ colored by $c$ colors and $Q_n$ be the scaled adjacency matrix of a graph $G_n$, i.e. $ Q_n := \frac{1}{\|G_n\|_1}G_n$, where $\|G_n\|_1= \frac{1}{n^2}\sum_{i,j}G_n(i, j)$ and $G_n(i, j):=\mathbf{1}_{i \sim j}$. 
Set
\begin{align}\label{eq:333}
 Q_n(i_1,\ldots,i_v)=\frac{1}{v! \|G_n\|_1^{e}}\sum_{\sigma \in S_v} \prod_{\{a,b\}\in E(H)} G_n\Big(i_{\sigma(a)}, i_{\sigma(b)}\Big)\mathbf 1_{\{\text{all}\,i_1,\ldots,i_v\,\text{are distinct}\}},   
\end{align}
 where $\|G_n\|_1:= n^{-2}\sum_{i,j}G_n(i, j)$ and $G_n(i, j):=\mathbf{1}_{\{i \sim j\}}$.
With this choice, \eqref{defNinGPT} reduces to
\begin{align}\label{sbgrphcont}
    N^{(r)}_n(Q_n,{\bf X})= \frac{1}{n^v \|G_n\|_1^{e}}\underset{\substack{ i_1,\ldots,i_v\in [n]\\  {\rm  distinct}}}{\sum} \mathbf{1}_{\{ X_{i_1}=\ldots=X_{i_v}=r\}} \prod_{\{a,b\}\in E(H)}G_n(i_a,i_b)=:N_r(H,G_n,{\bf{X}}).
\end{align}
One can think of  $N_r(H,G_n,\bf{X})$ as the (scaled) number of copies of the subgraph $H$ in $G_n$ in which all vertices have color $r$. %Define the $c$-length vector 
%$$\boldsymbol N(H,G_n,\mathbf{X}):= (N_1(H,G_n,\mathbf{X}), \ldots, N_c(H,G_n,\mathbf{X}))^\top.$$
% Similarly, given the function $\boldsymbol{N}(H,G_n,\bf{X})$ and $\boldsymbol\alpha,\boldsymbol h \in \R^c$, define the Gibbs measure
% \begin{align}\label{eq:gibbs_potts2}
% 	\R_{Q_n,{ {\boldsymbol{\alpha}}},\boldsymbol h}(\mathbf{X})=\exp\Big(n\boldsymbol\alpha^\top \boldsymbol N(H,Q_n,\mathbf{X})+\sum_{i=1}^{n}\sum_{a=1}^{c} h_a\mathbf{1}_{X_i=a}-nZ_{Q_n}({\boldsymbol{\alpha}},\boldsymbol{h})\Big),
% \end{align}
% and $$
% Z_{Q_n}({\boldsymbol{\alpha}},\boldsymbol h):=\frac{1}{n}\log \E_{(\mathscr{U}[c])^{\otimes n}} \exp\Big(n \boldsymbol\alpha^\top \boldsymbol N(H,Q_n,\mathbf{X})+\sum_{i=1}^{n}\sum_{a=1}^{c}h_a\mathbf{1}_{X_i=a}\Big).$$

\cref{prop:ldp+zn}  implies an LDP for  ${\boldsymbol \alpha}^\top N(H,G_n,\mathbf{X}):=\sum_{r=1}^c \alpha_r N_r(H,G_n,\mathbf{X})$, which recovers the LDP studied in ~\cite[Corollary 1.7, Theorem 1.8]{bhattacharya2024ldp}. This demonstrates that our tensor setup incorporates as a special case tensors obtained from a graph as in \eqref{eq:333}. This is formalized in the following corollary.

\begin{corr}\label{thisgivesthat}
Given a sequence of graphs $G_n$ and a fixed graph $H$ above, set $Q_n$ by \eqref{eq:333} and set $\widetilde{G}_n:=G_n/\lVert G_n\rVert_1$. Let $W^{(2)}_{\widetilde{G}_n}\in L^1([0,1]^2)$ be defined as in \eqref{eq:funext}. Suppose the following assumptions are satisfied: 
\begin{align}\label{eq:twomatbd}
    \limsup_{n\to\infty}\,\lVert W^{(2)}_{\widetilde{G}_n}\rVert_{q\Delta} < \infty \qquad \mbox{and} \qquad \lVert W^{(2)}_{\widetilde{G}_n}-W_\infty^{(2)}\rVert_{\square^*}\overset{n\to\infty}{\to} 0, 
\end{align}
for some $q>1$ and $W_\infty^{(2)}\in L^1([0,1]^2)$. Then for any ${\boldsymbol \alpha}\in \R^c$, the  statistic
${\boldsymbol\alpha}^\top {\boldsymbol N}(H,G_n,{\bf{X}})$ satisfies an LDP with speed $n$ and the good rate function from \eqref{eq:i2}, with $G_{W_{\infty,\mathbf{1}_r}}$ replaced by $G^{(2)}_{{\infty},r}$, $1\le r\le c$ defined as
\begin{align}\label{eq:limfunc}
    G^{(2)}_{{\infty},r}(\boldsymbol{f}):=\int_{[0,1]^v}\prod_{\{i,j\}\in E(H)} W^{(2)}_{\infty}(x_{i},x_{j})\prod_{a=1}^v f_r(x_a)\, d x_a.
\end{align}
\end{corr}

\begin{defn}
    With $\mathcal T_r^{(m)}[W,\phi,{\boldsymbol f}](x)$ as in \eqref{eq:def_T_general_m}, for all ${\boldsymbol{f}}\in \mathcal{F}_c$,  $m\in[v]$ and $r\in [c]$, taking $\phi \equiv1$ we obtain
\begin{align}\label{deft1}
\mathcal T_r^{(m)}[W,1,{\boldsymbol f}](x)
    =\int_{[0,1]^{v-1}} W\left(x, x_2, \ldots, x_v\right) \prod_{a=2}^v d x_a=:\mathcal{V}[W](x),
\end{align}
where we use the fact that ${\boldsymbol{f}}\in \mathcal{F}_c$.
Moreover, defining $\mathbf{1}_r:[c]^v\to\R$ as
\begin{align}\label{def1k}
 \mathbf{1}_r(x_1,\ldots,x_v):=\mathbf{1}_{\{x_1=\ldots, x_v=r\}},  
\end{align}
for all $m\in [v]$ we have,
\begin{align}\label{deftg}
   \mathcal{T}_r^{(m)}[W,\mathbf{1}_r,{\boldsymbol f}](x)=\int_{[0,1]^{v-1}} W\left(x, x_2, \ldots, x_v\right) \prod_{a=2}^v f_r(x_a)\prod_{a=2}^vd x_a=:\mathcal{V}_r[W,{\boldsymbol{f} }](x). 
\end{align}
\end{defn}

\begin{prop}\label{prop:propotts}
    %Suppose \eqref{eq:cut} and \eqref{eq:l1} hold.
    For any $W_\infty\in \mathcal{W}_v$,  $\boldsymbol{\alpha}=(\alpha_1,\ldots,\alpha_c)\in \R^c$, and $\boldsymbol{h}=(h_1,\ldots,h_c)\in \R^c$, consider the optimization problem in the RHS of \eqref{eq:potts_unconstrained_opt}. Then the following conclusions hold:
\begin{itemize}
    \item [(i)](Characterization of optimizers)
If ${\boldsymbol f}=(f_1,\ldots,f_c)$ is a maximizer, it satisfies
\begin{equation}\label{eq:propottshow} 
f_r(x)\stackrel{\lambda-a.e.}{=}\frac{\exp(\alpha_r v \mathcal{V}_{r}[W_\infty, \boldsymbol{f}](x)+h_r)}{\sum_{s=1}^c \exp(\alpha_s v \mathcal{V}_{s}[W_\infty, \boldsymbol{f}](x)+h_s)},\ r\in [c].
\end{equation}

\item [(ii)](Replica-symmetry breaking) Suppose there exist $r,s\in[c]$ such that $\alpha_r e^{(v-1)h_r} \neq \alpha_s e^{(v-1)h_s}$. Moreover, assume $\mathcal{V}[W_\infty](\cdot)$ (as in~\eqref{deft1}) is not constant $\lambda$-a.e. Then none of the maximizers are constant $\lambda$-a.e.

\item [(iii)](Replica symmetry) If $\mathcal{V}[W_\infty](\cdot)=1$   $\lambda$-a.e., $W_\infty>0$ $\lambda$-a.e., and $\alpha_r \ge 0$ for $r\in[c]$, then all of the maximizers are constant $\lambda$-a.e.
\end{itemize}
\end{prop}
\begin{remark}
   Part (i) of~\cref{prop:propotts} characterizes the optimizers of \eqref{eq:potts_unconstrained_opt} in terms of a fixed point equation. Parts (ii) and (iii) give sufficient conditions for replica-symmetry breaking (optimizers are all nonconstant functions), and replica symmetry (optimizers are all constants), respectively. Note that
under the conditions of~\cref{prop:propotts} part (iii), all optimizers are constant functions, which are independent of $W_\infty$. Thus this proves universality of the set of optimizers for this class of functions $W_\infty\in\mathcal{W}_v$.
%In this sense, the geometry of the limiting tensor $W_\infty$ directly determines the phase structure of the system. 
We also stress that neither assumption can, in general, be omitted for replica symmetry. We refer the reader to \cite[Examples 1.2 and 1.3]{bhattacharya2023gibbs} for the relevant counterexamples in the case $v=c=2$, where the authors show that replica symmetry can be violated if either of the aforementioned conditions fails to hold.
\end{remark}

\subsection{Optimization}\label{OPT}
We now turn to a more detailed analysis of the associated optimization problems and the structure of the corresponding rate functions. In particular, we aim to understand when the variational problem admits explicit solutions and how the form of the optimizer influences the rate function. In the entire subsection,
% for each rate function $I$ below, the effective domain is specified
% in the statement; outside the attainable range we have $I(y)=\infty$. Moreover,
we use the convention $0\log 0=0$ and $\mu$ is defined as \eqref{mu}.
% \begin{prop}\label{smllalpha-uniqueness}
%  Assume that~\eqref{eq:cut}~and~\eqref{eq:normcontrol} hold, $W>0$ a.e., and $\mathcal{V}[W](x)=1$ a.e. Moreover, assume $0\leq \alpha_i\leq \frac{1}{v(v-1)}$ for all $i\in [c]$.  Then all the global optimizer of optimization problem \eqref{eq:potts_unconstrained_opt} are constant functions and the optimizer is unique.  
% \end{prop}
% \reihaneh{I need to mention that i will skip the proof for the following.}
% \begin{prop}[LDP for Linear Projections]
% Suppose \eqref{eq:cut} and \eqref{eq:normcontrol} hold, and ${\bf{X}}=(X_1,\ldots,X_n)$ where $X_i \overset{\text{i.i.d.}}{\sim} \mu$.
% Then for each vector ${\boldsymbol{\alpha}} \in \R^c$, $\boldsymbol{\alpha}^\top {\boldsymbol N}(H,Q_n,{\bf{X}})$ satisfies a LDP with the good rate function 

% \begin{equation}\label{eq:Irho}
% I_{\boldsymbol{\alpha}}({t}):=\inf_{{\boldsymbol f}\in \mathcal{F}_c:\ \boldsymbol{\alpha}^\top {\boldsymbol G}_{W}({\boldsymbol f})=t}  \Big\{ \int_0^1\sum_{r=1}^c f_{r}(u)\log f_{r}(u)du-\sum_{r=1}^{c}h_r \int_0^1 f_r(u)du\Big\}.
% \end{equation}
% \end{prop}  

\subsubsection{Results for general $c,v$}
\begin{prop}[Phase transition and critical threshold]\label{prop:potts_gibbs_ldp}
	Consider the optimization problem \eqref{eq:potts_unconstrained_opt}, under the assumptions
    % for some strictly positive function
    %  $W_\infty\in \mathcal{W}_v$ satisfying $\mathcal{V}[W_\infty]\equiv 1$, and  $\alpha_r=\theta\ge 0$ and $h_r=0$ for all $r\in [c]$.  
    \begin{align*}\label{eq:w0}
		W_\infty>0 \,\,\,\lambda\text{-a.e.},\quad \mathcal{V}[W_\infty](x)=1\,\,\lambda\text{-a.e.},\quad
    \text{and} \newline \quad\alpha_r = \theta\geq0, h_r=0 \quad\text{for all  } r\in [c]. 
		\end{align*}
    Then the following conclusions hold:
	\begin{itemize}
		
		\item[(i)]
		Any optimizer $(x_1,\cdots,x_c)\in \Delta_{c-1}$
		 with $x_1\ge \max_{2\le r\le c}x_r$ has the following form:
		\begin{equation}\label{eq:potts_one_large}
			x_1=\frac 1c (1+(c-1)y^*), \quad x_r= \frac 1c (1-y^*), \quad 2 \le r \le c,
		\end{equation}
		for some $0 \le y^* < 1$. 
		\item[(ii)] There exists $\theta_{\rm crit} \in (0,\infty)$ such that, if $\theta<\theta_{\rm crit}$, $(1/c, \ldots, 1/c)$ is the unique optimizer of \eqref{eq:potts_unconstrained_opt} and if $\theta>\theta_{\rm crit}$, there exist exactly $c$ distinct optimizers, none of which are uniform. With $(x_{\theta,1},\ldots ,x_{\theta,c})$ denoting the unique optimizer satisfying $x_{\theta,1}>\max_{2\le r\le c} x_{\theta,r}$, the map $\theta\mapsto (x_{\theta,1},\ldots ,x_{\theta,c})$ is continuous on $(\theta_{\rm crit},\infty)$.

	\end{itemize}
\end{prop}
% \reihaneh{I think we need to tell stories about the phenomenon in $\theta=\theta_c$.}

%The next result provides conditions under which the variational problem admits a unique maximizer. In particular, it shows that when one coordinate is sufficiently favored—either through the interaction parameters or the external field—the system concentrates on a single dominant state.
\begin{prop}\label{prop:potts_diffalpha}
  Consider the optimization problem \eqref{eq:potts_unconstrained_opt}, under the assumption 
$$W_\infty>0\quad \lambda\text{-a.e.},\quad \text{  }\mathcal{V}[W_\infty](x)=1\quad\lambda\text{-a.e.}.$$ %Moreover, suppose \eqref{eq:cut} and \eqref{eq:l1} are satisfied.  
Then the following conclusions hold:
\begin{itemize}
  %   \item [(i)] Assume $h_i=0$ for all $i \in [c]$.
		% Then there exists $\eta=\eta(c,v)>0$ such that for any $\boldsymbol{\alpha}=(\alpha_1,\alpha_2,\ldots,\alpha_c)$ with $\alpha_1,\ldots,\alpha_c\geq 0$ and $\alpha_1 \ge \eta+ \max_{2\le r\le c} \alpha_r$, the optimization problem \eqref{eq:potts_unconstrained_opt} has a unique maximizer. Furthermore, denoting this maximizer by $(x_{\boldsymbol{\alpha},1},\ldots , x_{\boldsymbol{\alpha},c})$, we have $x_{\boldsymbol{\alpha},1}>\max_{2\le r\le c} x_{\boldsymbol{\alpha},r}$, and the map $\boldsymbol{\alpha}\mapsto (x_{\boldsymbol{\alpha},1},\ldots ,x_{\boldsymbol{\alpha},c})$ is continuous.
		\item [(i)] For any ${\boldsymbol h} \in \mathbb{R}^c,$ define $\tilde{h}:=\max_{1\le r\le c} h_r - h_1$. Then, there exists $\eta=\eta(c,v,{\tilde{ h}})>0$, such that for any $\boldsymbol{\alpha}=(\alpha_1,\alpha_2,\ldots,\alpha_c)$ with $\alpha_1,\ldots,\alpha_c\geq 0$ and $\alpha_1 \ge \eta+ \max_{2\le r\le c} \alpha_r$, the optimization problem \eqref{eq:potts_unconstrained_opt} has a unique maximizer. Furthermore, denoting this maximizer by $ (x_{\boldsymbol{\alpha}, {\boldsymbol h} ,1},\ldots ,x_{ \boldsymbol{\alpha},{\boldsymbol h} ,c})$, we have $ x_{\boldsymbol{\alpha}, {\boldsymbol h} , 1}>\max_{2\le r\le c} x_{ \boldsymbol{\alpha},{\boldsymbol h} ,r}$, and the map ${(\boldsymbol{\alpha},\boldsymbol h}) \mapsto (x_{ \boldsymbol{\alpha},{\boldsymbol h} ,1},\ldots ,x_{\boldsymbol{\alpha}, {\boldsymbol h} ,c})$ is continuous.
        \item [(ii)] For any $\boldsymbol{\alpha} \in \mathbb{R}_{\geq0}^c,$ define $\widetilde{\alpha}:=\max_{1\le r\le c} \alpha_r-\alpha_1$. Then, there exists $\eta=\eta(c,v,\widetilde{\alpha})>0$, such that for any ${\boldsymbol h}=(h_1,\ldots,h_c)$ with $h_1 \ge \eta+ \max_{2\le r\le c} h_r$, the optimization problem \eqref{eq:potts_unconstrained_opt} has a unique maximizer. Furthermore, denoting this maximizer by $ (x_{\boldsymbol{\alpha}, {\boldsymbol h} ,1},\ldots ,x_{ \boldsymbol{\alpha},{\boldsymbol h} ,c})$, we have $x_{{\boldsymbol{\alpha}},\boldsymbol{h},1}>\max_{2\le r\le c} x_{{\boldsymbol{\alpha}},\boldsymbol{h},r}$, and the map $({\boldsymbol{\alpha}},\boldsymbol{h})\mapsto  (x_{\boldsymbol{\alpha}, {\boldsymbol h} ,1},\ldots ,x_{ \boldsymbol{\alpha},{\boldsymbol h} ,c})$ is continuous.
        % \item [(iv)] Suppose $\alpha_i=\theta>0$ for all $i \in [c]$, where $\theta<\frac{2}{v(v-1)}$. Then there exist $\eta=\eta(c,v)>0$ such that for any ${\boldsymbol h}=(h_1,h_2,\ldots, h_c)$ with $h_1=h_2 \geq \eta + \max_{3\le r\le c} h_r $, the optimization problem \eqref{eq:potts_unconstrained_opt} has a unique maximizer. Furthermore, denoting this maximizer by $(x_{{\boldsymbol h},1},\ldots , x_{{\boldsymbol h},c})$, we have $x_{{\boldsymbol h},1}=x_{{\boldsymbol h},2} >\max_{r=3}^c x_{{\boldsymbol h},r}$, and the map ${\boldsymbol h}\mapsto (x_{{\boldsymbol h},1},\ldots ,x_{{\boldsymbol h},c})$ is continuous.
        \item [(iii)] Assume $0\leq \alpha_r\leq \frac{1}{v(v-1)}$ for all $r\in [c]$ and ${\boldsymbol h} \in \mathbb{R}^c$.  Then the maximizer of the optimization problem \eqref{eq:potts_unconstrained_opt} is unique.
        \end{itemize}
		\end{prop}

        \begin{remark}[Comparison with literature]\label{rem:complit}
	Substantial attention has been devoted to studying the case $v=2, W_\infty\equiv 1$ (see~\cite[Theorem 2.3]{gandolfo2010limit} for a summary of earlier results in the area). %For instance, in~\cite{wu1982potts}, the author shows the existence of a sharp critical point $\theta_{\rm crit}$ when $\boldsymbol h=0$. In~\cite{blanchard2008thermodynamic}, the authors provide conditions on $\theta$ under which the optimization problem \eqref{eq:potts_unconstrained_opt} admits a unique optimizer (potentially up to permutations).
    Going beyond the $v=2$ case,~\cref{prop:potts_gibbs_ldp} guarantees the existence of a phase transition point $\theta_{\rm crit}$ under which the optimization problem \eqref{eq:potts_unconstrained_opt} admits a unique optimizer, and above which there are exactly $c$ optimizers, in the case where ${\bf h}={\mathbf 0}$. \cref{prop:potts_diffalpha}
provides sufficient conditions on $(\theta,\boldsymbol h)$  for a unique optimizer.  Both these propositions apply to a significantly broader class of models for any $v\ge 2$, and $W_\infty$ is allowed to be a positive nonconstant tensor function satisfying the regularity condition $\mathcal{V}[W_\infty]\equiv 1$.

% 	One of the key takeaways of the preceding corollary is that the limiting log partition function (see \eqref{logpf}) of the model \eqref{eq:gibbs_potts2} is universal for any $W$ with $\mathcal{V}[W](x)=1$ $\lambda$-a.e. In the special case $H=K_2$ (quadratic interaction), this condition reduces to demanding $\int_{[0,1]} W(.,y)\,dy=1$ $\lambda$-a.e., which essentially means that $W$ is regular. This special case also follows from \cite[Thm 2.1]{Basak2017}. In this case, the optimization problem that arises in \eqref{eq:potts_unconstrained_opt} has been analyzed in~\cite[Theorem 2.3]{gandolfo2010limit}, where the authors show that \[
% \theta_{\rm crit}
% =
% \begin{cases}
% 1, & c=2,\\
% \frac{c-1}{c-2}\log(c-1), & c\ge3.
% \end{cases}
% \]  It seems to be of interest to compute this value for general $H$.
    
\end{remark}
        % \cref{prop:potts_diffalpha} identifies regimes where the optimization problem is effectively one-dimensional, with a single dominant coordinate determining the maximizer. This highlights how strong asymmetry in the parameters enforces uniqueness and suppresses competing configurations.

So far we have focused entirely on the unconstrained optimization problem in \eqref{eq:potts_unconstrained_opt}. The following result shows that, under suitable assumptions, the constrained optimization problem from the rate function (see \eqref{eq:i2}) can also be solved somewhat explicitly.

\begin{prop}[Closed-Form Rate Function in the Symmetric Case]\label{cor:Pottsconopt}
% Suppose ${\bf{X}}=(X_1,\ldots,X_n)$ where $X_i \overset{\text{i.i.d.}}{\sim} \mu$, defined in \eqref{mu}.
	%Consider the optimization problem \eqref{eq:potts_unconstrained_opt}, under the assumptions
 Suppose we have
\begin{align*}
		W_\infty>0 \,\,\,\lambda\text{-a.e.},\quad \mathcal{V}[W_\infty](x)=1\,\,\lambda\text{-a.e.},\quad
    \text{and} \newline \quad\alpha_r = \theta\geq0 \,\,\text{for all  } r\in [c]. 
		\end{align*}
        Then the following conclusions hold:

    % satisfying $\mu_r=P(X_1=r)\propto \exp(h_r)$ for all $r\in [c]$ for some $\boldsymbol h=(h_1,\ldots ,h_c)$.
    
    % Assume that \eqref{eq:cut} and \eqref{eq:q} hold, $\mathcal{V}[W](x)=1$ a.e., $W>0$ a.e., and for all $i$, $\alpha_i =\theta\geq 0$. Then the following conclusions hold:
    
	\begin{itemize}
		\item[(i)] Let $\boldsymbol h$ be such that $h_1\ge \eta+\max_{2\le r\le c}h_r$, where $\eta=\eta(c,v)$ is as in~\cref{prop:potts_diffalpha}, part (ii). For every $\theta\geq 0$, let $(x_{\theta,\boldsymbol h,1},\ldots ,x_{\theta,\boldsymbol h,c})$ be the unique optimizer from~\cref{prop:potts_diffalpha}, part (ii). Then the good rate function $I_{W_\infty,(1,\ldots, 1)}({\cdot} )$ (from \eqref{eq:i2}) simplifies  for  $t \in \big[\sum_{r=1}^c \mu_r^v, 1\big)$ as $$I_{W_\infty,(1,\ldots,1)}({ t})=\sum_{r=1}^c x_{\theta_{t},\boldsymbol h,r}\log{\frac{x_{\theta_{ t},\boldsymbol h,r}}{\mu_r}}.$$ Here $\theta_{ t}\geq 0$ denotes the unique solution to the equation $\sum_{r=1}^c x_{\theta_{ t},\boldsymbol h,r}^v=t$.
        
        % Moreover, for ${t}$ such that $t>1,$ we have $I_{(1,\ldots,1)}({t})=\infty.$

  %       \reihaneh{THIS IS FOR $N_2$ NOT $N$: on the range $t\in [\sum_{r=1}^c \mu_r^v,\infty)$ as follows:
		% $$I(t)=\begin{cases} \sum_{r=1}^c x_{\theta_t,r}\log{\frac{x_{\theta_t,r}}{\mu_r}} & \mbox{if}\ t\in \big[\sum_{r=1}^c \mu_r^v,1\big), \\  \infty & \mbox{if}\ t>1\end{cases}.$$
		% Here $\theta_t\geq 0$ is the unique solution to the equation $\sum_{r=1}^v x_{\theta,r}^v=t$ in $\theta\ge 0$.}
		
		\item[(ii)] Let $h_r=0$ for all $r\in [c]$ and $\theta_{\rm crit}$ be as in~\cref{prop:potts_gibbs_ldp} part (ii). For any $\theta>\theta_{\rm crit}$, define $(x_{\theta,1},\ldots ,x_{\theta,c})$ as the unique optimizer which satisfies $x_{\theta,1}>\max_{2\le r\le c} x_{\theta,r}$, the existence of which is guaranteed by ~\cref{prop:potts_gibbs_ldp} part (ii). Then the good rate function $I_{W_\infty,(1,\ldots,1)}$  simplifies as in part (i), on the range ${ t}$ where $t \in \big(y,1\big)$, where $y:= \lim_{\theta\downarrow \theta_{\rm crit}}\sum_{r=1}^c x^v_{{\theta},r}.$
        
        % \text{and } \theta_{\rm crit}^+$ denotes any number such that $\theta_{\rm crit}<\theta_{\rm crit}^+.$ 
        
        % $\theta_{\boldsymbol t}$ now denotes the unique solution to the equation $\sum_{r=1}^c x^v_{\theta_{\boldsymbol t},r}=\sum_{r=1}^ct_r$.
        % Moreover, for $t>1,$ we have $I_{(1,\ldots,1)}({ t})=\infty.$
		%on the range $t\in (\sum_{r=1}^c x_{\tilde\theta_{\rm crit},r}^v,\infty)$ as follows:
		%$$I_2(t)=\begin{cases} \sum_{r=1}^c x_{\theta_t,r}\log{ c x_{\theta_t,r}} & \mbox{if}\ t\in \big(\sum_{r=1}^v x_{\tilde\theta_{\rm crit},r}^v,1\big], \\ \log{q} & \mbox{if}\ t=1, \\ \infty & \mbox{if}\ t>1\end{cases},$$
		%where $\theta_t$ is defined as in part (i).
	\end{itemize}
\end{prop}

%This result shows that, in the symmetric regime, the complexity of the variational problem reduces to solving a one-parameter equation. In particular, the rate function can be computed explicitly once the optimizer is identified, providing a direct link between the optimization problem and the large deviation behavior.

 \subsubsection{$c=v=2$}
To further illustrate the structure of the optimization problem, we now specialize to the two-color and two-vertex case, where more explicit analysis is possible. 
% Therefore, in this subsection, we consider the case $c=2$ and $H=K_2$. 
\begin{prop}\label{c2-1}
 Suppose $c=v=2$, $\int_0^1 W_\infty (., y) dy = 1$ $\lambda$-a.e., $W_\infty>0$ $\lambda$-a.e., $\mu(\{1\})=p_1=p\in(0,1)$, $\mu(\{2\})= p_2=1-p$, and $\boldsymbol{\alpha}= (\alpha_1,\alpha_2)$ where $\alpha_1,\alpha_2\geq 0$.
 \begin{itemize}
     \item [(i)] If $\alpha_1 + \log p_1\neq \alpha_2+\log p_2$, then the optimization problem \eqref{eq:potts_unconstrained_opt} has a unique maximizer.
 \item[(ii)] If $\alpha_1 + \log p_1=\alpha_2+\log p_2$, then \eqref{eq:potts_unconstrained_opt} has one optimizer if $\frac{\alpha_1+\alpha_2}{2}\leq 1$ and two maximizers if $\frac{\alpha_1+\alpha_2}{2}> 1$.
 \end{itemize}
 
 % \item [(ii)] In particular, consider the case $\boldsymbol{\alpha}=(\alpha_1,0)$. If $p\neq \frac{1}{1+\exp({\alpha_1})}$, then \eqref{eq:potts_unconstrained_opt} has a unique maximizer.  If  $p= \frac{1}{1+\exp({\alpha_1})}$ and $\alpha_1\le 2 \,(\alpha_1>2)$, \eqref{eq:potts_unconstrained_opt}  has one (two) maximizer(s). 
\end{prop}
\cref{c2-1} shows that even in the simplest two-color setting, the optimization problem can exhibit multiple regimes, depending on the balance between the parameters. In particular, the transition from a unique maximizer to multiple maximizers reflects a qualitative change in the structure of the solution.

\begin{prop}\label{hmm?}
    Under the assumption of \cref{c2-1}, suppose we further have $\alpha_1=\alpha_2\geq 0$ and $p> 1/2$.
        % \item [(i)] For $\boldsymbol t$ such that $t_1+t_2 \in [p^2+(p^2-2p+1),1)$, we have $I({\boldsymbol t})=x_{\theta_{\bf{t}}}\log (\frac{x_{\theta_{\bf{t}}}}{p})+(1-x_{\theta_{\bf{t}}})\log(\frac{1-x_{\theta_{\bf{t}}}}{1-p}),$ where $x_{\theta_{\bf{t}}}>1/2$ is the unique solution to $x_{\theta_{\bf{t}}}^2+(x_{\theta_{\bf{t}}}^2-2x_{\theta_{\bf{t}}}+1)=t_1+t_2$. Moreover, for $\boldsymbol t$ such that $t_1+t_2>1$, we have $I(\boldsymbol t)= \infty$.
%  According to \eqref{eq:Irho}, $I_{(1,1)}$
% denotes the good rate function for $N_1(K_2,Q_n,{\bf{X}})+N_2(K_2,Q_n,{\bf{X}})$.    %  If $\tilde{I}$ denotes the good rate function for $N_1+N_2$, we have $$
    % \tilde{I}({y}):=\inf_{{\boldsymbol f}\in \mathcal{F}_2:\ G_{W,1}({\boldsymbol f})+G_{W,2}({\boldsymbol f})=y}  \Big\{ \int_0^1 f(u)\log \frac{f(u)}{p}du+\int_0^1 (1-f(u))\log \frac{1-f(u)}{1-p}du\Big\}.$$
    
    Then, for  $y \in [p^2+(1-p)^2,1)$ we have \[I_{W_\infty,(1,1)}({y})=q_{y}\log \Big(\frac{q_{{{y}}}}{p}\Big)+(1-q_{{{y}}})\log\Big(\frac{1-q_{y}}{1-p}\Big),\] where $I_{{W_\infty},{\boldsymbol \alpha}}(\cdot)$ is as in \eqref{eq:i2}, and $q_{y}>1/2$ is the unique solution to $q^2+(1-q)^2=y$. 
    % For  $y>1$, we have $I_{(1,1)}(y)= \infty$. 

\end{prop}

    This provides an explicit expression for the rate function in the symmetric case $\alpha_1=\alpha_2$. Our next result focuses on an asymmetric case, showing that the situation is more delicate there,
    %showing that it depends only on a scalar parameter determined by the constraint. Such simplifications are specific to low-dimensional settings and do not persist in general.
%We now consider a setting where the symmetry between coordinates is broken. As we will see, this leads to a more intricate structure of the optimization problem,
including the possibility of nonconstant optimizers in the constrained formulation.

\begin{prop}\label{unc}
 Under the assumption of \cref{c2-1}, suppose $\alpha_1\ge 0$ and $\alpha_2=0$.

\begin{itemize}
    \item [(i)] If $p\ge\frac{1}{1+e^2}$, for $y\in [p^2,1)$ we have 
    \begin{align}\label{eq:ratei11}I_{W_\infty,(1,0)}({y})= \sqrt{y}\log (\frac{\sqrt{y}}{p})+(1-\sqrt{y})\log(\frac{1-\sqrt{y}}{1-p}),
\end{align}
where $I_{W_\infty,(1,0)}$ is defined in \eqref{eq:i2}.
% and for $y>1$, $I_{(1,0)}({y})=\infty$.
\item [(ii)] \begin{itemize}
    \item [(a)]
If $p<\frac{1}{1+e^2} $, for \[y\in [p^2, (1-\eta)^2)\cup (\eta^2,1),\]
$I_{W_\infty,(1,0)}({y})$ is the same as \eqref{eq:ratei11}, where $\eta>1/2$ is the unique value satisfying $2\eta-1=\tanh(\frac{1}{2}\log \frac{1-p}{p}(2\eta-1)).$

\item[(b)] There exist  $y_\infty\in \big[(1-\eta)^2, \eta^2\big],$ and  $W_\infty$ satisfying $W_\infty>0$ $\lambda$-a.e. and $\int_0^1 W_\infty(.,y)dy=1$ $\lambda$-a.e.,
such that all optimizers of the optimization problem 
\[\inf_{{\boldsymbol f}=(f,1-f):\ G_{W_{\infty},\mathbf{1}_1}({\boldsymbol f})=y_\infty}  \Big\{ \int_0^1f(u)\log \frac{f(u)}{p}+(1-f(u))\log\frac{1-f(u)}{1-p}du\Big\}\]
are nonconstant functions.
\end{itemize}
% \begin{align*}\label{eq:ratei11}I({y})=\begin{cases} \sqrt{y}\beta(\sqrt{y}) - \gamma(\beta(\sqrt{y})), & \mbox{if}\ y\in [p^2,1)\\ \infty & \mbox{if}\ y>1\end{cases},
% \end{align*}
% where the functions $\gamma, \beta$ are defined in \cref{def:tilt1}.
\end{itemize}

\end{prop}
\begin{remark}
\cref{hmm?} and \cref{unc} illustrate how the distinction between the unconstrained
and constrained optimization problems depends on whether
$\alpha_1=\alpha_2$ or $\alpha_1\neq\alpha_2$.
% illustrate a key distinction between the unconstrained and constrained optimization problems between the symmetric case $\alpha_1=\alpha_2$ and the asymmetric case $\alpha_1\ne \alpha_2$.
While the unconstrained problem admits constant optimizers in both cases, the constrained problem can exhibit nonconstant solutions in the asymmetric case, reflecting a richer structure in the presence of asymmetry. The existence of symmetry breaking was already demonstrated in \cite[Theorem 1.1(b)]{SomabhaBhattacharya2020}, with a graphon $W$ which is not positive a.e. We extend this to allow for $W$ to be strictly positive everywhere, using continuity arguments for local perturbations.
%which draws a true contrast between symmetric and the asymmetric cases.
\end{remark}

% \begin{defn}\label{def:tilt1}
% If $\mu$ is Bernoulli distribution with parameter $p$, its log moment generating function is $\gamma(\theta):=\log ( pe^\theta +1-p)$, and  $\gamma'(\theta)=\frac{p e^{\theta}}{pe^{\theta}+1-p}$. One could verify that $\gamma'(.)$ is strictly increasing on $\R$, and has an inverse $\beta(.):(0,1)\mapsto \R $. \end{defn}
Finally, we  consider a special case  in the contrasting regime where the interaction parameters $(\alpha_1,\alpha_2)$ have different signs (so outside the ferromagentic regime). As it turns out, in this case there is a unique explicit constant optimizer.
\begin{prop}\label{negvalue1}
   Suppose $c=v=2$, $\int_0^1 W_\infty (., y) dy = 1$ $\lambda$-a.e., $W_\infty>0$ $\lambda$-a.e., $\mu(\{1\})=p_1=p\in(0,1)$, $\mu(\{2\})= p_2=1-p$, and $\boldsymbol{\alpha}=(\theta,-\theta) $ for some $\theta \in \R$. 
 \begin{itemize}
     \item[(i)] The maximizer of \eqref{eq:potts_unconstrained_opt} is $(\frac{p\exp({2\theta})}{1-p+p\exp({2\theta})},\frac{1-p}{1-p+p\exp({2\theta})})$ $\lambda$-a.e. and hence it is unique. 
     \item [(ii)] For $y\in(-1,1)$, the good rate function $I_{W_\infty,(1,-1)}(y)$ (see \eqref{eq:i2}) can be simplified as 
     \[ \Big(\frac{y+1}{2}\Big)\log \Big(\frac{y+1}{2p}\Big)+\Big(\frac{1-y}{2}\Big)\log\Big(\frac{1-y}{2(1-p)}\Big).\]
%      \begin{align*}I_{(1,-1)}({y})=\begin{cases} (\frac{y+1}{2})\log (\frac{y+1}{2p})+(\frac{1-y}{2})\log(\frac{1-y}{2(1-p)}), & \mbox{if}\ y\in (0,1)\\ \infty & \mbox{if}\ y>1\,\end{cases}.
% \end{align*}
\end{itemize}\end{prop}

\section{Proofs}
\subsection{Proof of Main Results}
Throughout this section, the values of generic constants $C, C',\ldots$, etc.  may
change from line to line. Before beginning the proof of \cref{NewLDP}, we state two lemmas which are proven in \cref{AuxiliarySection}.
    
    \begin{lmm}\label{bijection}
The map $\nu\mapsto {\boldsymbol{f}}^\nu$ from $\mathcal{P}$ to $\mathcal{F}_c$,
defined by $f^{\nu}_r(x):=\nu(B=r\mid A=x)$ for $r\in[c]$, is a bijection. Moreover, for all $\phi$ and $W\in \mathcal{W}_v$, $T_{W,\phi}(\nu)=G_{W,\phi}({\boldsymbol f}^\nu)$.
\end{lmm}
\begin{lmm}[Functional form of the cut norm]\label{functionalcut}
Let \(W\in L^1([0,1]^v)\). Then
\[
\|W\|_{\square}
=
\sup_{0\le g_1,\ldots,g_v\le 1}
\left|
\int_{[0,1]^v}
W(x_1,\ldots,x_v)
\prod_{a=1}^v g_a(x_a)\,
dx_1\cdots dx_v
\right|,
\]
where the supremum is over measurable functions \(g_a:[0,1]\to[0,1]\).
\end{lmm}

\begin{proof}[Proof of \cref{NewLDP}]
\begin{itemize}
    \item [(i)]
   Let ${\mathcal{P}}$ denote the set of probability measures on $[0,1] \times[c]$, equipped with weak topology, where the first marginal is uniform on $[0,1]$.
For $W\in \mathcal{W}_v$, we define $T_{W,\phi}$ on ${\mathcal{P}}$ as 
\[
T_{W,\phi}(\nu) := \mathbb{E} \left[ W(A_1, \ldots, A_v)\phi(B_1, \ldots,B_v) \right],
\]
where the expectation is over $\{(A_a, B_a)\}_{1 \leq a \leq v} \overset{i.i.d.}{\sim} \nu$.

% \textcolor{red}{Make tensor $Q_n$, make $(B_i',B_i)=(A_i,B_i)$.}

Now it is not hard to check that
\[{N}_n(Q_n,\phi,\mathbf{X})={T}_{W_{Q_n},\phi}(\tilde{\mathcal{L}}_n),\] where $\tilde{\mathcal{L}}_n=\tilde{\mathcal{L}}_n({\bf X})\in {\mathcal{P}}$  is the joint law of $\left(U, X_{\lceil n U\rceil}\right)$ conditional on $\bf X$, and  $U$ is a uniform random variable on $[0,1]$. For $r\in[c]$, define
\[
f_{n,r}(x):=f^{\widetilde{\mathcal L}_n}_r(x)=
\widetilde{\mathcal L}_n(B=r\mid A=x).
\]
Then ${\boldsymbol f}_n=(f_{n,1},\ldots,f_{n,c})\in\mathcal F_c$, and
$T_{W,\phi}(\widetilde{\mathcal L}_n)
=
G_{W,\phi}({\boldsymbol f}_n)$ for all $W\in \mathcal{W}_v$, by \cref{bijection}. Now,

\begin{align}\label{eq:gpt_best}
&\left|{T}_{W_{Q_n},\phi}(\tilde{\mathcal{L}}_n)-{T}_{W_\infty,\phi}(\tilde{\mathcal{L}}_n)\right|=
\left|
G_{W_{Q_n},\phi}({\boldsymbol f}_n)
-
G_{W_\infty,\phi}({\boldsymbol f}_n)
\right| \nonumber\\
&=\left|\int_{[0,1]^v}
(W_{Q_n}-W_{\infty})(x_1,\ldots,x_v)\,\Phi_{\boldsymbol f_n}(x_1,\ldots,x_v)\,dx_1\cdots dx_v,\right| \nonumber \\
&=\left|\int_{[0,1]^v}
(W_{Q_n}-W_{\infty})(x_1,\ldots,x_v)\,\sum_{r_1,\ldots,r_v\in[c]}
\phi(r_1,\ldots,r_v)\prod_{a=1}^v f_{n,r_a}(x_a)\,dx_1\cdots dx_v,\right| \nonumber \\
&\le
\sum_{r_1,\ldots,r_v\in[c]}
|\phi(r_1,\ldots,r_v)|
\left|
\int_{[0,1]^v}
(W_{Q_n}-W_\infty)(x_1,\ldots,x_v)
\prod_{a=1}^v
f_{n,r_a}(x_a)
\,dx_1\cdots dx_v
\right| \nonumber \\
&\le
\left(
\sum_{r_1,\ldots,r_v\in[c]}
|\phi(r_1,\ldots,r_v)|
\right)
\|W_{Q_n}-W_\infty\|_{\square} \nonumber \\
&\le
\frac{1}{C_v}
\left(
\sum_{r_1,\ldots,r_v\in[c]}
|\phi(r_1,\ldots,r_v)|
\right)
\|W_{Q_n}-W_\infty\|_{\square^*},
\end{align}

where in the fourth line we used \cref{functionalcut}.
% Since $\phi:[c]^v\mapsto \R$, we can write 
% \begin{align*}
% \phi(x_1,\ldots,x_v)=\sum_{r_1,\ldots,r_v=1}^c \phi(r_1,\ldots,r_v) \mathbf{1}_{\{x_1=r_1,\ldots, x_v=r_v\}},
% \end{align*}
% and so $\phi$ is a finite linear combination of product functions.
% \begin{align*}
%     T_{W_{Q_n},\phi}(\tilde{\mathcal{L}}_n)=\sum_{r_1,\cdots,r_v=1}^c \phi(r_1,\cdots,r_v) \mathbf{1}_{x_1=r_1,\ldots, x_v=r_v}
% \end{align*}
Therefore, \eqref{eq:cut} gives the exponential equivalence of ${T}_{W_{Q_n},\phi}(\tilde{\mathcal{L}}_n)$ and ${T}_{W_\infty,\phi}(\tilde{\mathcal{L}}_n)$ (see \cite[Theorem 4.2.13]{DZ}), and so it suffices to derive an LDP for ${T}_{W_\infty,\phi}(\tilde{\mathcal{L}}_n)$.
% \begin{proof}[Proof of \cref{new Lemma 2.5}]
% This proof is very similar to the proof of \cite[Lemma 2.5]{bhattacharya2024ldp}, considering $W_{Q_n}, W_\infty$ instead of $W^*_n, W^*$ respectively, so we skip the details. The main difference is 

% that we prove in \cref{AuxiliarySection}
% \begin{lmm}\label{New Proposition 3.1}
% Let $\{W_{Q_n}\}_{n\ge1}$ and $W_\infty$ be $v$-hypergraphons such that
% \[
% \limsup_{n\to\infty}\|W_{Q_n}\|_1<\infty,
% \qquad
% \|W_{Q_n}-W_\infty\|_{\square^*}\longrightarrow0.
% \]
% Then, setting 
% \[t(W,{\bf g}):=\int_{[0,1]^v}W(x_1, x_2, \ldots, x_v) \prod_{r=1}^v g_r(x_r)dx_r,\]
% we have
% \begin{align}\label{t}
% &\lim_{n\to\infty}
% \sup_{{\bf g}:[0,1]^v \mapsto [-1,1]^v}
% \big| t(W_{Q_n},{\bf g}) - t(W_\infty,{ \bf g}) \big|
% = 0.\\ \nonumber
% \end{align}
% \end{lmm}
% \end{proof}

% \begin{lmm}\label{new Lemma 2.5}
% Suppose $\{W_{n}\}_{n\ge 1}$ be a sequence of functions in $ \mathcal{W}_v$ that satisfy \eqref{eq:cut} and \eqref{eq:l1}.
% Then, for any function $\phi:[c]^v\to \R$ and $\delta > 0$, we have
% \[\lim_{n\to\infty} \frac{1}{n} \log \mathbb{P}
% \left(
% \left| 
% T_{W_{n},\phi}(\tilde{\mathcal L}_n) 
% - 
% T_{W_\infty,\phi}(\tilde{\mathcal L}_n)
% \right| 
% \ge \delta
% \right)
% = -\infty.\]
% \end{lmm}
Since $W_\infty\in L^1([0,1]^v)$, we get that the function $\nu \mapsto {T}_{W_\infty,\phi}(\nu)$ is well-defined and finite.
Moreover, by approximating $W_\infty$ in $L^1([0,1]^v)$ by bounded continuous functions, one can show that the function $\nu\mapsto T_{W_\infty,\phi}(\nu)$ is  continuous w.r.t. weak topology on ${\mathcal{P}}$. Invoking the known LDP of $\tilde{\mathcal{L}}_n$ (see \cite[Lemma 2.1 (ii)]{bhattacharya2024ldp}) along with the contraction principle (see \cite[Theorem 4.2.1]{DZ}), it follows that ${T}_{W_\infty,\phi}(\tilde{\mathcal{L}}_n)$ satisfies an LDP with the good rate function 
    \begin{equation*}
    J_{W_\infty}({{t}})=\inf _{\nu \in {\mathcal{P}}:  {{T}}_{W_\infty, \phi}(\nu)={{t}}} D(\nu \mid {{ \nu^*}}),
    \end{equation*}
    where $\nu^*:=\mathscr{U}[0,1]\times \mu$, $\mathscr{U}[0,1]$ denotes the uniform random variable on $[0,1]$, and $D(\cdot|\cdot)$ denotes the standard Kullback-Leibler divergence.
    
% For any $\nu\in \mathcal{P}$, defining $f^{\nu}_r(x):=\nu(B=r|A=x)$ for $r\in[c]$,
% we have ${\boldsymbol f}^\nu=(f^\nu_1,\ldots,f^\nu_c)\in \mathcal{F}_c$, i.e.~$\sum_{r=1}^cf_r^\nu(x)=1$ for all $x\in [0,1]$. Also, given  ${\boldsymbol{f}}=(f_1,\dots,f_c)\in \mathcal{F}_c$, one can construct a $\nu\in \mathcal{P}$  such that ${\boldsymbol{f}}^\nu={\boldsymbol{f}}$. Then the map $\nu\mapsto {\boldsymbol{f}}^\nu$ from $\mathcal{P}$ to $\mathcal{F}_c$ is a bijection.

Using \cref{bijection}, we can write
\[
D(\nu\mid\nu^*)
=
\int_0^1\sum_{r=1}^c f^\nu_r(x)\log\frac{f_r^\nu(x)}{\mu_r}\,dx .
\]
% Also with $G_{W_\infty,\phi}(\cdot)$ as in \cref{def:g2},
% using the law of iterated expectations one can check that $T_{W_\infty,\phi}(\nu)=G_{W_\infty,\phi}({\boldsymbol f}^\nu)$.

Thus, our rate function $J_{W_\infty}(\cdot)$ simplifies as
\[J_{W_\infty}(t)=\inf_{{\boldsymbol f}\in \mathcal{F}_c:\ { G}_{W_\infty,\phi}({\boldsymbol f})={t}}  \left\{ \int_0^1\sum_{r=1}^c f_{r}(u)\log \frac{f_{r}(u)}{\mu_r}du\right\}.\]

\item[(ii)] Using \eqref{eq:gibbs_potts} gives
				\begin{align*}
				    e^{nZ_{Q_n}(\beta,\mu)}=
				    &\E_{{\mu^{\otimes n}}}\exp\Big(n\beta {N}_n(Q_n,\phi, \mathbf{X})\Big),
				\end{align*}
	which along with Varadhan's Lemma ~\cite[Theorem 4.3.1]{DZ} gives
                \begin{align*}
                    \frac{1}{n}\log\E_{\mu^{\otimes n}}\exp\Big(n\beta {N}_n(Q_n,\phi,\mathbf{X})\Big) \xrightarrow{n\to\infty}&\sup_{x\in \R}\, \{\beta  x-J_{W_\infty}(x)\}\\
                 =& \sup_{{\boldsymbol f}\in \mathcal{F}_c}  \Big\{\beta{ G}_{{W_\infty},\phi}(\boldsymbol f) -\int_0^1\sum_{r=1}^c f_{r}(u)\log \frac{f_{r}(u)}{{\mu}_r}du\Big\},
                \end{align*}
                where the last equality uses part (i). 
                % Notice that the condition 4.3.3 in ~\cite[Theorem 4.3.1]{DZ} is guaranteed by the assumption \eqref{eq:l1}.
                Notice that condition 4.3.3 in~\cite[Theorem 4.3.1]{DZ} is satisfied since~\eqref{eq:gpt_best} and \eqref{eq:cut} imply that there exists a constant $C<\infty$ such that
\[
\sup_{n\ge 1}\sup_{x\in[c]^n}\big|N_n(Q_n,\phi,x)\big|\le C.
\]
                To verify that the maximizers are attained, using \cref{bijection}, it suffices to show that the maximizers of 
                the optimization problem
                \begin{align*}
                   \sup_{\nu \in \mathcal{P}} \{\beta T_{W_\infty,\phi}(\nu)-D(\nu|\nu^*)\}
                \end{align*}
                are attained. But this follows on noting that the space $\mathcal{P}$ is compact w.r.t. weak topology, and the function $\nu\mapsto \beta T_{W_\infty,\phi}(\nu)-D(\nu|\nu^*)$ is upper semi-continuous.

 \end{itemize}
\end{proof}

\begin{proof}[Proof of \cref{prop:propotts2}]
			\begin{itemize}
			    
			\item[(i)] Using \cref{NewLDP} (ii) the set of maximizers of ~\eqref{eq:potts_unconstrained_opt2} is nonempty. Let ${\boldsymbol f}\in\mathcal{F}_c$ be  a maximizer of~\eqref{eq:potts_unconstrained_opt2}.
            Define
            \[\Xi(\boldsymbol{f}):=\beta{G}_{{W_\infty},\phi}({\boldsymbol f})
-\int_0^1\sum_{r=1}^c f_r(u)\log \frac{f_r(u)}{\mu_r}\,du.\]
We first claim that
\[
0<f_r(x)<1,
\qquad r\in[c],
\quad \lambda\text{-a.e. }x\in[0,1].
\]
Indeed, it is enough to prove that \(f_r>0\) \(\lambda\)-a.e. for every \(r\), since
\(\sum_{r=1}^c f_r=1\) and \(c\ge2\) then imply \(f_r<1\) \(\lambda\)-a.e.

Suppose, toward a contradiction, that for some \(r\in[c]\), the set
\[
B_r:=\{x\in[0,1]: f_r(x)=0\}
\]
has positive Lebesgue measure. For \(\varepsilon\in(0,1)\), define
${\boldsymbol f}^\varepsilon:=(1-\varepsilon){\boldsymbol f}+\varepsilon\mu$,
where \(\mu\) is viewed as the constant element of \(\mathcal F_c\). Then
\({\boldsymbol f}^\varepsilon\in\mathcal F_c\). Since \(G_{W_\infty,\phi}\) is multilinear in
\({\boldsymbol f}\) and \(W_\infty\in L^1([0,1]^v)\), there exists \(C<\infty\) such that
\[
\left|
\beta G_{W_\infty,\phi}({\boldsymbol f}^\varepsilon)
-
\beta G_{W_\infty,\phi}({\boldsymbol f})
\right|
\le C\varepsilon .
\]
% Therefore, for some \(C'<\infty\),
% \[
% \int_0^1\sum_{s=1}^c f_s^\varepsilon(u)\log\frac{f_s^\varepsilon(u)}{\mu_s}\,du
% -
% \int_0^1\sum_{s=1}^c f_s(u)\log\frac{f_s(u)}{\mu_s}\,du
% \le
% \varepsilon\mu_r\lambda(B_r)\log\varepsilon
% +
% C'\varepsilon .
% \]

Let
\[
\psi_s(t):=t\log\frac{t}{\mu_s},\qquad t\in[0,1],
\]
with the convention \(0\log 0=0\). 
On \(B_r\), we have \(f_r^\varepsilon(x)=\varepsilon\mu_r\), and so
$f_r^\varepsilon(x)\log[ f_r^\varepsilon(x)/{\mu_r}]
=
\varepsilon\mu_r\log\varepsilon$. Therefore,
\[
\psi_r(f_r^\varepsilon(x))-\psi_r(f_r(x))
=
\varepsilon\mu_r\log\frac{\varepsilon\mu_r}{\mu_r}-0
=
\varepsilon\mu_r\log\varepsilon,
\qquad x\in B_r.
\]
Integrating over \(B_r\), this gives
\[
\int_{B_r}
\left[
\psi_r(f_r^\varepsilon(u))-\psi_r(f_r(u))
\right]du
=
\varepsilon\mu_r\lambda(B_r)\log\varepsilon.
\]

It remains to bound the contribution from all other terms. Since each \(\psi_s\) is
convex on \([0,1]\), we have
\[
\psi_s(f_s^\varepsilon)
=
\psi_s((1-\varepsilon)f_s+\varepsilon\mu_s)
\le
(1-\varepsilon)\psi_s(f_s)+\varepsilon\psi_s(\mu_s).
\]
But \(\psi_s(\mu_s)=0\), so
\[
\psi_s(f_s^\varepsilon)-\psi_s(f_s)
\le
-\varepsilon\psi_s(f_s).
\]
Since \(\psi_s\) is bounded below on \([0,1]\), there exists \(C_s<\infty\) such that
\[
-\psi_s(t)\le C_s,\qquad t\in[0,1].
\]
Hence
\[
\psi_s(f_s^\varepsilon)-\psi_s(f_s)\le C_s\varepsilon.
\]
Summing over \(s\in[c]\) and integrating over \([0,1]\), while keeping the sharper
contribution of the \(r\)-th coordinate on \(B_r\), we obtain, for some
\(C'<\infty\),
\[
\int_0^1\sum_{s=1}^c \psi_s(f_s^\varepsilon(u))\,du
-
\int_0^1\sum_{s=1}^c \psi_s(f_s(u))\,du
\le
\varepsilon\mu_r\lambda(B_r)\log\varepsilon
+
C'\varepsilon .
\]
Equivalently,
\[
\int_0^1\sum_{s=1}^c f_s^\varepsilon(u)
\log\frac{f_s^\varepsilon(u)}{\mu_s}\,du
-
\int_0^1\sum_{s=1}^c f_s(u)
\log\frac{f_s(u)}{\mu_s}\,du
\le
\varepsilon\mu_r\lambda(B_r)\log\varepsilon
+
C'\varepsilon .
\]
It follows that
\[
\Xi({\boldsymbol f}^\varepsilon)-\Xi({\boldsymbol f})
\ge
\varepsilon\mu_r\lambda(B_r)|\log\varepsilon|-C''\varepsilon .
\]
For \(\varepsilon>0\) sufficiently small, the right-hand side is positive, contradicting
the maximality of \({\boldsymbol f}\). Thus \(f_r>0\) \(\lambda\)-a.e. for every \(r\), and hence
\(0<f_r<1\) \(\lambda\)-a.e. for every \(r\).

For $r\in[c]$, define
\[
q_r(x):=
\frac{
\mu_r\exp\bigl(\beta\mathcal T_r[W_\infty,\phi,\boldsymbol f](x)\bigr)
}{
\sum_{a=1}^c
\mu_a\exp\bigl(\beta\mathcal T_a[W_\infty,\phi,\boldsymbol f](x)\bigr)
}.
\]
We claim that $f_r=q_r$ $\lambda$-a.e. for every $r\in[c]$.
Suppose, toward a contradiction, that they disagree on a set of
positive Lebesgue measure. Since
$\sum_{r=1}^c f_r(x)=\sum_{r=1}^c q_r(x)=1$
for $\lambda$-a.e. $x$, there exist distinct $r,s\in[c]$ and a
measurable set $A\subseteq[0,1]$ with $\lambda(A)>0$ such that
\[
f_r(x)>q_r(x)
\qquad\text{and}\qquad
f_s(x)<q_s(x),
\qquad x\in A.
\]
After removing a null subset of $A$, we may assume that
$f_r>0$, $0<f_s<1$, and
$\mathcal T_r[W_\infty,\phi,\boldsymbol f]$ and
$\mathcal T_s[W_\infty,\phi,\boldsymbol f]$ are finite on $A$.
For each integer $M\geq2$, define
\[
A_M
:=
A\cap
\left\{
f_r\geq\frac1M,\quad
\frac1M\leq f_s\leq1-\frac1M,\quad
\left|\mathcal T_r[W_\infty,\phi,\boldsymbol f]\right|
+
\left|\mathcal T_s[W_\infty,\phi,\boldsymbol f]\right|
\leq M
\right\}.
\]
Then $A_M\uparrow A$ as $M\to\infty$. Hence, there exists an integer
$M\geq2$ such that $\lambda(A_M)>0$.

For $0<t<1/(2M)$, define
$\boldsymbol f^{(t)}=(f_1^{(t)},\ldots,f_c^{(t)})$ by
\[
f_r^{(t)}(x):=f_r(x)-t\mathbf 1_{A_M}(x),
\qquad
f_s^{(t)}(x):=f_s(x)+t\mathbf 1_{A_M}(x),
\]
and
\[
f_a^{(t)}(x):=f_a(x),
\qquad a\notin\{r,s\}.
\]
Then $\boldsymbol f^{(t)}\in\mathcal F_c$ for every
$0<t<1/(2M)$. Indeed, on $A_M$,
\[
f_r^{(t)}(x)
\geq\frac1M-t
\geq\frac1{2M}>0,
\]
whereas
\[
\frac1M
\leq f_s^{(t)}(x)
\leq1-\frac1M+t
\leq1-\frac1{2M}<1.
\]
Outside $A_M$, the coordinates are unchanged. Moreover,
$\sum_{a=1}^c f_a^{(t)}(x)
=
\sum_{a=1}^c f_a(x)=1$.

Since $\boldsymbol f$ maximizes $\Xi$, for every
$0<t<1/(2M)$,
$\Xi(\boldsymbol f^{(t)})-\Xi(\boldsymbol f)\leq0$,
and therefore
\[
\frac{\Xi(\boldsymbol f^{(t)})-\Xi(\boldsymbol f)}{t}\leq0.
\]
We now let $t\downarrow0$. On $A_M$, for every
$0<t<1/(2M)$,
\[
f_r^{(t)}(x)\in\left[\frac1{2M},1\right],
\qquad
f_s^{(t)}(x)\in
\left[\frac1M,1-\frac1{2M}\right].
\]
Hence, by the mean value theorem, the difference quotients
corresponding to
\[
u\longmapsto u\log\frac{u}{\mu_r}
\qquad\text{and}\qquad
u\longmapsto u\log\frac{u}{\mu_s}
\]
are uniformly bounded on $A_M$, independently of
$0<t<1/(2M)$. Since $A_M$ has finite Lebesgue measure, the dominated
convergence theorem applies to the entropy terms.

Moreover, by the multilinearity of $G_{W_\infty,\phi}$,
\[
\lim_{t\downarrow0}
\frac{
G_{W_\infty,\phi}(\boldsymbol f^{(t)})
-
G_{W_\infty,\phi}(\boldsymbol f)
}{t}
=
\int_{A_M}
\left[
\mathcal T_s[W_\infty,\phi,\boldsymbol f](x)
-
\mathcal T_r[W_\infty,\phi,\boldsymbol f](x)
\right]\,dx.
\]
Consequently,
\begin{align*}
0
&\geq
\lim_{t\downarrow0}
\frac{\Xi(\boldsymbol f^{(t)})-\Xi(\boldsymbol f)}{t}\\
&=
\int_{A_M}
\Bigg[
\beta\mathcal T_s[W_\infty,\phi,\boldsymbol f](x)
-\beta\mathcal T_r[W_\infty,\phi,\boldsymbol f](x)
+\log\frac{f_r(x)}{\mu_r}
-\log\frac{f_s(x)}{\mu_s}
\Bigg]\,dx\\
&=
\int_{A_M}
\log\left(
\frac{f_r(x)q_s(x)}
{f_s(x)q_r(x)}
\right)\,dx.
\end{align*}
For every $x\in A_M\subseteq A$,
\[
f_r(x)>q_r(x)
\qquad\text{and}\qquad
f_s(x)<q_s(x),
\]
and therefore
$\frac{f_r(x)q_s(x)}{f_s(x)q_r(x)}>1$.
The last integrand is thus strictly positive on $A_M$. It is also
bounded on $A_M$, by the defining bounds of $A_M$, and
$\lambda(A_M)>0$. Hence
\[
\int_{A_M}
\log\left(
\frac{f_r(x)q_s(x)}
{f_s(x)q_r(x)}
\right)\,dx>0,
\]
which is a contradiction. This proves
\eqref{eq:general_fp_onlyWsym}.

			\item[(ii)] Suppose there exists a $\lambda$-a.e. constant function ${\boldsymbol f}=(f_1,\ldots ,f_c)\in\mathcal{F}_{c}$, say $f_r(x)=y_r$ for $\lambda$-a.e. $x\in [0,1]$, which maximizes \eqref{eq:potts_unconstrained_opt2}. With $\mathcal{T}_r, \mathcal{T}_r^{(m)}, \mathcal{V}$ as in \cref{def:tr}, using \eqref{eq:def_T_general_sum} and \eqref{deft1}  we get
            \begin{align*}
 & \mathcal T_r[{W_\infty},\phi,{\boldsymbol y}](x)=\sum_{m=1}^v\mathcal T_r^{(m)}[{W_\infty},\phi,{\boldsymbol y}](x)=\mathcal{V}[{W_\infty}](x) \Gamma_r({\boldsymbol y}),
\end{align*}
where $\Gamma_r$ is as in \eqref{eq:def_Gamma_general}.
Using~\eqref{eq:general_fp_onlyWsym} with $f_r(x)=y_r$ gives
\begin{align*}\label{niceform3}
       &\frac{y_r}{y_s}\stackrel{\lambda-a.e.}{=}\frac{\mu_r}{\mu_s}\exp\Big(\beta\big(\mathcal T_r[{W_\infty},\phi,{\boldsymbol y}](x)-\mathcal T_s[{W_\infty},\phi,{\boldsymbol y}](x)\big)\Big)\\
       &\notag=\frac{\mu_r}{\mu_s}\exp\Big(\beta\mathcal{V}[{W_\infty}](x) \big(\Gamma_r({\boldsymbol y})- \Gamma_s({\boldsymbol y})\big)\Big),
   \end{align*}
   for all $r,s\in[c]$. Since $\mathcal{V}[{W_\infty}](\cdot)$ is not constant $\lambda$-a.e. and $\beta\neq 0$, we must have 
\begin{align*}%\label{niceform4}
\Gamma_1({\boldsymbol y})
=\Gamma_2({\boldsymbol y})
=\cdots
=\Gamma_c({\boldsymbol y}).
\end{align*}
So the above two displays give
   \[\frac{y_r}{y_s}=\frac{\mu_r}{\mu_s},
\qquad r,s\in[c],\]
which means $(y_1,\ldots,y_c)=(\mu_1,\ldots,\mu_c)$
and hence $\Gamma_r({(\mu_1,\ldots,\mu_c)})=\Gamma_s((\mu_1,\ldots,\mu_c))$, for all $r,s\in [c]$, which is a contradiction.
          \end{itemize}  \end{proof}

            % \subsection{Proofs of \cref{}}
			
% Let $\mathcal N(H,\mathcal{G}_n), \mathcal N(\widetilde H,\mathcal{G}_n) $ be the number of copies of $H, \widetilde{ H}$ in the graph $\mathcal{G}_n$, respectively. Formally,
% \begin{align}\label{N}
% \mathcal{N}(H,\mathcal{G}_n):=\sum_{(i_1,\ldots,i_v)\in [n]^v} \prod_{\{a,b\}\in E(H)}\mathcal{G}_n(i_a,i_b).\end{align}

% Moreover, define
%  \begin{align}\label{M}
%  & \mathcal{M}(E',\mathcal{G}_n):=\sum_{(i_u)_{u\in V'}\in[n]^{\omega}}\prod_{\{a,b\}\in E'} \mathcal{G}_n(i_a,i_b).    
%  \end{align}
\begin{proof}[Proof of \cref{condldpdef}]
As shown in the proof of \cref{NewLDP},
\[
\sup_{\mathbf x\in[c]^n}
\left|
T_{W_{Q_n},\phi}\big(\widetilde{\mathcal L}_n(\mathbf x)\big)
-
T_{W_\infty,\phi}\big(\widetilde{\mathcal L}_n(\mathbf x)\big)
\right|
\leq
C_{\phi,v}\|W_{Q_n}-W_\infty\|_{\square^*}.
\]
Moreover,
$N_n(Q_n,\phi,\mathbf X)
=
T_{W_{Q_n},\phi}
\big(\widetilde{\mathcal L}_n(\mathbf X)\big)$.
Recall that $\widetilde{\mathcal L}_n$ satisfies an LDP on $\mathcal P$
with speed $n$, and that the map
$\nu\longmapsto T_{W_\infty,\phi}(\nu)$
is continuous w.r.t.\ the weak topology. Hence, by the contraction
principle \cite[Theorem 4.2.1]{DZ},
$T_{W_\infty,\phi}
\big(\widetilde{\mathcal L}_n(\mathbf X)\big)$
satisfies an LDP with speed $n$ and good rate function
$J_{W_\infty}$, as identified in the proof of \cref{NewLDP}.

Fix a Borel set $F\subseteq\mathbb R$ and $\varepsilon>0$. For
$\delta>0$, define
\[
F^\delta:=\{t:\operatorname{dist}(t,F)\le\delta\},
\qquad
F_{-\delta}:=\{t:\operatorname{dist}(t,F^c)>\delta\}.
\]
On the event
$\{C_{\phi,v}\|W_{Q_n}-W_\infty\|_{\square^*}\leq\delta/2\}$,
using the independence of $Q_n$ and $\mathbf X$, we have
\[
\mathbb P\left(
T_{W_\infty,\phi}
\big(\widetilde{\mathcal L}_n(\mathbf X)\big)\in F_{-\delta}
\right)
\leq
\mathbb P\big(N_n(Q_n,\phi,\mathbf X)\in F\mid Q_n\big)
\leq
\mathbb P\left(
T_{W_\infty,\phi}
\big(\widetilde{\mathcal L}_n(\mathbf X)\big)\in F^\delta
\right).
\]

Since $J_{W_\infty}$ is a good rate function, as $\delta\downarrow0$,
\[
\inf_{{F^\delta}}J_{W_\infty}
\uparrow
\inf_{\overline F}J_{W_\infty},
\qquad
\inf_{F_{-\delta}}J_{W_\infty}
\downarrow
\inf_{F^\circ}J_{W_\infty}.
\]
Thus, we may choose $\delta>0$ sufficiently small that
\[
\inf_{{F^\delta}}J_{W_\infty}
\geq
\inf_{\overline F}J_{W_\infty}-\frac{\varepsilon}{2},
\qquad
\inf_{F_{-\delta}}J_{W_\infty}
\leq
\inf_{F^\circ}J_{W_\infty}+\frac{\varepsilon}{2}.
\]

Applying the LDP lower bound to the open set $F_{-\delta}$ and the
upper bound to the closed set $F^\delta$, and combining
these bounds with the preceding inequalities, we obtain, for all
sufficiently large $n$, on the event
$\{C_{\phi,v}\|W_{Q_n}-W_\infty\|_{\square^*}\leq\delta/2\}$,
\[
-\inf_{t\in F^\circ}J_{W_\infty}(t)-\varepsilon
\leq
\frac1n\log
\mathbb P\big(
N_n(Q_n,\phi,\mathbf X)\in F\mid Q_n
\big)
\leq
-\inf_{t\in\overline F}J_{W_\infty}(t)+\varepsilon.
\]
Finally, since
$\|W_{Q_n}-W_\infty\|_{\square^*}\xrightarrow{\mathbb P}0$, we obtain
$\mathbb P\left(
C_{\phi,v}\|W_{Q_n}-W_\infty\|_{\square^*}\leq\frac{\delta}{2}
\right)\longrightarrow1$,
which proves \eqref{conldp}.
\end{proof}
\subsection{Proofs for \cref{ER Hypergraph}}
Throughout this subsection, we will work in the setting of \cref{ER Hypergraph}. In particular, we will use the notation $\mathcal{G}_n^{(v)}$ defined in \eqref{DEF:hyptensor}.
\begin{proof}[Proof of \cref{LDPforERHyp}]
We need to show
\[
\|W_{\mathcal{G}_n^{(v)}}-\mathbbm{1}\|_{\square^*}=o_\P(1),
\]
after which \cref{condldpdef} and \cref{NewLDP} complete the proof. By \cref{DtC}, it is enough to show
\[
\frac{1}{n^v}\sup_{S\subseteq[n]}
\left|
\sum_{\substack{i_1,\ldots,i_v\in S\\ \text{all distinct}}}
\big(\mathcal G_n^{(v)}(i_1,\ldots,i_v)-1\big)
\right|=o_\P(1).
\]
For fixed $S\subseteq[n]$, the sum in the above display can be written as
$v!\sum_{e\in\binom{S}{v}}\left(\frac{Y_e}{p_n}-1\right)$,
where $\{Y_e:e\in\binom{[n]}{v}\}$ are independent $\operatorname{Ber}(p_n)$ random variables. Since
$\left|Y_e/p_n-1\right|\leq p_n^{-1}$ and
$\operatorname{Var}(Y_e/p_n-1)\leq p_n^{-1}$, Bernstein's inequality gives, for every $\varepsilon>0$,
\[
\P\Bigg(
\Big|
\sum_{\substack{i_1,\ldots,i_v\in S\\ \text{all distinct}}}
\big(\mathcal G_n^{(v)}(i_1,\ldots,i_v)-1\big)
\Big|>\varepsilon n^v
\Bigg)
\leq 2\exp(-C_{\varepsilon,v}n^vp_n),
\]
uniformly over $S\subseteq[n]$. Taking a union bound over the at most $2^n$ choices of $S$ gives
\[
\P\left(
\frac{1}{n^v}\sup_{S\subseteq[n]}
\left|
\sum_{\substack{i_1,\ldots,i_v\in S\\ \text{all distinct}}}
\big(\mathcal G_n^{(v)}(i_1,\ldots,i_v)-1\big)
\right|>\varepsilon
\right)
\leq 2^{n+1}e^{-C_{\varepsilon,v}n^vp_n}=o(1),
\]
where the last equality follows from $n^{1-v}\ll p_n$.
\end{proof}

Now we state a lemma proven in \cref{AuxiliarySection}.
\begin{lmm}\label{DtC}
Suppose $D_n$ is a zero-diagonal real-valued function on $[n]^v$. Then
\begin{align}
 &\label{DtC1}\sup_{\{{\tau_i}^\ell\in [0,1] :i\in [n], \ell \in [v]\}}\left| \sum_{(i_1,\ldots,i_v)\in [n]^v}D_n(i_1,\ldots,i_v)\prod_{r=1}^v \tau^r_{i_r}\right|=\sup_{\{{\beta_i}^\ell\in \{0,1\} :i\in [n], \ell \in [v]\}}\left|\sum_{(i_1,\ldots,i_v)\in [n]^v}D_n(i_1,\ldots,i_v) \prod_{r=1}^v\beta^r_{i_r} \right|,\\
&\notag\text{and}\\
&\label{DtC2}\underset{\substack{\tau_1,\ldots,\tau_n \\ \in [0,1]}}{\sup}\left| \sum_{(i_1,\ldots,i_v)\in [n]^v}D_n(i_1,\ldots,i_v)\prod_{r=1}^v \tau_{i_r}\right|=\underset{\substack{\beta_1,\ldots,\beta_n \\ \in \{0,1\}}}{\sup} \left|\sum_{(i_1,\ldots,i_v)\in [n]^v}D_n(i_1,\ldots,i_v) \prod_{r=1}^v\beta_{i_r} \right|.
\end{align}

\end{lmm}
% \begin{proof}[Proof of \cref{DtC}]
%   We only prove the first equality, the latter can be shown similarly.  
% Fix $\tau^{\ell}_1, \ldots,\tau^{\ell}_n \in [0,1]$, for all $\ell \in [v]$. Define independent random variables $\zeta^{\ell}_i\sim \mathrm{Ber}(\tau^{\ell}_i)$ for $i \in [n], \ell \in [v]$. If $\mathbb{E}$ denotes the expectation w.r.t random variables $\{\zeta^{\ell}_1\ldots\zeta^{\ell}_n: \ell\in[v]\}$, one could write 
% \begin{align*}
% & \left|\sum_{(i_1,\ldots,i_v)\in [n]^v}D_n(i_1,\ldots,i_v) \prod_{r=1}^v \tau^{r}_{i_r} \right|=\left|\E\Bigg[\sum_{(i_1,\ldots,i_v)\in[n]^v}D_n(i_1,\ldots,i_v) \prod_{r=1}^v\zeta^{r}_{i_r} \Bigg]\right|\\
% &\le \E\left|\Bigg[\sum_{(i_1,\ldots,i_v)\in [n]^v}D_n(i_1,\ldots,i_v) \prod_{r=1}^v\zeta^{r}_{i_r} \Bigg]\right|\le\max_
% {\ell\in[v]}\sup_{\beta^{\ell}_1,\ldots,\beta^{\ell}_n\in \{0,1\}} \left|\sum_{(i_1,\ldots,i_v)\in [n]^v}D_n(i_1,\ldots,i_v) \prod_{r=1}^v\beta^r_{i_r} \right|.
% \end{align*}
% These inequalities are valid for all choices of $\tau^{\ell}_1, \ldots,\tau^{\ell}_n \in [0,1]$, so we have 
% \begin{align*}
% &\max_
% {\ell\in[v]}\sup_{\tau^\ell_1,\ldots,\tau^\ell_n \in [0,1]}\left| \sum_{(i_1,\ldots,i_v)\in [n]^v}D_n(i_1,\ldots,i_v)\prod_{r=1}^v \tau^r_{i_r}\right|\le\max_
% {\ell\in[v]}\sup_{\beta^{\ell}_1,\ldots,\beta^{\ell}_n\in \{0,1\}} \left|\sum_{(i_1,\ldots,i_v)\in [n]^v}D_n(i_1,\ldots,i_v) \prod_{r=1}^v\beta^r_{i_r} \right|.   
% \end{align*}
% The reverse inequality is obvious, so we are done.
% \end{proof}

\subsection{Proofs for \cref{SIERgraph}}
Throughout this subsection, we will work in the setting of \cref{SIERgraph}. In particular, we will use the notations $\mathcal{G}_n, {\rm Sym}[\mathcal{G}_n], H, \widetilde{\Delta}$ from \cref{symg}, and \cref{def:notation}.

    \begin{proof}[Proof of \cref{LDPforINSER}]
    We need to show 
\begin{align*}
 \|W_{\mathrm{Sym}[\mathcal{G}_n]}-\mathbbm{1}\|_{\square^*}=o_\P(1).
\end{align*}
This implies, by \cref{condldpdef}, the conditional LDP in part (i). The convergence of the scaled log-partition function in part (ii) follows from \cref{NewLDP} (ii) applied conditionally, together with the bound above. 
%then using \cref{condldpdef} and applying \cref{NewLDP} complete the proof. 
We prove the above display in the following proposition.
\end{proof}
       \begin{prop}[$\,\square^*$-Convergence of Sparse Erd\H{o}s-R\'enyi Graphs]\label{graphthm2}
Suppose $H$ is not a tree. Also assume that $(n\log n)^{-1/ {\widetilde \Delta}}\ll p_n \ll1 $. Then
\[\|W_{\mathrm{Sym}[\mathcal{G}_n]}- \mathbbm{1}\|_{\square^*}=o_\P(1).\] 

\end{prop}
% \begin{prop}[Norm Control for Sparse Erd\H{o}s-R\'enyi Graphs]\label{normcinser}In the setting of \cref{SIERgraph}, where $\mathrm{Sym}[\mathcal{G}_n]$ is defined in \cref{symg}, one has $\|W_{\mathrm{Sym}[\mathcal{G}_n]}\|_1 = O_{\mathbb P}(1).$
% \end{prop}
Before proving \cref{graphthm2}, we need the following lemmas proven in \cref{AuxiliarySection}. For \(S\subseteq[n]\) and a graph $F$ define
\begin{align}\label{defon}
\mathcal N_F(S;\mathcal G_n)
:=
\sum_{\substack{(i_x)_{x\in V(F)}\in S^{|V(F)|}\\
i_x\text{ all distinct}}}
\prod_{\{x,y\}\in E(F)}
\mathcal G_n(i_x,i_y)=\sum_{\substack{\varphi:V(F)\hookrightarrow S}}
\prod_{\{a,b\}\in E(F)}
\mathcal G_n(\phi(a),\phi(b)),
\end{align}

% Notice that $\mathcal N_F(S;\mathcal G_n)
% =
% \sum_{\substack{\varphi:V(F)\hookrightarrow S}}
% \prod_{\{a,b\}\in E(F)}
% \mathcal G_n(\phi(a),\phi(b)),$
where \(\varphi:V(F)\hookrightarrow S\) denotes an injective map from \(V(F)\) into
\(S\). 
\begin{lmm}[Uniform reduction to the \(2\)-core]
\label{lem:uniform-core-reduction-ER}
Assume \(H\) is not a tree. Assume also that
$(n\log n)^{-1/\widetilde\Delta}\ll p_n\ll1$. \
Then, with $\widetilde{H}$ denoting the core of $H$ (see \cref{corH}), we have
\[
\sup_{S\subseteq[n]}
\frac{1}{n^v}
\left|
p_n^{-e}\mathcal N_H(S;\mathcal G_n)
-
(|S|-\widetilde v)_\omega
p_n^{-\widetilde e}\mathcal N_{\widetilde H}(S;\mathcal G_n)
\right|
=o_{\mathbb P}(1),
\]
where we use the notation
$(m)_d:=m(m-1)\cdots(m-d+1)$,
and $(m)_0:=1$.
\end{lmm}
% \begin{lmm}[Upper tails control]
% \label{lem:hom-to-inj-upper-tail}
% Let \(F\) be a fixed graph with \(\widetilde v\) vertices, \(\widetilde e\) edges, and maximum degree
% \(\widetilde\Delta\). Assume that
% \[
% (n\log n)^{-1/\widetilde\Delta}\ll p_n\ll1.
% \]
% Let \(S\subseteq[n]\) with \(|S|\ge \rho n\), where \(\rho>0\) is fixed. Then, for
% every \(\varepsilon>0\), there exists \(c_{\varepsilon,\rho,F}>0\) such that
% \[
% \mathbb P\left(
% \mathcal N_F(S;\mathcal G_n)
% >
% (1+\varepsilon)p_n^{\widetilde e}(|S|)_{\widetilde v}
% \right)
% \le
% \exp\left\{
% -c_{\varepsilon,\rho,F}|S|^2p_n^{\widetilde\Delta}\log(1/p_n)
% \right\}.
% \]
% \end{lmm}
\begin{lmm}[Uniform discrepancy for the \(2\)-core]\label{lem:uniform-core-discrepancy-ER}
Assume $H$ is not a tree, and
$(n\log n)^{-1/\widetilde\Delta}\ll p_n\ll1$.
Then
\[
\sup_{S\subseteq[n]}
n^{-\widetilde v}
\left|
p_n^{-\widetilde e}\mathcal N_{\widetilde H}(S;\mathcal G_n)
-
(|S|)_{\widetilde v}
\right|
=o_{\mathbb P}(1).
\]
\end{lmm}

\begin{proof}[Proof of \cref{graphthm2}]
Fix a measurable set \(T\subseteq[0,1]\), and define
\[
I_{T,i}:=\left[\frac{i-1}{n},\frac{i}{n}\right)\cap T,
\qquad i\in[n].
\]
By the definition of \(\mathrm{Sym}[\mathcal G_n]\) and $W_{\mathrm{Sym}[\mathcal G_n]}$ (see \cref{symg}),
\begin{align*}
&\int_{ T^v}
\left(W_{\mathrm{Sym}[\mathcal G_n]}(x_1,\ldots,x_v)-1\right)
\prod_{r=1}^v dx_r \\
&=
\sum_{\substack{(i_1,\ldots,i_v)\in[n]^v\\ \mathrm{distinct}}}
\left[
\frac1{v!p_n^e}\sum_{\sigma\in S_v}
\prod_{\{a,b\}\in E(H)}
\mathcal G_n(i_{\sigma(a)},i_{\sigma(b)})
-1
\right]
\prod_{r=1}^v \lambda(I_{T,i_r})
-
\sum_{\substack{(i_1,\ldots,i_v)\in[n]^v\\ \mathrm{not\ distinct}}}
\prod_{r=1}^v \lambda(I_{T,i_r}).
\end{align*}
The last term is deterministic and negligible, since
\[
0\le
\sum_{\substack{(i_1,\ldots,i_v)\in[n]^v\\ \mathrm{not\ distinct}}}
\prod_{r=1}^v \lambda(I_{T,i_r})
\le C_H\frac{n^{v-1}}{n^v}
=o(1).
\]
Moreover, for every \(\sigma\in S_v\), the change of variables
\(j_r=i_{\sigma(r)}\) gives
\begin{align*}
&\sum_{\substack{(i_1,\ldots,i_v)\in[n]^v\\ \mathrm{distinct}}}
\prod_{\{a,b\}\in E(H)}
\mathcal G_n(i_{\sigma(a)},i_{\sigma(b)})
\prod_{r=1}^v \lambda(I_{T,i_r}) =
\sum_{\substack{(i_1,\ldots,i_v)\in[n]^v\\ \mathrm{distinct}}}
\prod_{\{a,b\}\in E(H)}
\mathcal G_n(i_a,i_b)
\prod_{r=1}^v \lambda(I_{T,i_r}).
\end{align*}
Therefore,
\begin{align*}
&\int_{ T^v}
\left(W_{\mathrm{Sym}[\mathcal G_n]}(x_1,\ldots,x_v)-1\right)
\prod_{r=1}^v dx_r =
\sum_{\substack{(i_1,\ldots,i_v)\in[n]^v\\ \mathrm{distinct}}}
\left[
p_n^{-e}
\prod_{\{a,b\}\in E(H)}
\mathcal G_n(i_a,i_b)
-1
\right]
\prod_{r=1}^v \lambda(I_{T,i_r})
+o(1),
\end{align*}
where the \(o(1)\) term is deterministic and uniform in \(T\).

Now set
\[
w_i:=n\lambda(I_{T,i}),\qquad i\in[n].
\]
Then \(0\le w_i\le1\), and
$\prod_{r=1}^v\lambda(I_{T,i_r})
=
n^{-v}\prod_{r=1}^v w_{i_r}$. Hence, by invoking \eqref{DtC2} in \cref{DtC} for \[D_n(i_1,\ldots,i_v)=\left[
p_n^{-e}
\prod_{\{a,b\}\in E(H)}
\mathcal G_n(i_a,i_b)
-1
\right]\mathbf 1_{\{\text{all}\,i_1,\ldots,i_v\,\text{are distinct}\}},\] it is enough to prove
\[
\sup_{S\subseteq[n]}
\frac1{n^v}
\left|
\sum_{\substack{i_1,\ldots,i_v\in S\\ \mathrm{distinct}}}
\left[
p_n^{-e}
\prod_{\{a,b\}\in E(H)}
\mathcal G_n(i_a,i_b)
-1
\right]
\right|
=o_{\mathbb P}(1).
\]
But the sum inside the absolute value is exactly
$p_n^{-e}\mathcal N_H(S;\mathcal G_n)-(|S|)_v$. Now by \cref{lem:uniform-core-reduction-ER},
\[
p_n^{-e}\mathcal N_H(S;\mathcal G_n)
=
(|S|-\widetilde v)_\omega
p_n^{-\widetilde e}\mathcal N_{\widetilde H}(S;\mathcal G_n)
+
o_{\mathbb P}(n^v),
\]
uniformly over \(S\subseteq[n]\).
Since
$\left|(|S|-\widetilde v)_\omega\right|
\leq C_H n^\omega$
uniformly over $S\subseteq[n]$, \cref{lem:uniform-core-discrepancy-ER} implies
\[
(|S|-\widetilde v)_\omega
p_n^{-\widetilde e}
\mathcal N_{\widetilde H}(S;\mathcal G_n)
=
(|S|-\widetilde v)_\omega
(|S|)_{\widetilde v}
+
o_\mathbb P(n^v),
\]
uniformly over $S\subseteq[n]$.
Finally,
$(|S|-\widetilde v)_\omega (|S|)_{\widetilde v}
=
(|S|)_v$.
Therefore
\[
\sup_{S\subseteq[n]}
n^{-v}
\left|
p_n^{-e}\mathcal N_H(S;\mathcal G_n)
-
(|S|)_v
\right|
=o_{\mathbb P}(1),
\]
and the proof is complete.
\end{proof}
\subsection{Proofs for \cref{GPM}}

Throughout this subsection, we will work in the setting of \cref{GPM}. In particular, we will use the notations $N^{(r)}_n(Q_n,{\bf X}), \boldsymbol{N}_n(Q_n,{\bf X})$, $\mathcal{R}_{Q_n,{ {\boldsymbol{\alpha}}},\boldsymbol h}, \mathcal{Z}_{Q_n}(\boldsymbol{\alpha},\boldsymbol{h})$ from \eqref{defNinGPT}, \eqref{defNinGPT2}, \eqref{eq:gibbs_potts2} and \eqref{logpf}.

\begin{proof}[Proof of \cref{prop:ldp+zn}]

    Parts (i) and (ii) follow by invoking \cref{NewLDP} parts (i) and (ii), respectively, with
    $$\phi(x_1,\ldots,x_v)=\sum_{r=1}^c \alpha_r \mathbf{1}_{\{x_1=\ldots=x_v=r\}},\quad \beta=1,\quad \mu_r=\frac{e^{h_r}}{\sum_{s=1}^c e^{h_s}}.$$
     
    \begin{itemize}
        \item [(iii)]
     Using the definition of $\tilde{\mathcal{L}}_n(\bf X)$ in the proof of  \cref{NewLDP}, we have 

\[
{\boldsymbol m}_n({\bf X})=\Big(\int_0^1 f^{\tilde{\mathcal{L}}_n(\bf X)}_1(u)\,du,\ldots,\int_0^1 f^{\tilde{\mathcal{L}}_n(\bf X)}_c(u)\,du\Big).
\]

By \cref{bijection}, we may work on $\mathcal P$ instead of $\mathcal F_c$. Defining
\[
\Psi(\nu)
:=
\sum_{r=1}^c{\alpha}_r T_{W_\infty,\mathbf{1}_r}(\nu)
+\sum_{r=1}^c h_r \nu(B=r)
-\int_0^1 \sum_{r=1}^c f_r^\nu(u)\log f_r^\nu(u)\,du,
\]
we get \[\Psi(\nu)=\sum_{r=1}^c \alpha_r { G}_{W_\infty,\mathbf 1_r}({\boldsymbol f}^\nu)
+\sum_{r=1}^c h_r\int_0^1 f^\nu_r(u)\,du
-\int_0^1\sum_{r=1}^c f^{\nu}_r(u)\log f^{\nu}_r(u)\,du.\] So $\boldsymbol{f}^\nu$ maximizes \eqref{eq:potts_unconstrained_opt}, if and only if $\nu$ maximizes $\Psi(\cdot)$. Now fix $\varepsilon>0$ and define the set
\[
\mathcal P_\varepsilon
:=
\Big\{\nu\in\mathcal P:\ 
\operatorname{dist}\Big(
\Big(\int_0^1 f^\nu_1(u)\,du,\ldots,\int_0^1 f^\nu_c(u)\,du\Big),\mathcal S
\Big)\ge\varepsilon
\Big\}.
\]
Since the map $\nu\mapsto \int_0^1 f_r^\nu(u)du=\nu(B=r)$ on $\mathcal{P}$ is continuous w.r.t. weak topology, it follows that the set $\mathcal{P}_\varepsilon$ is weakly closed.
Also, we have $\{\operatorname{dist}({\boldsymbol m}_n({\bf X}),\mathcal S)
>\varepsilon\}
\subseteq\{{\tilde{\mathcal{L}}_n(\bf X)}\in\mathcal P_\varepsilon\}
$. It is thus enough to show that 
\begin{align}
& \limsup_{n\to\infty}\frac1n\log
\mathcal{R}_{Q_n,{ {\boldsymbol\alpha}},\boldsymbol h}(\{\operatorname{dist}({\boldsymbol m}_n({\bf X}),\mathcal S)
>\varepsilon\})\le\limsup_{n\to\infty}\frac1n\log
\mathcal{R}_{Q_n,{ {\boldsymbol\alpha}},\boldsymbol h}({\tilde{\mathcal{L}}_n(\bf X)}\in\mathcal P_\varepsilon)<0\nonumber . 
\end{align}

Using the definition of the Gibbs measure~\eqref{eq:gibbs_potts2}, we get
\[
{\mathcal{R}}_{Q_n,{ {\boldsymbol\alpha}},\boldsymbol h}({\tilde{\mathcal{L}}_n(\bf X)}\in\mathcal P_\varepsilon)
=
\frac{
\mathbb E_{\mu^{\otimes n}}\left[
\exp\Big(n\boldsymbol\alpha^\top \boldsymbol{N}_n(Q_n,{\bf X})\Big)\,
\mathbf 1_{\{{\tilde{\mathcal{L}}_n(\bf X)}\in\mathcal P_\varepsilon\}}
\right]
}{
\mathbb E_{\mu^{\otimes n}}\left[
\exp\Big(n\boldsymbol\alpha^\top \boldsymbol{N}_n(Q_n,{\bf X})\Big)
\right]
}.
\]

Now,
\begin{align*}
& \lim_{n\to\infty}\frac1n\log\mathbb E_{\mu^{\otimes n}}\left[
\exp\Big(n\boldsymbol\alpha^\top \boldsymbol{N}_n(Q_n,{\bf X})\Big)
\right]= \sup_{{\nu}\in \mathcal{P}}  \Psi(\nu)-\log\sum_{s=1}^c e^{h_s},
\end{align*}
where in the last line we used part (ii).

Applying \cite[Exercise 4.3.11]{DZ} on the closed set $\mathcal P_\varepsilon$ yields
\begin{align*}
& \limsup_{n\to\infty}\frac1n\log
\mathbb E_{\mu^{\otimes n}}
\bigg(\exp\Big(n\boldsymbol\alpha^\top \boldsymbol{N}_n(Q_n,{\bf X})\Big)
\mathbf 1_{\{{\tilde{\mathcal{L}}_n(\bf X)}\in\mathcal P_\varepsilon\}}\bigg)
\le \sup_{\nu\in\mathcal P_\varepsilon} \Psi(\nu)-\log\sum_{s=1}^c e^{h_s}.
\end{align*}

Therefore,
\begin{align}\label{LLNeq}
\limsup_{n\to\infty}\frac1n\log
\mathcal{R}_{Q_n,{ {\boldsymbol\alpha}},\boldsymbol h}(\tilde{\mathcal{L}}_n({\bf X})\in\mathcal P_\varepsilon)
\le
\sup_{\nu\in\mathcal P_\varepsilon} \Psi(\nu)-\sup_{\nu\in\mathcal P} \Psi(\nu) .   
\end{align}

The functional $\Psi$ is upper semi-continuous on $\mathcal P$, since the negative
entropy term is upper semi-continuous and the rest is continuous. The second supremum in the above display is achieved by part (ii).
Moreover, since $\mathcal P_\varepsilon$ is a closed subset of the
compact space $\mathcal P$, it is compact w.r.t. weak topology, and the first supremum is
also attained. Now by definition of $\mathcal{P_\varepsilon}$, we have $\mathcal P_\varepsilon\cap\{\nu\in \mathcal{P}: \boldsymbol{f}^\nu \in \mathcal{A}\}=\varnothing$, therefore,
\[
\sup_{\nu\in\mathcal P_\varepsilon} \Psi(\nu)<\sup_{\nu\in\mathcal P} \Psi(\nu).
\]
So the right-hand side (RHS) in \eqref{LLNeq} is strictly negative, which completes the proof.
%If $\mathcal A=\{{\boldsymbol f}^*\}$ is a singleton, then $\mathcal S$ is a singleton as well,
%and the convergence reduces to the stated law of large numbers.
\end{itemize}

\end{proof}
\begin{proof}[Proof of \cref{thisgivesthat}]
The LDP follows on invoking \cref{prop:ldp+zn} with the choice $Q_n$ as in \eqref{eq:333}. It suffices to verify~\eqref{eq:cut}, which follows from \cite[Proposition 3.1, part (ii)]{bhattacharya2024ldp}.

%show that the assumptions in \eqref{eq:twomatbd} implies the assumption of \cref{prop:ldp+zn}, namely \eqref{eq:cut}. The former claim follows from \cite[Proposition 3.1, part (ii)]{bhattacharya2024ldp} while the latter claim follows from \cite[Proposition 3.1, part (i)]{bhattacharya2024ldp} by choosing the functions $f_a(x)=1$, $1\le a\le v$ in \cite[Proposition 3.1]{bhattacharya2024ldp}.
\end{proof}

\begin{proof}[Proof of \cref{prop:propotts}]
			Notice that (i) and (ii) are simply implied by \cref{prop:propotts2} (i) and (ii), with
    $$\phi(x_1,\ldots,x_v)=\sum_{r=1}^c \alpha_r \mathbf{1}_{r},\quad \beta=1,\quad \mu_r=\frac{e^{h_r}}{\sum_{s=1}^c e^{h_s}}.$$
    In this case, 
    
    \begin{align*}
& \mathcal T_r[W,\phi,{\boldsymbol f}](x)
=
\sum_{m=1}^v \mathcal T_r^{(m)}[W,\phi,{\boldsymbol f}](x)=
\sum_{m=1}^v \sum_{s=1}^c \alpha_s\mathcal T_r^{(m)}[W,  \mathbf{1}_{s},{\boldsymbol f}](x) \\
&=
\sum_{m=1}^v \alpha_r\mathcal T_r^{(m)}[W,  \mathbf{1}_{r},{\boldsymbol f}](x)=v \alpha_r\mathcal{V}_r[W,{\boldsymbol{f} }](x).
    \end{align*}
 Notice that one can check $ T_r^{(m)}[W,  \mathbf{1}_{s},{\boldsymbol f}](x)=\mathcal{V}_r[W,{\boldsymbol{f} }](x) \mathbf{1}_{\{r=s\}}$ for all $m\in [v]$.
\begin{itemize}
			\item[(iii)] First assume $\alpha_r>0$ for all $r\in[c]$. Since $\mathcal V[W_\infty](x)=1$ for $\lambda$-a.e. $x\in[0,1]$, $W_\infty(x_1, \ldots, x_v)$ is a probability density function on $[0,1]^v$ with all marginals being uniformly distributed on $[0,1]$. 
            Define 
            \begin{align*}\label{eq:delta}
                \Delta_{c-1}:=\{(y_1,\ldots ,y_c)\in [0,1]^c:\ \sum_{r=1}^c y_r=1\}.
                \end{align*}
                By an application of H\"older's inequality (see \cite[Theorem 8.1]{BorgsLPII}), using the nonnegativity of $\boldsymbol{\alpha}$ we  have
			\begin{align*}
			    {\boldsymbol{\alpha}}^\top {\boldsymbol G}_{W_\infty}({\boldsymbol f})=\sum_{r=1}^c \alpha_r \E_{(Z_1,\ldots ,Z_v)\sim W_\infty}\bigg[\prod_{i=1}^v f_r(Z_i)\bigg]\leq \sum_{r=1}^c \int_{[0,1]} \alpha_r f_r^v(x)\,dx
			\end{align*}
			for all ${\boldsymbol f}\in\mathcal{F}_c$. Note that the equality condition in H\"older's inequality implies that the equality holds in the above display if and only if $f_r(\cdot)$'s are all $\lambda$-a.e. constant functions.

            Consequently
			\begin{align*}
				&\sup_{{\boldsymbol f}\in \mathcal{F}_c}\left\{{\boldsymbol\alpha}^\top {\boldsymbol G}_{W_\infty}({\boldsymbol f})
+\sum_{r=1}^c h_r\int_0^1 f_r(u)\,du
-\int_0^1\sum_{r=1}^c f_r(u)\log f_r(u)\,du\right\}\\
&\le \sup_{{\boldsymbol f}\in\mathcal{F}_c}\left\{\sum_{r=1}^c \int_0^1 \left(\alpha_r f_r^v(x) +h_r f_r(x)-f_r(x)\log{f_r(x)}\right)\,dx\right\}\\
&\leq \sup_{(y_1,\ldots, y_c)\in \Delta_{c-1}} \sum_{r=1}^c\left(\alpha_r  y_r^v +h_r y_r- y_r \log y_r \right).
			\end{align*}

			The equality condition in H\"older's inequality implies that the equality holds in the above display if and only if $f_r(\cdot)$'s are all $\lambda$-a.e. constant functions.
            
            It remains to consider the case where $\alpha_r=0$ for some
$r\in[c]$. For any $r \in [c]$ where $\alpha_r= 0$, the fixed-point equation \eqref{eq:propottshow} gives that $f_r(x)$ does not depend on $x$, so it is a constant function. Consequently, all
coordinates of $\boldsymbol f$ are $\lambda$-a.e. constant.
\end{itemize}

   \end{proof}
   \begin{proof}[Proof of~\cref{prop:potts_gibbs_ldp}]
By \cref{prop:propotts} (iii), under the assumptions
$W_{\infty}>0$ $\lambda$-a.e., $\mathcal V[W_{\infty}](x)=1$ $\lambda$-a.e.,
and $\alpha_r=\theta\ge 0$, $h_r=0$ for all $r\in[c]$, every maximizer
of~\eqref{eq:potts_unconstrained_opt} is constant $\lambda$-a.e. Therefore,
writing such a maximizer as $f_r(x)=x_r$ for $r\in[c]$,
the optimization problem~\eqref{eq:potts_unconstrained_opt} reduces to
\[
\sup_{\boldsymbol x\in \Delta_{c-1}}
\left\{
\theta\sum_{r=1}^c x_r^v
-
\sum_{r=1}^c x_r\log x_r
\right\}.
\]

This is precisely the finite-dimensional optimization problem for the
$v$-tensor Curie--Weiss Potts model with $c$ colors and no external field,
with the correspondence
\[
p=v,\qquad q=c,\qquad \beta=\theta,\qquad h=0,
\]
studied in \cite{bhowal2025limit}.
By ~\cite[Proposition F.1.]{bhowal2025limit}, every global
maximizer of this finite-dimensional problem is, up to permutation, of the
form
\[
\left(
\frac{1+(c-1)y^*}{c},
\frac{1-y^*}{c},
\ldots,
\frac{1-y^*}{c}
\right)
\]
for some $0\le y^*<1$. Hence, any optimizer whose first coordinate is maximal,
that is, $x_1\ge \max_{2\le r\le c}x_r$, has the form \eqref{eq:potts_one_large}. This proves part (i).

It remains to prove part~(ii). By the phase-transition analysis in ~\cite[Section 4]{bhowal2025limit}, for the above
finite-dimensional Curie--Weiss Potts optimization problem with $h=0$,
there exists a critical threshold $\theta_{\rm crit}\in(0,\infty)$ such that, for
$\theta<\theta_{\rm crit}$, the uniform vector
$\left(\frac1c,\ldots,\frac1c\right)$
is the unique global maximizer, while for $\theta>\theta_{\rm crit}$ the global
maximizers are exactly the $c$ non-uniform permutations of the symmetry-broken
maximizer. Consequently, if
$(x_{\theta,1},\ldots,x_{\theta,c})$
denotes the unique maximizer satisfying
$x_{\theta,1}>\max_{2\le r\le c}x_{\theta,r}$, then the implicit function theorem implies
that the map
\[
\theta\mapsto (x_{\theta,1},\ldots,x_{\theta,c})
\]
is continuous on $(\theta_{\rm crit},\infty)$. This proves part~(ii).
\end{proof}
   \begin{proof}[Proof of~\cref{prop:potts_diffalpha}]
			
			 First notice that, by applying \cref{prop:propotts} part (iii), all maximizers
are constant functions. Therefore, the optimization problem \eqref{eq:potts_unconstrained_opt} reduces to maximizing
\begin{align}\label{defofH}
\mathscr H({\boldsymbol z})
:=
\sum_{r=1}^c\alpha_r z_r^v
+\sum_{r=1}^ch_rz_r
-\sum_{r=1}^cz_r\log z_r,
\qquad
{\boldsymbol z}\in\Delta_{c-1}.    
\end{align}

Moreover, in parts (i) and (ii), it is enough to establish uniqueness of the
maximizer. Indeed, the continuity of the maximizer map
follows from the continuity of \(\mathscr H\), the compactness of
\(\Delta_{c-1}\), and the uniqueness of the maximizer. 
            % First notice that by applying \cref{prop:propotts} part (iii), all maximizers are constant functions. Moreover, in both parts (i) and (ii), it is enough to establish uniqueness of the maximizer. Indeed, the continuity of the maximizer map then follows from the continuity of the objective function together with the compactness of the set $\Delta_{c-1}$ and the uniqueness of the maximizer.

Fix a maximizer ${\boldsymbol z}=(z_1,\ldots,z_c)\in \Delta_{c-1}$.  Then by \cref{prop:propotts} (i),\begin{equation}\label{2}
					z_r= \frac{\exp\left(\alpha_r v z_r^{v-1}+h_r\right)}{\sum_{s=1}^{c}\exp \left(\alpha_s  v z_s^{v-1}+h_s\right)},\ r\in [c].
				\end{equation}
                To prove uniqueness of the optimizer in parts (i) and (ii), it is enough to show that  $(z_2,\ldots,z_c)\in\left(0,\exp\left(-\tilde{\eta}\right)\right]^{c-1}$ for some $\tilde{\eta}>0$ and the function 
                \begin{equation}\label{simpleoptfunc}
                    \mathcal{H}(p_2,\ldots,p_c):=\mathscr H(1- \sum\limits_{r=2}^c p_r,p_2,\ldots,p_c)
                \end{equation} is strictly concave on  $\left(0,\exp\left(-\tilde{\eta}\right)\right]^{c-1}$.
				Notice that we have
			\begin{align}\notag\frac{\partial^2\mathcal{H}}{\partial p^2_r}(p_2,\ldots,p_c)=&  v(v-1)(\alpha_r p^{v-2}_r +\alpha_1 p^{v-2}_1) -\frac{1}{p_1}-\frac{1}{p_r}, \quad r \geq 2,\\
            \text{and} \quad
		\Big|\frac{\partial^2\mathcal{H}}{\partial p_r\partial p_s}(p_2,\ldots,p_c)\Big|=&\Big|\alpha_1 v(v-1)p^{v-2}_1-\frac{1}{p_1}\Big| \quad r \neq s, r,s\geq 2.\label{eq:hhh}
		\end{align}
        In both parts we show
        \begin{align}\label{hessiancondition}
       \frac{\partial^2\mathcal H}{\partial p_r^2}
+
\sum_{s\ne r}
\left|
\frac{\partial^2\mathcal H}{\partial p_r\partial p_s}
\right|
<0,
\qquad r=2,\ldots,c,     
        \end{align}
on $\left(0,\exp\left(-\tilde{\eta}\right)\right]^{c-1}$. Thus the Hessian of \(\mathcal H\) is strictly diagonally dominant with negative diagonal entries, and hence negative definite. Therefore \(\mathcal H\) is strictly concave on $\left(0,\exp\left(-\tilde{\eta}\right)\right]^{c-1}$, which gives uniqueness of the optimizer.
				  \begin{itemize}
        \item[(i)] 
				Suppose $\alpha_1 \geq \eta+\max_{ 2 \le r \le c} \alpha_r$ for some large $\eta>0$, to be specified later. First, note that if $\eta>\eta_1:=c^{v-1}\tilde{h}$ where $\tilde{h}=\max_{1\le r\le c} h_r-h_1$, then $z_1 \ge \max_{ 2 \le r \le c}z_r$. We prove the claim by contradiction. Suppose $\exists r \ge 2$ such that $z_r=\max_{s\in [c]} z_s > z_1$. In particular, $z_r \geq {c}^{-1}$. Now we compare the value of \eqref{defofH} for $\tilde{{\boldsymbol z}}=(z_r, z_2,\ldots, z_{r-1}, z_1, z_{r+1},\ldots, z_c)$ and ${\boldsymbol z}=(z_1,\ldots,z_c)$. 

                The optimality of ${\boldsymbol z}$ indicates that we have $\mathscr{H}({\boldsymbol z})\geq\mathscr{H}(\tilde{{\boldsymbol z}}).$ Since
\begin{align*}
    & \mathscr{H}({\boldsymbol z})-\mathscr{H}(\tilde{{\boldsymbol z}})=(\alpha_r-\alpha_1) (z^v_r-z^v_1)+ (h_r - h_1)(z_r - z_1),
\end{align*}
the condition $\mathscr{H}({\boldsymbol z})\geq\mathscr{H}(\tilde{{\boldsymbol z}})$ is equivalent to 
    $$\alpha_1 \leq\alpha_r + (h_r - h_1)\frac{(z_r-z_1)}{(z^v_r-z^v_1)}.$$
Also $\alpha_1\geq \eta + \alpha_r,$ so 
  \begin{align}\label{??}
      c^{v-1}\tilde{h}=\eta_1<\eta \leq \alpha_1-\alpha_r \leq (h_r - h_1)\frac{(z_r-z_1)}{(z^v_r-z^v_1)}.
  \end{align}
  However, $h_r - h_1\leq \tilde{h}$ and 
${(z_r-z_1)}/{(z^v_r-z^v_1)} \leq {z_r^{1-v}}\leq c^{v-1}$, gives us \[(h_r - h_1)\frac{(z_r-z_1)}{(z^v_r-z^v_1)}\le c^{v-1}\tilde{h}.\] This contradicts \eqref{??}. 

 Therefore, $z_1 \ge \max_{ 2 \le r \le c}z_r$ and hence $z_1 \geq c^{-1}$.
				Now, by \eqref{2} we get
				\begin{align*}
					\frac{z_r}{z_1} &= \exp \Big(  v (\alpha_r z_r^{v-1}-\alpha_1 z_1^{v-1})+(h_r - h_1)\Big) \\
                    & \le \exp (vz_1^{v-1}(\alpha_r-\alpha_1)+(h_r - h_1))\leq \exp\left( -\eta v c^{1-v}+\tilde{h}\right) .
				\end{align*}
                Set $\eta_2:={(\tilde{h}+\log(4c))c^{v-1}}/{v}$, and note that for $\eta>\eta_2$ we have
                $$-\eta v c^{1-v}+\tilde{h}<-\log(4c)\Rightarrow \frac{z_r}{z_1}\le \frac{1}{4c}.$$
        %$\tilde{\eta}:=\eta v (\frac{1}{c})^{v-1}-\tilde{h}$. If we take $\tilde{\eta}>0$, we get $\frac{z_r}{z_1}\leq \exp(-\tilde{\eta})<1$. So $z_r \leq \exp(-\tilde{\eta}) $ and $z_1>\max_{2\le r\le c} z_r$. Thus with $\eta_2:=\log(4c)$, for  $\tilde{\eta} >\eta_2$ we have $\max_{2\le r\le c}z_r\le (4c)^{-1}$.
       In particular, $z_1>\max_{2\le r\le c} z_r$. Moreover, this gives $z_r\le {(4c)}^{-1}$, and so $z_1\ge \frac{3}{4}$. Now, we can write for $r \in \{2,\ldots,c\}$.
                 
                \begin{align*}
           & \alpha_1 \geq \alpha_r + \eta \Rightarrow  \frac{\alpha_1}{2}(3c)^{v-1}>\alpha_r + \eta\Rightarrow \alpha_1(3c)^{v-1}>\alpha_r + \eta +\frac{\alpha_1}{2}(3c)^{v-1}\\
             & \Rightarrow \alpha_1 \left(\frac{3}{4}\right)^{v-1} > \alpha_r \left(\frac{1}{4c}\right)^{v-1}+ \eta \left(\frac{1}{4c}\right)^{v-1} + \frac{\alpha_1}{2}\left(\frac{3}{4}\right)^{v-1}\\
        & \Rightarrow \alpha_1 z_1^{v-1}> \alpha_r z_r^{v-1}+\eta \left(\frac{1}{4c}\right)^{v-1} + \frac{\alpha_1}{2}(\frac{3}{4})^{v-1}\\
            & \Rightarrow \alpha_1 z_1^{v-1}- \alpha_r z_r^{v-1}> \eta \left(\frac{1}{4c}\right)^{v-1} +\frac{\alpha_1}{2}\left(\frac{3}{4}\right)^{v-1}.
               \end{align*}
            
Therefore,
\begin{align*}
\frac{z_1}{z_r} 
= \exp\left(v(\alpha_1 z_1^{v-1} - \alpha_r z_r^{v-1}) + (h_1 - h_r)\right) 
> e^{\tilde{\eta}},\text{ where }\tilde{\eta}:=v\eta \left(\frac{1}{4c}\right)^{v-1} + \frac{v\alpha_1}{2} \left(\frac{3}{4}\right)^{v-1} - \tilde{h}.
\end{align*}
% \begin{align}\label{ratio1}
% = \exp\left(\tilde{\eta} \left(\frac{1}{4}\right)^{v-1} + v\frac{\alpha_1}{2} \left(\frac{3}{4}\right)^{v-1} - \Big(1-4^{1-v}\Big)\tilde{h}\right).
% \end{align}
Therefore,
\begin{align}\label{ratio1}
    z_r\le \exp(-\tilde{\eta}).
\end{align}
				Now we show that the function $\mathcal{H} $ defined in \eqref{simpleoptfunc} is strictly concave on $\left(0,\exp\left(-\tilde{\eta}\right)\right]^{c-1}$ for all $\eta>\eta_\circ$, where  $\eta_\circ=\eta_\circ(c,v,\tilde{h})$ will be specified later.
				To this end, for $\eta>\max(\eta_1,\eta_2)$ we have
			\begin{align*}\notag\frac{\partial^2\mathcal{H}}{\partial p^2_r}(p_2,\ldots,p_c)=&  v(v-1)(\alpha_r p^{v-2}_r +\alpha_1 p^{v-2}_1) -\frac{1}{p_1}-\frac{1}{p_r}\le 2v^2\alpha_1-\frac{1}{p_r}, \quad r \geq 2,\\
            \text{and} \quad
		\Big|\frac{\partial^2\mathcal{H}}{\partial p_r\partial p_s}(p_2,\ldots,p_c)\Big|=&\Big|\alpha_1 v(v-1)p^{v-2}_1-\frac{1}{p_1}\Big|\le \alpha_1  v^2+\frac{4}{3} \quad r \neq s, r,s\geq 2.
		\end{align*}
        Using the above bound along with \eqref{ratio1} we have
        \begin{align*}
           \frac{\partial^2\mathcal{H}}{\partial p_r^2}(p_2,\ldots,p_c)+ \sum_{s: s\ne r} \Big|\frac{\partial^2\mathcal{H}}{\partial p_r\partial p_s}(p_2,\ldots,p_c)\Big|\le &cv^2\alpha_1-\frac{1}{p_r}+\frac{4c}{3}\\
           \le &cv^2\alpha_1+\frac{4c}{3}-\exp\left(v\eta \left(\frac{1}{4c}\right)^{v-1} + \frac{v\alpha_1}{2} \left(\frac{3}{4}\right)^{v-1} - \tilde{h}\right)<0
        \end{align*}
        for all $\eta>\eta_3(v,c,\tilde{h})$. Thus setting $\eta_\circ:=\max(\eta_1,\eta_2,\eta_3)$, for $\eta>\eta_\circ$ the function $\mathcal{H}$ is strictly concave in the required domain, thus giving uniqueness.
        
				%If $v >2$, by choosing $p_i, i \ge 2$ small enough yields concavity. When $v=2$, 
		% 	Now using \eqref{ratio1}, one could check that if we take $\eta_3:=4^{v-1}\log(\kappa) $, where $$\kappa=\kappa(c,v,\tilde{h}):=\exp\left((1-4^{1-v})\tilde{h}\right)\times\sup_{\alpha_1>0}\frac{\big[(2+c)\alpha_1 v(v-1)+4c/3\big]}{\exp(\frac{1}{2}\alpha_1 v(\frac{3}{4})^{v-1})}\in (0,\infty),$$ we get $$\frac{\partial^2\mathcal{H}}{\partial p^2_r} <0 \quad\text{and}\quad \sum_{s\ne r}\Big|\frac{\partial^2\mathcal{H}}{\partial p_r\partial p_s}\Big|< \Big|\frac{\partial^2\mathcal{H}}{\partial p^2_r}\Big|$$ for all $r \in \{2,\ldots,c\}$. 
    
		% Thus the Hessian matrix of $\mathcal{H}$ is negative diagonally dominant, and hence negative definite, for all $\tilde{\eta}> \eta_3$. So we need to take $\tilde{\eta}>\max\{\eta_2, \eta_3\}$. Recalling $\tilde{\eta}=\eta v (\frac{1}{c})^{v-1}-\tilde{h}$ and $\eta_1=c^{v-1}\tilde{h}$, it is enough to take $\eta>\eta_\circ$, where 
  %       \begin{align*}
  %       &\eta_\circ:=c^{v-1} \max \{\tilde{h}, (\tilde{h}+\max\{\eta_2,\eta_3\}) /v\}\\
  %       &=c^{v-1} \max \left\{\tilde{h}, \frac{\tilde{h}+\max\{\log(4c),4^{v-1}\log\kappa(c,v,\tilde{h})\}} {v}\right\}.    
  %       \end{align*}
        %Finally, the continuity follows from implicit function theorem.
		%$\frac{\delta^2\mathcal{H}}{\delta p^2_r} \le \kappa -e^{\eta}$ and $\frac{\partial^2\mathcal{H}}{\partial p_r\partial p_s} \ge -2$ as $p_1 > \frac12$. Hence the Hessian matrix is negative definite for $\eta>\eta_3$. 

        \item [(ii)]
				Suppose $h_1 \geq \eta+\max_{ 2 \le r \le c} h_r$ for some large $\eta>0$, to be specified later. 
 First, note that if \[\eta>\eta_1:={v}\widetilde{\alpha}+\log(4c),\] where $\widetilde{\alpha}=\max_{1\le r\le c} \alpha_r-\alpha_1$, then $z_1 \ge \max_{ 2 \le r \le c}z_r$. As in the previous part, we prove the claim by contradiction. Suppose $\exists r \ge 2$ such that $z_r=\max_{s\in [c]} z_s >z_1$. In particular, $z_r \geq c^{-1}$. Optimality of ${\boldsymbol z}$ indicates that $\mathscr{H}({\boldsymbol z})\geq\mathscr{H}(\tilde{{\boldsymbol z}}),$ where $\mathscr{H}({\boldsymbol z})$ is defined in \eqref{defofH}. Notice that $\mathscr{H}({\boldsymbol z})\geq\mathscr{H}(\tilde{{\boldsymbol z}})$ is equivalent to 
    $$h_1 \leq h_r + (\alpha_r - \alpha_1)\frac{(z^v_r-z^v_1)}{(z_r-z_1)}.$$
  Also $h_1\geq \eta + h_r,$ so 

  \begin{align*}
v\widetilde{\alpha}+\log(4c)=\eta_1<\eta \leq h_1-h_r \leq (\alpha_r - \alpha_1)\frac{(z^v_r-z^v_1)}{(z_r-z_1)}.
\end{align*}

However, $(\alpha_r - \alpha_1)\leq \widetilde{\alpha}$ and 
$\frac{z^v_r-z^v_1}{z_r-z_1} \leq v$, which leads us to a contradiction. Therefore, $z_1 \ge \max_{ 2 \le r \le c}z_r$ and hence $z_1 \geq c^{-1}$. Again using \eqref{2} gives
				\begin{align*}
					&\frac{z_r}{z_1} = \exp \Big(  v (\alpha_r z_r^{v-1}-\alpha_1 z_1^{v-1})+(h_r - h_1)\Big) \\
                    &\le\exp \Big(  v (\alpha_r -\alpha_1) z_1^{v-1}+(h_r - h_1)\Big)  \leq \exp(\widetilde{\alpha} v -\eta)\le(4c)^{-1} .
				\end{align*}
				  %which is strictly less than $1$ 
                   %But this is a contradiction, and so $z_1>\max_{2\le r\le c} z_r$.
                  
                  % Define $\tilde{\eta}:= \eta-\widetilde{\alpha} v$, so we have $z_r \leq \exp(-\tilde{\eta})$ for all $r\neq 1$.
                Therefore, we get $\max_{2\le r\le c}z_r\le (4c)^{-1}$ and hence $z_1\ge {3}/{4}$.
Moreover, we have
\[
\begin{aligned}
\alpha_1 z_1^{v-1}-\alpha_r z_r^{v-1}
&\ge
\alpha_1\left(\frac34\right)^{v-1}
-\alpha_r\left(\frac1{4c}\right)^{v-1}  \ge
\alpha_1\left[
\left(\frac34\right)^{v-1}
-\left(\frac1{4c}\right)^{v-1}
\right]
-\widetilde{\alpha}\left(\frac1{4c}\right)^{v-1}.
\end{aligned}
\]
Therefore, by \eqref{2},
\[
\frac{z_r}{z_1}
=\exp \Big(  v (\alpha_r z_r^{v-1}-\alpha_1 z_1^{v-1})+(h_r - h_1)\Big)\le
\exp\left(
-v\alpha_1\left[
\left(\frac34\right)^{v-1}
-\left(\frac1{4c}\right)^{v-1}
\right]
+v\widetilde{\alpha}\left(\frac1{4c}\right)^{v-1} -\eta
\right).
\]
Since \(z_1\le1\), this gives for any $r \ge 2$,
\[
z_r
\le
\exp\left(
-\eta
-v\alpha_1\left[
\left(\frac34\right)^{v-1}
-\left(\frac1{4c}\right)^{v-1}
\right]
+v\widetilde{\alpha}\left(\frac1{4c}\right)^{v-1}
\right)= \exp\left(
-\tilde{\eta}
\right),
\]
where
\[\tilde{\eta}:=\eta
+vA_{c,v}\alpha_1
-v\widetilde{\alpha}\left(\frac1{4c}\right)^{v-1}\]
for $A_{c,v}:=
\left(\frac34\right)^{v-1}
-\left(\frac1{4c}\right)^{v-1}>0$.
%Set
%\[
%A_{c,v}:=
%\left(\frac34\right)^{v-1}
%-\left(\frac1{4c}\right)^{v-1}>0.
%\]
%Then
% \[
% z_r\le
% \exp\left(
% -\eta
% -vA_{c,v}\alpha_1
% +v\widetilde{\alpha}\left(\frac1{4c}\right)^{v-1}
% \right),
% \qquad r\ge2.
% \]
% Now we show that the function $\mathcal{H} $ defined in \eqref{simpleoptfunc} is strictly concave on $\left(0,\exp\left(-\tilde{\eta}\right)\right]^{c-1}$. For \(r\ge2\),
% \[
% \frac{\partial^2\mathcal H}{\partial p_r^2}
% =
% v(v-1)
% \left(
% \alpha_rp_r^{v-2}
% +\alpha_1p_1^{v-2}
% \right)
% -\frac1{p_r}
% -\frac1{p_1},
% \]
% and for \(r\ne s\),
% \[
% \left|
% \frac{\partial^2\mathcal H}{\partial p_r\partial p_s}
% \right|
% =
% \left|
% \alpha_1v(v-1)p_1^{v-2}
% -\frac1{p_1}
% \right|.
% \]
Since \(z_1\ge 3/4\) and \(\alpha_r\le \alpha_1+\widetilde{\alpha}\), we get
\[
\frac{\partial^2\mathcal H}{\partial p_r^2}(p_2,\ldots,p_c)
\le
v^2(2\alpha_1+\widetilde{\alpha})
-\frac1{p_r}, \qquad \left|
\frac{\partial^2\mathcal H}{\partial p_r\partial p_s}(p_2,\ldots,p_c)
\right|
\le
\alpha_1v^2+\frac43.
\]
% and
% \[
% \left|
% \frac{\partial^2\mathcal H}{\partial p_r\partial p_s}
% \right|
% \le
% \alpha_1v(v-1)+\frac43.
% \]
Therefore,
\[
\frac{\partial^2\mathcal H}{\partial p_r^2}(p_2,\ldots,p_c)
+
\sum_{s:s\ne r}
\left|
\frac{\partial^2\mathcal H}{\partial p_r\partial p_s}(p_2,\ldots,p_c)
\right|
\le
cv^2\alpha_1
+\widetilde{\alpha}v^2
+\frac{4c}{3}
-\frac1{p_r}.
\]
On the above box, $p_r^{-1}
\ge
\exp\left(
\tilde\eta
\right)$. Since the exponential term in \(\alpha_1\) dominates the linear term in \(\alpha_1\), there exists $\eta_2=\eta_2(c,v,\widetilde{\alpha})$
such that for all \(\eta>\eta_2\),
\[
cv^2\alpha_1
+\widetilde{\alpha}v^2
+\frac{4c}{3}
<
\exp\left(
\tilde\eta
\right)
\]
uniformly over all \(\alpha_1\ge0\).
% Hence
% \[
% \frac{\partial^2\mathcal H}{\partial p_r^2}
% +
% \sum_{s\ne r}
% \left|
% \frac{\partial^2\mathcal H}{\partial p_r\partial p_s}
% \right|
% <0,
% \qquad r=2,\ldots,c.
% \]
Thus the Hessian of \(\mathcal H\) is strictly diagonally dominant with negative diagonal entries, and hence negative definite. Therefore \(\mathcal H\) is strictly concave on the above box, which gives uniqueness of the optimizer. Finally, we take
$\eta\ge\eta_\circ:=\max\{\eta_1,\eta_2\}$.
\item[(iii)]
Notice that \eqref{2} implies that every maximizer lies in
the interior of \(\Delta_{c-1}\). Thus it is enough to show that
\(\mathcal H\) is strictly concave on $\operatorname{int}(\Delta_{c-1})$. Fix $x \in \operatorname{int}(\Delta_{c-1})$. For \(r\in[c]\), define
\[
\gamma_r
:=
\alpha_r v(v-1) x_r^{v-2}
-\frac1{x_r}.
\]
Since \(0<x_r<1\) and $0\le \alpha_r\le \frac1{v(v-1)}$,
we have
$\gamma_r
\le
x_r^{v-2}-{x_r}^{-1}<0$.

Using \eqref{eq:hhh}, the Hessian of \(\mathcal H\) at
\((x_2,\ldots,x_c)\) can be written as
\[
\nabla^2\mathcal H(x_2,\ldots,x_c)
=
\gamma_1 {\boldsymbol 1}{\boldsymbol 1}^{\top}
+
\operatorname{diag}(\gamma_2,\ldots,\gamma_c),
\]
where \({\boldsymbol 1}\in\mathbb R^{c-1}\) is the all-one vector.

Therefore, for every nonzero
\({\boldsymbol u}=(u_2,\ldots,u_c)\in\mathbb R^{c-1}\),
\[
{\boldsymbol u}^{\top}
\nabla^2\mathcal H(x_2,\ldots,x_c)
{\boldsymbol u}
=
\gamma_1\left(\sum_{r=2}^c u_r\right)^2
+
\sum_{r=2}^c \gamma_r u_r^2
<0,
\]
because \(\gamma_r<0\) for every \(r\in[c]\). Hence the Hessian is negative
definite. Thus \(\mathcal H\) is strictly concave on
\(\operatorname{int}(\Delta_{c-1})\), and so the maximizer is unique.

\end{itemize}
\end{proof}

We need the following lemma, the proof of which follows from \cite[Lemma 1.4]{bhattacharya2024ldp}, replacing graphons with $v$-hypergraphons and invoking \cref{bijection}.
% We need the following lemma whose proof follows from \cite[Lemma 1.4]{bhattacharya2024ldp}, replacing graphons with $v$-hypergraphons and revoking \cref{bijection}.
\begin{lmm}\label{lmm1.4ext}  Let  $G_{W_\infty, \phi}$  be as in \cref{def:g2}.
For an arbitrary function $\phi:[c]^v\to \R$, a probability measure $\mu$ on $[c]$, and $\beta\in \R$, let $Z(\beta,\mu)$ be as in
  %   \begin{align*}
  % Z({\beta},\mu):= \sup_{{{\boldsymbol f}}\in \mathcal{F}_c}  \Big\{ \beta G_{W_{\infty},\phi}(\boldsymbol f)- \int_0^1\sum_{r=1}^c f_{r}(u)\log \frac{f_{r}(u)}{\mu_r}du \Big\}    
  %   \end{align*}
 the RHS of \eqref{eq:potts_unconstrained_opt2}. Also, let $\mathscr{F}_{\beta, \mu}\subseteq \mathcal{F}_c$  be the set of maximizers in \eqref{eq:potts_unconstrained_opt2}.
    \\
    \begin{itemize}
        \item[(i)]     If $\big\{ G_{W_{\infty},\phi}(\boldsymbol f): {\boldsymbol f} \in \mathscr{F}_{\beta, \mu}\big\}$ has cardinality one for some $\left(\beta, \mu\right)$, and $Z(., \mu)$ is differentiable at $\beta$, then
 $Z^\prime\left(\beta, \mu\right)= G_{W_{\infty},\phi}({\boldsymbol f}^*)$ for ${\boldsymbol f}^* \in \mathscr{F}_{\beta, \mu}$.
 \item[(ii)] Moreover, if for the same $\left(\beta, \mu\right)$ as above we also have $ Z^\prime\left(\beta, \mu\right)=t$, then $$\underset{\substack{\boldsymbol f \in \mathcal{F}_c \\ { G_{W_{\infty},\phi}(\boldsymbol f)= t}}}{\arginf} \int_0^1\sum_{r=1}^c f_{r}(u)\log \frac{f_{r}(u)}{\mu_r}du=\mathscr{F}_{\beta, \mu}. $$
 \end{itemize}
 \end{lmm}

\begin{proof}[Proof of~\cref{cor:Pottsconopt}] 
				\begin{itemize}
				    \item [(i)] With $\mathcal{Z}_{Q_n}(\boldsymbol{\alpha},\boldsymbol{h})$ as in \eqref{logpf},  set $\mathscr Z_{n}(\theta, \boldsymbol h):=\mathcal Z_{Q_n}((\theta, \ldots, \theta), {\boldsymbol h})$, and use \cref{prop:ldp+zn} (ii) to get 
                \begin{align}\label{limzn} 
                &\mathscr{Z}^{(\rm eq)}(\theta, {\boldsymbol h}):=\lim_{n\to\infty}\mathscr{Z}_n(\theta, {\boldsymbol h}) \\
                &=\sup_{{\boldsymbol f}\in \mathcal{F}_c}  \Big\{  \theta\sum_{r=1}^c G_{W_{\infty},\mathbf 1_r}(\boldsymbol f)+\sum_{r=1}^{c}h_r \int_0^1 f_r(u)du
- \int_0^1\sum_{r=1}^c f_{r}(u)\log f_{r}(u)du \Big\}-\log c\nonumber.
\end{align}
Since  $\mathscr Z_n(\cdot, {\boldsymbol h})$ is convex, so is
             $\mathscr{Z}^{(\rm eq)}(\cdot, {\boldsymbol h})$, and hence it is differentiable almost everywhere.
% To begin, use \eqref{eq:potts_unconstrained_opt2} to note that the function ${Z}
% (\beta, \mu)$ is a limit of convex functions, and hence convex. Consequently, so is the function 
% \begin{align}\label{limzn} \mathscr{Z}(\theta, {\boldsymbol h}):= \sup_{{\boldsymbol f}\in \mathcal{F}_c}  \Big\{  \theta\sum_{r=1}^c G_{W_{\infty},\mathbf{1}_r}(\boldsymbol f)+\sum_{r=1}^{c}h_r \int_0^1 f_r(u)du
% - \int_0^1\sum_{r=1}^c f_{r}(u)\log f_{r}(u)du, \Big\}.%\\ &=:Z(\theta,\boldsymbol h).\nonumber
% 		\end{align}
%         \begin{align}\label{limzn} \mathscr{Z}(\theta, {\boldsymbol h}):=\mathcal{Z}_{Q_n}(\theta {\mathbf 1}, {\boldsymbol h})= \sup_{{\boldsymbol f}\in \mathcal{F}_c}  \Big\{  \theta\sum_{r=1}^c G_{W_{\infty},\mathbf{1}_r}(\boldsymbol f)+\sum_{r=1}^{c}h_r \int_0^1 f_r(u)du
% - \int_0^1\sum_{r=1}^c f_{r}(u)\log f_{r}(u)du, \Big\}.%\\ &=:Z(\theta,\boldsymbol h).\nonumber
% 		\end{align}
                 %Invoking \cref{prop:ldp+zn} (ii) we have
                For any  $\theta\geq0$, and $\boldsymbol h$ such that $h_{1} \ge \eta+ \max_{2\le r\le c} h_{r}$ 
               where $\eta=\eta(c,v)$ is as introduced in ~\cref{prop:potts_diffalpha} part (ii), invoking \cref{prop:potts_diffalpha}(ii) gives that $(x_{\theta,\boldsymbol h,1},\ldots,x_{\theta,\boldsymbol h,c})$ is the unique maximizer of \eqref{eq:potts_unconstrained_opt} and $x_{\theta,\boldsymbol h,1}>\max_{2\le r\le c} x_{\theta,\boldsymbol h,r}$. It then follows from applying \cref{lmm1.4ext} part (i) for $\phi=\sum_{r=1}^c\mathbf 1_r$ that for almost every $\theta\geq0$, we have:

% 				 By $\mathcal Z_{n}(\theta, \boldsymbol h)$, we denote $\mathcal Z_{Q_n}(\boldsymbol{\alpha}, {\boldsymbol h})$ when $\boldsymbol{\alpha}=(\theta, \ldots, \theta)$. First notice that for all $n$, $\mathcal Z_n(\theta, {\boldsymbol h})$ is a convex function of $\theta$,
%                 therefore $\mathcal{Z}(\theta, {\boldsymbol h})= \lim_{n \to \infty}\mathcal Z_n(\theta, {\boldsymbol h}) $ is convex and hence it is differentiable almost everywhere as a function of $\theta$ and by \cref{prop:ldp+zn} (ii) we have 
%                 \begin{align}\label{limzn} \mathcal{Z}(\theta, {\boldsymbol h})= \sup_{{\boldsymbol f}\in \mathcal{F}_c}  \Big\{  \theta\sum_{r=1}^c G_{W_{\infty},\mathbf{1}_r}(\boldsymbol f)+\sum_{r=1}^{c}h_r \int_0^1 f_r(u)du
% - \int_0^1\sum_{r=1}^c f_{r}(u)\log f_{r}(u)du \Big\}.%\\ &=:Z(\theta,\boldsymbol h).\nonumber
% 		\end{align}
                
%                 For any  $\theta\geq0$, and $\boldsymbol h$ such that $h_{1} \ge \eta+ \max_{2\le r\le c} h_{r}$ 
%                where $\eta=\eta(c,v)$ is as introduced in ~\cref{prop:potts_diffalpha} part (ii), invoking ~\cref{prop:potts_diffalpha} gives that $(x_{\theta,1},\ldots,x_{\theta,c})$ is the unique maximizer of \eqref{eq:potts_unconstrained_opt}. It then follows from applying \cref{lmm1.4ext} part (i) for $\phi=\sum_{r=1}^c\mathbf 1_r$ that for almost every $\theta\geq0$, we have:
				\begin{equation}\label{eq:impeq1}
				\big(\mathscr{Z}^{(\rm eq)}\big)^\prime\left(\theta, {\boldsymbol h}\right)=\sum_{r=1}^cx_{\theta,\boldsymbol h,r}^v.
				\end{equation}
				The map $g:\theta\mapsto (x_{\theta,1},\ldots ,x_{\theta,c})$ is continuous by~\cref{prop:potts_diffalpha} part (ii). Consequently %the right hand side of~\eqref{eq:impeq} admits a continuous extension on the non-negative half line. This implies
				$\mathscr{Z}^{(\rm eq)}$ is differentiable everywhere as a function of $\theta$  and~\eqref{eq:impeq1} holds for all $\theta\geq 0$. Also, it is easy to check that $g(0)=(\mu_1,\ldots ,\mu_c)$ and $\lim_{\theta\to\infty} g(\theta)=(1,0,\ldots ,0)$. As $\boldsymbol h$ is fixed in this result, we will drop $\boldsymbol h$ from all our notation for the sake of simplicity. By continuity of $\big(\mathscr{Z}^{(\rm eq)}\big)^\prime(\cdot)$ we have $\big[\sum_{r=1}^c \mu_r^v,1 \big)\subseteq \big(\mathscr{Z}^{(\rm eq)}\big)^\prime[0,\infty)$. So, for any $s\in \big[\sum_{r=1}^c \mu_r^v,1\big)$, there exists some $\theta_s$ such that  $\sum_{r=1}^c x_{\theta_s,r}^v=s$. We now claim such a solution is unique (on $[0,\infty)$), and consequently we will refer to it as $\theta_s$. To this effect, it suffices to show that $\big(\mathscr{Z}^{(\rm eq)}\big)^\prime(\cdot)$ is strictly increasing. Equivalently, we will now show that \begin{equation}\label{eq:monotone}
				\theta_1\neq \theta_2\in [0,\infty)\quad \implies \quad \big(\mathscr{Z}^{(\rm eq)}\big)^\prime(\theta_1)\neq\big(\mathscr{Z}^{(\rm eq)}\big)^\prime(\theta_2).
				\end{equation}
				
				First we show that the function $g$ is injective. %$(x_{\theta_1,1},\ldots ,x_{\theta_1,c})\neq (x_{\theta_2,1},\ldots ,x_{\theta_2,c})$. 
				If not, then we would have $(x_{\theta_1,1},\ldots ,x_{\theta_1,c})= (x_{\theta_2,1},\ldots ,x_{\theta_2,c})$ for some $\theta_1<\theta_2$. In particular, $x_{\theta_1,1}=x_{\theta_2,1}$.
                On the other hand,  by~\eqref{2}, we have
				\begin{align}\label{contrd}
				    x_{\theta_1,1}=\frac{e^{v \theta_1x_{\theta_1,1}^{v-1}+h_1}}{\sum_{r=1}^c e^{v\theta_1 x_{\theta_1,r}^{v-1}+h_r}}, \qquad x_{\theta_2,1}=\frac{e^{v \theta_2x_{\theta_2,1}^{v-1}+h_1}}{\sum_{r=1}^c e^{v\theta_2 x_{\theta_2,r}^{v-1}+h_r}}.
				\end{align}
				 Moreover, one can check that the map 
				$$\theta\mapsto \frac{\exp(v\theta z_1^{v-1}+h_1)}{\sum_{r=1}^c \exp(v\theta z_r^{v-1}+h_r)},$$
				is strictly increasing in $\theta\in [0,\infty)$, provided $z_1>\max_{2\le r\le c} z_r$, for any ${\boldsymbol h}=(h_1,\ldots ,h_c)$.  Therefore, \eqref{contrd} with $z_r=x_{\theta_1,r}=x_{\theta_2,r}$ for $r\in [c]$ yield  $$x_{\theta_1,1}=\frac{\exp(v\theta_1 z_1^{v-1}+h_1)}{\sum_{r=1}^c \exp(v\theta_1 z_r^{v-1}+h_r)}<\frac{\exp(v\theta_2 z_1^{v-1}+h_1)}{\sum_{r=1}^c \exp(v\theta_2 z_r^{v-1}+h_r)}=x_{\theta_2,1}$$ which is a contradiction. So $g$ is an injective function. To  establish~\eqref{eq:monotone} by contradiction, assume $\big(\mathscr{Z}^{(\rm eq)}\big)^\prime(\theta_1)=\sum_{r=1}^c x_{\theta_1,r}^v=\big(\mathscr{Z}^{(\rm eq)}\big)^\prime(\theta_2)=\sum_{r=1}^c x_{\theta_2,r}^v$ for some $\theta_1\ne \theta_2$. By the optimality and uniqueness of $(x_{\theta_1},\ldots ,x_{\theta_1,c})$, coupled with the fact that $(x_{\theta_1,1},\ldots ,x_{\theta_1,c})\neq (x_{\theta_2,1},\ldots ,x_{\theta_2,c})$, we have
				\begin{align*}
				&\theta_1\sum_{r=1}^c x_{\theta_1,r}^v+\sum_{r=1}^c (h_r-\log{(x_{\theta_1,r})}) x_{\theta_1,r}> \theta_1\sum_{r=1}^c x_{\theta_2,r}^v+\sum_{r=1}^c (h_r-\log{(x_{\theta_2,r})}) x_{\theta_2,r}\\ \Rightarrow &\sum_{r=1}^c (h_r-\log{(x_{\theta_1,r})}) x_{\theta_1,r} > \sum_{r=1}^c (h_r-\log{(x_{\theta_2,r})}) x_{\theta_2,r}.  
				\end{align*}
				By switching the roles of $\theta_1$ and $\theta_2$ in the above argument, we get the reverse inequality, which immediately yields a contradiction. This establishes~\eqref{eq:monotone}.
				So for any $s \in \big[\sum_{r=1}^c \mu_r^v,1 \big)$, there exists a unique $\theta_s$ where $s=\big(\mathscr{Z}^{(\rm eq)}\big)^\prime(\theta_s)=\sum_{r=1}^cx_{{\theta}_s,r}^v$, and using \cref{lmm1.4ext} part (ii) we are done.

% 				Now, assume $ t>1.$ With ${\boldsymbol G}_{W_\infty}(\boldsymbol f)$ as in~\cref{def:g2}, consider any function $\boldsymbol f\in \mathcal{F}_c$ such that $\sum_{r=1}^c G_{W_{\infty},\mathbf{1}_r}(\boldsymbol f)= { t}$.
                
%                 Since $\mathcal{V}[W_\infty](x)=1$ for $\lambda$-a.e. $x\in [0,1]$, we have
%                 \begin{align*}
%                     &1<t=\sum_{r=1}^c G_{W_{\infty},\mathbf{1}_r}(\boldsymbol f)= \sum_{r=1}^c \E_{(Y_1,\ldots ,Y_v)\sim W_\infty}\bigg[\prod_{i=1}^v f_r(Y_i)\bigg]\\
%                     & \leq \sum_{r=1}^c \int_{[0,1]}  f_r^v(x)\,dx \leq \sum_{r=1}^c \int_{[0,1]}  f_r(x)\,dx =1, 
%                 \end{align*}
%                 which is a contradiction. The first $\leq$ above is justified by applying generalized H\"older's inequality.
%                 So this ${\boldsymbol f}$ does not exist and hence $I_{W_\infty,(1,\ldots,1)}({ t})=\infty.$
% A similar proof works for other rate functions discussed in \cref{OPT}, so we do not repeat.

				 \item [(ii)]
                % The proof follows along similar lines  as in part (i) on invoking~\cref{prop:potts_gibbs_ldp} part (ii). 

    %             Since for $\theta> \theta_{\rm crit}$ the maximizer in the \eqref{limzn} where $z_1 \geq \max_{2\le r\le c} z_r$ is unique, following the similar steps as in the proof of part (i), we have for every $\theta>\theta_{\rm crit}$, $\mathscr{Z}'(\theta)=\sum_{r=1}^c x_{\theta,r}^v$,   
				% and for any $\theta_{\rm crit}^+>\theta_{\rm crit}$, we have  $\mathscr{Z}'[\theta_{\rm crit}^+,\infty)=\big[\mathscr{Z}^\prime(\theta_{\rm crit}^+),1\big)$. So, for any $s\in \big[\mathscr{Z}^\prime(\theta_{\rm crit}^+),1\big)$, there exists some $\theta_s$ such that  $\sum_{r=1}^c x_{\theta_s,r}^v=s$.
    %             % The only difference is that in this case we use~\cref{prop:potts_gibbs_ldp} part (iv) to note the symmetry of the  $c$ distinct optimizers, and conclude that the quantity $\sum_{r=1}^c z_r^v$ is invariant on the set of optimizers. Thus \cite[Lemma 1.4]{bhattacharya2024ldp} is applicable, giving $$Z'(\theta)=\sum_{r=1}^c x_{\theta,r}^v.$$ 
    %             The rest of the proof is identical to part (i), which we omit for brevity. 
               
             The proof follows along similar lines  as in part (i) on invoking~\cref{prop:potts_gibbs_ldp} part (ii). For every \(\theta>\theta_{\rm crit}\), \cref{prop:potts_gibbs_ldp} (ii) implies that the maximizers are precisely the \(c\) permutations of a vector
\[
x_\theta=(x_{\theta,1},\ldots,x_{\theta,c}),
\qquad x_{\theta,1}>\max_{2\le r\le c}x_{\theta,r},
\]
and that the map \(\theta\mapsto x_\theta\) is continuous on \((\theta_{\rm crit},\infty)\). Although the optimizer is not unique, all maximizers are permutations of one another, so the value of
$\sum_{r=1}^c x_{\theta,r}^v$
is the same for every maximizer. Therefore \cref{lmm1.4ext} part (i), applied with \(\phi=\sum_{r=1}^c{\bf 1}_r\), gives, exactly as in part (i),
\[
\big(\mathscr{Z}^{(\rm eq)}\big)^\prime\left(\theta\right)=\sum_{r=1}^cx_{\theta,r}^v,
\qquad \theta>\theta_{\rm crit}.
\]
The continuity of \(g:\theta\mapsto x_\theta\) upgrades the a.e. identity to every \(\theta>\theta_{\rm crit}\).

Repeating the argument from part (i), we get that \(\theta\mapsto \sum_{r=1}^c x_{\theta,r}^v\) is strictly increasing on \((\theta_{\rm crit},\infty)\).

By continuity and the fact that
$\lim_{\theta\to\infty}\sum_{r=1}^c x_{\theta,r}^v=1$,
 \(\theta\mapsto \sum_{r=1}^c x_{\theta,r}^v\) sends \((\theta_{\rm crit},\infty)\) bijectively onto \((y,1)\), where
\[
y:=\lim_{\theta\downarrow\theta_{\rm crit}}\sum_{r=1}^c x_{\theta,r}^v.
\]
Thus, for every \(t\in(y,1)\), there is a unique \(\theta_t>\theta_{\rm crit}\) such that
$\sum_{r=1}^c x_{\theta_t,r}^v=t$.
Applying both parts of \cref{lmm1.4ext} proves the claim. \end{itemize}
			\end{proof}
                  \begin{proof}[Proof of \cref{c2-1}]

Notice that if $\alpha_1=\alpha_2=0$, the whole statement holds trivially, so we assume $\alpha_1+\alpha_2>0$.
   \begin{itemize}
    \item [(i)]
Using \cref{prop:propotts} (i) and (iii), we get that the maximizers of \eqref{eq:potts_unconstrained_opt} are constants of the form $(f_1,f_2)\equiv(s,1-s)$ where \[s=\frac{\exp (2s(\alpha_1+\alpha_2)+2p^*)}{\exp (2s(\alpha_1+\alpha_2)+2p^*) +1}\] and $p^*:= \frac{1}{2}\log(\frac{p_1}{p_2})-\alpha_2$. Assume without loss of generality, $\alpha_1 + \log p_1> \alpha_2+\log p_2$. One can easily check that we must have $s\geq \frac{1}{2}$ (otherwise, the value of \eqref{eq:potts_unconstrained_opt} for $(1-s,s)$ is larger than $(s,1-s)$ which contradicts our optimality assumption). Now if $t:=2s-1$, we have $t\in [0,1]$ and,
$$t=\tanh\bigg(t \left(\frac{\alpha_1+\alpha_2}{2}\right)+ \frac{\alpha_1-\alpha_2+\log p_1-\log p_2}{2}\bigg).$$
     Now the function $x \mapsto \tanh(Ax+B)$ with $A,B>0$ has exactly one fixed point in $[0,1]$, so $t$ and hence $s$ is unique. Moreover, notice that $t>0$ and therefore $s>1/2$. The argument for the case $\alpha_1 + \log p_1< \alpha_2+\log p_2$ is similar, so we omit this.

    \item[(ii)] Arguing similarly to the previous part, we need to find the fixed points of the function $t \mapsto \tanh(t(\frac{\alpha_1+\alpha_2}{2}))$. It is easy to see that if $\frac{\alpha_1+\alpha_2}{2}\leq 1$, the only fixed point is $t=0$, so $(1/2,1/2)$ is the only optimizer of \eqref{eq:potts_unconstrained_opt}. If $\frac{\alpha_1+\alpha_2}{2}> 1$, then we have three fixed points in the form $-r,0,r$, for some $r \in (0,1)$. So $s$ equals $\frac{1-r}{2}, \frac{1}{2}, \frac{1+r}{2}$. We can also check that the second derivative of the function in \eqref{eq:potts_unconstrained_opt} is positive for $s=\frac{1}{2}$, so it is indeed a minimum. The maximizers of \eqref{eq:potts_unconstrained_opt} are $(\frac{1-r}{2},\frac{1+r}{2})$ and $(\frac{1+r}{2},\frac{1-r}{2})$.
    % (ii) It follows from two previous parts. 
    \end{itemize}
\end{proof}
\begin{proof}[Proof of \cref{hmm?}]
% The proof follows exactly the same argument as in the proof of
% \cref{cor:Pottsconopt}, part~(i), specialized to the case \(c=v=2\).
Using \cref{prop:propotts} part (iii), the optimization problem in \eqref{eq:potts_unconstrained_opt} reduces to 
\[\mathscr{Z}^{(\rm eq)}(\theta,p):=\sup_{q\in [0,1]}\left\{ \theta q^2+\theta (1-q)^2-q\log\frac{q}{p}-(1-q)\log\frac{1-q}{1-p}\right\}.\]
Using \cref{c2-1} part (i) together with
\(\alpha_1=\alpha_2\ge0\) and \(p>1/2\), we conclude uniqueness of the optimizer of the above problem, which is of the form $q_\theta$ with $q_\theta>\frac{1}{2}$,
% \[
% (s_\theta,1-s_\theta),
% \qquad s_\theta>\frac12 ,
% \]
where $\theta:=\alpha_1=\alpha_2$.
Therefore, using an argument similar to the proof of \cref{cor:Pottsconopt} part (i), we obtain that the map
$\theta\mapsto \big(\mathscr{Z}^{(\rm eq)}\big)^\prime(\theta,p)$
is continuous and strictly increasing on \([0,\infty)\).
%where $ \mathscr{Z}^{(2)}(\theta)$ is defined in \eqref{limzn}, and simplifies to
% \[ \mathscr{Z}^{(2)}(\theta)= \sup_{{\boldsymbol f}\in \mathcal{F}_2}  \Big\{  \theta G_{W_{\infty},\mathbf 1_1}(\boldsymbol f)+\theta G_{W_{\infty},\mathbf 1_2}(\boldsymbol f)-\int_0^1 \big(f_{1}(u)\log \frac{f_{1}(u)}{p}+ f_{2}(u)\log \frac{f_{2}(u)}{1-p}\big)du \Big\}.\]
Moreover, it is not hard to check that
\[
 \big(\mathscr{Z}^{(\rm eq)}\big)^\prime(0,p)=p^2+(1-p)^2,
\qquad
\lim_{\theta\to\infty} \big(\mathscr{Z}^{(\rm eq)}\big)^\prime(\theta,p)=1.
\]
Hence, for every
$y\in[p^2+(1-p)^2,1)$,
there exists a unique \(q_y>1/2\) satisfying
$q^2+(1-q)^2=y$.
Applying \cref{lmm1.4ext} part (ii) (similar to the proof of
\cref{cor:Pottsconopt} part (i)) yields
\[
I_{W_\infty,(1,1)}(y)
=
q_y\log\frac{q_y}{p}
+
(1-q_y)\log\frac{1-q_y}{1-p},
\]
which completes the proof.

\end{proof}  % We will use the following lemma to simplify the function $I(\cdot)$ (see \eqref{eq:i2}). Its proof follows directly from \cite[Lemma 1.4 and Corollary 1.5]{bhattacharya2024ldp}. 

% \begin{lmm}\label{lem:simrate}
%     Suppose the assumptions from \cref{unc} hold. Recall the definitions of $Z(\cdot)$, $G_W(\cdot)$, and $\beta(\cdot)$, $\gamma(\cdot)$ from \eqref{eq:opt_uncon}, \eqref{eq:quadG} and \ref{def:tilt} respectively. Define $F_{\theta} \subseteq \mathcal{L}$ to be the set where the supremum of \eqref{eq:opt_uncon} is achieved. Suppose $\theta\ge 0$ is such that the set $G_W(F_{\theta})$ is a singleton and $Z(\cdot)$ is differentiable at $\theta$. Then the following conclusions hold:
%     \begin{enumerate}
%         \item[(a)] $Z'(\theta)=G_W(F_{\theta})$.
%         \item[(ii)] Suppose $Z'(\theta)=t$. Then, $I(t)=\int_0^1 \gamma(\beta(f(x)))\,dx$ for any $f\in F_{\theta}$. In particular, the right hand side of the above is the same for any $f\in F_{\theta}$       
%     \end{enumerate}
% \end{lmm}
We need two preparatory lemmas for proving~\cref{unc}.
\begin{lmm}\label{smdef}
Let \(p\in(0,1)\) and \(\theta\ge 0\). Define
\[
\mv_{\theta,p}(x)
:=
\theta x^2
-x\log\frac{x}{p}
-(1-x)\log\frac{1-x}{1-p},
\qquad x\in[0,1],
\]
where here and throughout, we use the convention $0\log 0:=0$.

If $(\theta,p) \in \Omega$, where
 \begin{align}\label{defomega}
\Omega
:=
([0,\infty)\times(0,1))
\setminus
\left\{
\left(\log\frac{1-p}{p},p\right):
0<p<\frac1{1+e^2}
\right\},\end{align}

then \(\mv_{\theta,p}\) has a unique maximizer on \([0,1]\), denoted by $\sm$. Consequently, the map $(\theta,p)\mapsto \sm$ is well-defined on $\Omega$.
\end{lmm}

% \begin{lmm}\label{lem:rootlim}
% Consider the function $h(\theta,B)=\tm$ for $\theta\ge 0$ and $B\in\R$. Then
% \begin{itemize}
% %\item[(a)] For fixed $B\in\R$, the map $\theta\mapsto h(\theta,B)$ is continuous on $[0,\infty)$.
% \item[(a)] The function $h(\cdot,\cdot)$ is jointly continuous in the region $$\Omega:=\{(\theta,B):\ B=0,\ \theta\le (\alpha''(0))^{-1}/2\}\cup \{(\theta,B):\ B\neq 0,\ \theta\ge 0\}.$$
% \item[(ii)] Fix  $\theta_{\infty}>(\alpha''(0))^{-1}/2$. Then we have:
% 			$$\lim_{\theta\to\theta_{\infty},\ B\to 0^{-}} h(\theta,B)=-t_{\theta_{\infty},0,\mu},\quad \lim_{\theta\to\theta_{\infty},\ B\to 0^{+}} h(\theta,B)=t_{\theta_{\infty},0,\mu}.$$
%    In particular, $h(.,.)$ is not continuous on the closure of $\Omega$.
% \end{itemize}
% \end{lmm}
\begin{lmm}\label{lem:rootlim}
Consider the function $(\theta,p)\mapsto \sm$ on the region $\Omega$ defined as $\eqref{defomega}$. Then
\begin{itemize}
%\item[(a)] For fixed $B\in\R$, the map $\theta\mapsto h(\theta,B)$ is continuous on $[0,\infty)$.
\item[(i)] The function $s_{(\cdot,\cdot)}$ is jointly continuous in $\Omega$.
\item[(ii)] Fix  $0<p<\frac{1}{1+e^2}$. Then we have
			$$\lim_{ \theta\to \big(\log\frac{1-p}{p}\big)^{-}} s_{\theta,p}=\frac{1-t_{\log\frac{1-p}{p},p}}{2},\quad \lim_{ \theta\to \big(\log\frac{1-p}{p}\big)^{+}} s_{\theta,p}=\frac{1+t_{\log\frac{1-p}{p},p}}{2},$$
            where $t_{\log\frac{1-p}{p},p}$ is defined as the unique positive solution of $x=\tanh(\frac{x}{2}\log\frac{1-p}{p})$.
   In particular, $s_{(.,.)}$ does not admit a continuous extension to the closure of $\Omega$.
   \item[(iii)] One has $$\lim_{ p\to \big(\frac{1}{1+e^2}\big)^{+}} s_{\log\frac{1-p}{p},p}=1/2,\quad\lim_{ p\to \big(\frac{1}{2}\big)^{-}} s_{\log\frac{1-p}{p},p}=1/2.$$
   
 \item[(iv)]Fixing $p\in(0,1)$, one has $$\lim_{ \theta\to 0^{+}} s_{\theta,p}=s_{0,p}=p,\quad\lim_{  \theta\to \infty} s_{\theta,p}=1.$$
 
 \item [(v)]Fix $p \in (0,1)$. Assume that for every $\theta$ in an interval $J\subseteq\mathbb{R}$, $\mv_{\theta,p}$ admits a unique maximizer, denoted by $\sm$.
Then the map $\theta\mapsto \sm$ is strictly increasing on $J$.
\end{itemize}
\end{lmm}

\begin{proof}[Proof of \cref{unc}]

\begin{itemize}
%     \item [(i)]The assumption $p\ge \frac{1}{1+e^2}$ is equivalent to $ \log \frac{1-p}{p} \le 2$. Using \cref{c2-1}, \eqref{eq:potts_unconstrained_opt} has a unique maximizer called $(s_{\theta,p},1-s_{\theta,p})$. Using \cref{lem:rootlim}, and noticing $s_{0,p}=p, \lim_{\theta \to \infty}s_{\theta,p}=1$, we get that $\sm$ maps $[0,\infty)$ to $[p,1)$ injectively.  

% Now using \cref{lmm1.4ext} for $\mathscr{Z}$ as in \eqref{limzn}, we conclude that $\mathscr{Z}^\prime(\theta,(\log[(1-p)^{-1}-1],0))=\sm^2$ for almost all $\theta\ge 0$, and using continuity of $\sm$, $\mathscr{Z}^\prime(\theta,(\log[(1-p)^{-1}-1],0))=\sm^2$ for all $\theta\ge 0$.
% Again using \cref{lmm1.4ext} with $\phi=\mathbf 1_1$, we can simplify the good rate function as desired. 

% \item[(ii)]\begin{itemize}
%      \item[(a)] Fix $p<\frac{1}{1+e^2}$. For $\theta\neq \log \frac{1-p}{p}$, using \cref{c2-1}, \eqref{eq:potts_unconstrained_opt} has a unique maximizer called $(s_{\theta,p},1-s_{\theta,p})$. Now consider $\sm$ as a function of $\theta$. Using \cref{lmm1.4ext}, we can get that $\sm$ maps  $[0,\log \frac{1-p}{p})$ to $[p, 1-s_{\log \frac{1-p}{p},p})$ and $(\log \frac{1-p}{p},\infty)$ to $(s_{\log \frac{1-p}{p},p},1)$ injectively. Notice that from \cref{lem:rootlim} we infer that $\eta:=s_{\log \frac{1-p}{p},p}>1/2$. Similar to part (i) above, we can simplify the good rate function. 
\begin{itemize}
\item[(i)]

Using \cref{prop:propotts} part (iii), the optimization problem in \eqref{eq:potts_unconstrained_opt} reduces to 
\begin{align}\label{zun}
\mathscr{Z}^{(\rm un)}(\alpha_1,p):=\sup_{s\in [0,1]}\left\{ \alpha_1 s^2-s\log\frac{s}{p}-(1-s)\log\frac{1-s}{1-p}\right\}.    
\end{align}
Since \(p\ge (1+e^2)^{-1}\), equivalently
\(\log\frac{1-p}{p}\le 2\), by invoking both parts of \cref{c2-1}, the optimization problem
above admits a unique maximizer
$s_{\alpha_1,p}$ for every \(\alpha_1\ge0\). By \cref{lem:rootlim} part (v) for $J=[0,\infty)$, the function \(\alpha_1\mapsto s_{\alpha_1,p}\) is strictly increasing. Also by \cref{lem:rootlim} part (iv) we have
\(s_{0,p}=p\) and \(s_{\alpha_1,p}\to1\) as \(\alpha_1\to\infty\). Hence
\(\alpha_1\mapsto s_{\alpha_1,p}\) is a bijection from \([0,\infty)\) onto
\([p,1)\).

% Now define \reihaneh{! zeq?}
% \[
% \mathscr Z^{(\rm un)}(\alpha_1,p)
% :=
% \sup_{{\boldsymbol f}=(f,1-f)}
% \left\{
% \alpha_1 G_{W_\infty,\mathbf 1_1}({\boldsymbol f})
% -
% \left(s\log\frac{s}{p}
% +
% (1-s\log\frac{1-s}{1-p}\right)
% \right\},
% \]
% (see \eqref{def1k}).
Applying \cref{lmm1.4ext} part (i), we obtain
\[
(\mathscr Z^{(\rm un)})'(\alpha_1,p)=s_{\alpha_1,p}^2
\]
for a.e.\ \(\alpha_1\ge0\); continuity of \(s_{\alpha_1,p}\) from
\cref{lem:rootlim} part (i) extends this identity to all \(\alpha_1\ge0\).

Applying \cref{lmm1.4ext} part (ii) with $\phi=\mathbf 1_1$, and writing \(y=s_{\alpha_1,p}^2\) gives
\[
I_{W_\infty,(1,0)}(y)
=
\sqrt y\log\frac{\sqrt y}{p}
+
(1-\sqrt y)\log\frac{1-\sqrt y}{1-p},
\qquad
y\in[p^2,1).
\]

\item[(ii)(a)]
Fix \(p<(1+e^2)^{-1}\). By \cref{c2-1} part (i), for every
\(\alpha_1\neq\log\frac{1-p}{p}\), the optimization problem
\eqref{zun} has a unique maximizer
\(s_{\alpha_1,p}\).
By \cref{lem:rootlim} part (v) on the intervals \([0,\log\frac{1-p}{p})\) and \((\log\frac{1-p}{p},\infty)\), we get that \(\alpha_1\mapsto s_{\alpha_1,p}\) is strictly increasing on both intervals separately. Moreover, by
\cref{lem:rootlim} parts (ii) and (iv),
\[
s_{0,p}=p,
\qquad
\lim_{\alpha_1\uparrow\log\frac{1-p}{p}}s_{\alpha_1,p}=1-\eta,
\qquad
\lim_{\alpha_1\downarrow\log\frac{1-p}{p}}s_{\alpha_1,p}=\eta,
\]
where
\[
\eta:=\frac{1+t_{\log\frac{1-p}{p},p}}2>\frac12.
\]

Therefore \(\alpha_1\mapsto s_{\alpha_1,p}\) maps
\([0,\log\frac{1-p}{p})\) bijectively onto \([p,1-\eta)\), and
\((\log\frac{1-p}{p},\infty)\) bijectively onto \((\eta,1)\).

Again applying \cref{lmm1.4ext} part (i), \((\mathscr Z^{(\rm un)})'(\alpha_1,p)=s_{\alpha_1,p}^2\), and applying
\cref{lmm1.4ext} part (ii) with $\phi=\mathbf 1_1$, and writing \(y=s_{\alpha_1,p}^2\) yields
\[
I_{W_\infty,(1,0)}(y)
=
\sqrt y\log\frac{\sqrt y}{p}
+
(1-\sqrt y)\log\frac{1-\sqrt y}{1-p}
\]
for
\[
y\in[p^2,(1-\eta)^2)\cup(\eta^2,1).
\]

\item[(ii)(b)]
 For $\delta> 0$, consider the function $W_{(\delta)}:[0,1]^2\to (0,\infty)$ given by
				$$W_{(\delta)}(x,y):=\begin{cases} \frac{4}{2+\delta} & \mbox{if}\ (x,y)\in \left[0,\frac{1}{2}\right]^2\cup \left[\frac{1}{2},1\right]^2\\ \frac{2\delta}{2+\delta} & \mbox{if}\ (x,y)\in [0,1]^2\setminus \left( \left[0,\frac{1}{2}\right]^2\cup \left[\frac{1}{2},1\right]^2\right)\end{cases}.$$
				Note that \begin{align}\label{eq:supcon1}\lVert W_{(\delta)}-W_\circ\rVert_{\infty}\to 0 \qquad \mbox{as}\quad  \delta\to 0,\end{align}
                where
                \begin{align}\label{Wcirc}
                    W_\circ(x,y):=\begin{cases}2 & \mbox{if}\ (x,y)\in \left[0,\frac{1}{2}\right]^2\cup \left[\frac{1}{2},1\right]^2\\ 0 & \mbox{if}\ (x,y)\in [0,1]^2\setminus \left( \left[0,\frac{1}{2}\right]^2\cup \left[\frac{1}{2},1\right]^2\right)\end{cases}.\end{align}
                Note that
\[
\int_0^1 W_\circ(x,y)\,dy=1
\qquad \lambda\text{-a.e. }x.
\]
However, $W_\circ(\cdot,\cdot)$ does not satisfy the positivity
constraint. We use the following lemma, proved in
\cref{AuxiliarySection}.
                \begin{lmm}\label{Wc}
For \(W_\circ(\cdot,\cdot)\) defined in \eqref{Wcirc}, none of the optimizers of
\[
%I_{{W_\circ,(1,0)}}(\frac 14)
%=
\inf_{{\boldsymbol f}=(f,1-f):\  G_{W_\circ,\mathbf 1_1}(\boldsymbol f)=1/4}
\left\{
\int_0^1
f(u)\log \frac{f(u)}{p}
+
(1-f(u))\log\frac{1-f(u)}{1-p}
du
\right\}
\]
is a constant function.
\end{lmm}
                
                % In~\cite[Theorem 1.1(b)]{SomabhaBhattacharya2020} (also see~\cite[Lemma A.1]{lubetzky2015replica}), For $W_\circ(\cdot,\cdot)$, it is shown in \cref{wc} that there exists $1/4\in \big[(1-\eta)^2, \eta^2\big]$, such that none of the optimizers of the optimization problem $$I_{{W_\circ,(1,0)}}({1/4}):=\inf_{{\boldsymbol f}=(f,1-f):\  G_{W_\circ,\mathbf 1_1}(\boldsymbol f)=1/4}  \Big\{ \int_0^1f(u)\log \frac{f(u)}{p}+(1-f(u))\log\frac{1-f(u)}{1-p}du\Big\}$$ is a constant function.
                We now show that the same property holds with
$W_{(\delta)}(\cdot,\cdot)$, for all sufficiently small $\delta>0$.

Fix
${\boldsymbol f}^*(\cdot)=(f^*(\cdot),1-f^*(\cdot))$, so that
\begin{align}\label{eq:baseprob1}
{\boldsymbol f}^*
\in
\argmin_{
{\boldsymbol f}=(f,1-f):
\,G_{W_\circ,\mathbf 1_1}({\boldsymbol f})=1/4
}
\left\{
\int_0^1
f(u)\log\frac{f(u)}{p}
+
(1-f(u))\log\frac{1-f(u)}{1-p}
\,du
\right\}.
\end{align}
By \cref{Wc}, $f^*$ is nonconstant. Define
\[
y_\delta
:=G_{W_{(\delta)}, \mathbf 1_1}(\boldsymbol f^*)=
\int_{[0,1]^2}
W_{(\delta)}(x,y)f^*(x)f^*(y)\,dx\,dy,
\qquad \delta>0.
\]
By \eqref{eq:supcon1} and the fact that
$G_{W_\circ,\mathbf 1_1}({\boldsymbol f}^*)=\frac14$,
we have
\[
y_\delta\to\frac14
\qquad \mbox{as } \delta\to0.
\]
We claim that, for all sufficiently small $\delta>0$, the constrained
problem obtained from \eqref{eq:baseprob1} by replacing $W_\circ$ and
$1/4$ with $W_{(\delta)}$ and $y_\delta$, respectively, has no constant
optimizer. Suppose otherwise. Then there exists a sequence $\delta_k\downarrow0$
such that the corresponding constrained problem admits a constant
optimizer \({\boldsymbol f}_{\delta_k}(\cdot)
=
\left(f_{\delta_k}(\cdot),1-f_{\delta_k}(\cdot)\right).
\)
Since
\[
\int_0^1 W_{(\delta_k)}(x,y)\,dy=1
\qquad \lambda\text{-a.e. }x,
\]
the constraint $G_{W_{(\delta_k)},\mathbf 1_1}({\boldsymbol f}_{\delta_k})=y_{\delta_k}$
forces
\(
f_{\delta_k}(\cdot)=\sqrt{y_{\delta_k}}.
\)
By the definition of $y_{\delta_k}$, both $f^*$ and
$f_{\delta_k}\equiv\sqrt{y_{\delta_k}}$ satisfy this constraint. Therefore, by the optimality of
$f_{\delta_k}$,
\begin{align*}
&
\sqrt{y_{\delta_k}}
\log\frac{\sqrt{y_{\delta_k}}}{p}
+
\left(1-\sqrt{y_{\delta_k}}\right)
\log\frac{1-\sqrt{y_{\delta_k}}}{1-p}
\\
&\hspace{2cm}\leq
\int_0^1
f^*(u)\log\frac{f^*(u)}{p}\,du
+
\int_0^1
(1-f^*(u))
\log\frac{1-f^*(u)}{1-p}\,du.
\end{align*}
On letting $k\to\infty$, we get
\begin{align*}
& \sqrt{1/4}\log \frac{\sqrt{1/4}}{p}+  (1-\sqrt{1/4})\log \frac{1-\sqrt{1/4}}{1-p}\\
&\le \int_0^1 f^*(u)\log \frac{f^*(u)}{p}du+ \int_0^1 (1-f^*(u))\log \frac{1-f^*(u)}{1-p}du.
\end{align*}
Since $(f_\circ(\cdot), 1-f_\circ(\cdot))=(1/2,1/2)$ is feasible for
\eqref{eq:baseprob1} and $\boldsymbol f^*$ is an optimizer of that
problem, $(f_\circ(\cdot), 1-f_\circ(\cdot))=(1/2,1/2)$ is also an
optimizer of \eqref{eq:baseprob1}, contradicting \cref{Wc}.
% Consequently $(f_\circ(\cdot), 1-f_\circ(\cdot))=({1/2}, 1/2)$ is a constant optimizer of the problem~\eqref{eq:baseprob1}, this is a contradiction to \cref{Wc}.
                \end{itemize}\end{itemize}
Therefore, for all sufficiently small $\delta>0$, none of the
optimizers of the constrained problem associated with $W_{(\delta)}$
and $y_\delta$ is constant. Moreover, since
$y_\delta\to1/4\in\big((1-\eta)^2,\eta^2\big)$,
we have
$y_\delta\in\big[(1-\eta)^2,\eta^2\big]$
for all sufficiently small $\delta>0$. Choosing one such $\delta$ and
setting
\(
W_\infty:=W_{(\delta)},\ 
y_\infty:=y_\delta
\)
completes the proof.
\end{proof}

            \begin{proof}[Proof of \cref{negvalue1}]
\begin{itemize}
    \item[(i)] By \cref{prop:propotts} (i), noticing $f_1+f_2=1$, we get $\lambda$-a.e.
    \begin{align*}
        f_1(x) &=\frac{p\exp{(2\theta\int W_\infty(x,y)f_1(y)dy)}}{p\exp{(2\theta\int W_\infty(x,y)f_1(y)dy)}+(1-p)\exp{(-2\theta\int W_\infty(x,y)(1-f_1(y))dy)}}\\
        &= \frac{p\exp({2\theta})}{1-p+p\exp({2\theta})},
    \end{align*}
where the final equality uses $\int W_\infty(x,y)dy=1$. %Using the regularity assumption  $\lambda$-a.e., this simplifies to.

\item[(ii)]
By part (i), the unique maximizer of
\eqref{eq:potts_unconstrained_opt} is
$(f_1(\cdot),f_2(\cdot))
=
(u_{\theta,p},1-u_{\theta,p})$,
where
\[
u_{\theta,p}
=
\frac{p\exp(2\theta)}
{1-p+p\exp(2\theta)}
=
\frac{p}{p+(1-p)\exp(-2\theta)}.
\]
It is easy to check that $\theta\mapsto u_{\theta,p}$ is strictly
increasing and $C^1$, with
\[
\lim_{\theta\to-\infty}u_{\theta,p}=0,
\qquad
\lim_{\theta\to\infty}u_{\theta,p}=1.
\]
Consequently, $\theta\mapsto u_{\theta,p}$ is a $C^1$ bijection from
$\mathbb R$ onto $(0,1)$. Define
\[
{\mathscr Z}^{(\rm op)}(\theta,p)
:=
\sup_{{\boldsymbol f}\in\mathcal F_2}
\left\{
\theta\left(
G_{W_\infty,\mathbf 1_1}({\boldsymbol f})
-
G_{W_\infty,\mathbf 1_2}({\boldsymbol f})
\right)
-
\int_0^1
\left[
f_1(u)\log\frac{f_1(u)}p
+
f_2(u)\log\frac{f_2(u)}{1-p}
\right]du
\right\},
\]
and apply \cref{lmm1.4ext} part (i) with
\[
c=2,\qquad
\phi=\mathbf 1_1-\mathbf 1_2,\qquad
\beta=\theta,\qquad
\mu=(p,1-p),
\]
to obtain, for every $\theta\in\mathbb R$,
\[
\frac{d}{d\theta}{\mathscr Z}^{(\rm op)}(\theta,p)
=
u_{\theta,p}^2-(1-u_{\theta,p})^2
=
2u_{\theta,p}-1.
\]
% Since ${\mathscr Z}^{(\rm op)}(\cdot,p)$ is convex and
% $\theta\mapsto 2u_{\theta,p}-1$ is continuous, the above identity holds
% for every $\theta\in\mathbb R$.

Moreover, the map
$\theta\longmapsto 2u_{\theta,p}-1$
is a bijection from $\mathbb R$ onto $(-1,1)$. Thus, for every
$y\in(-1,1)$, there exists a unique $\theta_y\in\mathbb R$ such that
$y=2u_{\theta_y,p}-1$.
Applying \cref{lmm1.4ext} part (ii), the unique optimizer of the
constrained problem defining $I_{W_\infty,(1,-1)}(y)$ is
\[
(f_1,f_2)
=
(u_{\theta_y,p},1-u_{\theta_y,p})
=
\left(\frac{1+y}{2},\frac{1-y}{2}\right).
\]
Therefore, we obtain
\[
I_{W_\infty,(1,-1)}(y)
=
\left(\frac{1+y}{2}\right)
\log\left(\frac{1+y}{2p}\right)
+
\left(\frac{1-y}{2}\right)
\log\left(\frac{1-y}{2(1-p)}\right).
\]
\end{itemize}
\end{proof}

\subsection{Proofs for the Auxiliary Results}\label{AuxiliarySection}
% \subsubsection{Proofs for \cref{ER Hypergraph}}
% \subsubsection{Proofs for \cref{GPM}}
\begin{proof}[Proof of \cref{equivalence of two cut norms}]
\cref{defcutstar} trivially implies $\|W\|_{\square^*}\le  \|W\|_{\square}$, so we only show the other inequality.

\textbf{Step 1: An inequality for $v$-tensors.}

First we claim there exists a constant $C'_v>0$ such that for any symmetric zero-diagonal $v$-tensor $A$, %we have 
\begin{align}\label{discmultitosingle}
\sup_{B_1,\ldots,B_v \subseteq[n]}\left|\mathscr{E}(B_1,\dots,B_v)\right|
\le
C'_v\,\sup_{S\subseteq[n]} |\mathscr{M}(S)|,
\end{align}
where for $S\subseteq[n]$
\[
\mathscr{M}(S):=\sum_{(i_1,\dots,i_v)\in S^v} A(i_1,\dots,i_v),\]
and for $B_1,\dots,B_v\subseteq[n]$
\[
\mathscr{E}(
B_1,\dots,B_v):=\sum_{(i_1,\dots,i_v)\in B_1\times\cdots\times B_v} A(i_1,\dots,i_v).
\]
Given $B_1,\dots,B_v\subseteq[n]$, for each 
$\boldsymbol{\iota}\in \{0,1\}^v$ define 

\[ \mathfrak{B}_{\boldsymbol{\iota}}:=  \Big(\bigcap_{r: \iota_r=1} B_r\Big)\cap\Big(\bigcap_{r: \iota_r=0}B_r^c\Big).\]
% \[
% C_\Upsilon := \Big(\bigcap_{j\in\Upsilon} B_j\Big)\cap\Big(\bigcap_{j\notin\Upsilon}B_j^c\Big).
% \]
Then $\{\mathfrak B_{\boldsymbol{\iota}}\}_{\boldsymbol{\iota}\in \{0,1\}^v}$ are pairwise disjoint with $B_s=\bigsqcup_{\boldsymbol{\iota}:\iota_s=1} \mathfrak B_{\boldsymbol{\iota}}$,
for $s\in [v]$.
By multilinearity of $\mathscr{E}(\cdot)$, we have the following identity
\begin{equation*}\label{eq:atom-expand}
\mathscr{E}(B_1,\dots,B_v)
=
\sum_{{\boldsymbol{\iota}^{(1)}}:{{\iota}^{(1)}_1}=1}\cdots\sum_{{\boldsymbol{\iota}^{(v)}}:{{\iota}^{(v)}_v}=1}
\mathscr{E}( \mathfrak{B}_{\boldsymbol{\iota}^{(1)}},\dots, \mathfrak{B}_{\boldsymbol{\iota}^{(v)}}).
\end{equation*}
There are at most $(2^{v-1})^v$ terms and each pair $\{ \mathfrak{B}_{\boldsymbol{\iota}^{(i)}}, \mathfrak{B}_{\boldsymbol{\iota}^{(j)}}\}$ is either disjoint or identical,
so
% \begin{equation*}
% |\mathscr{E}(B_1,\dots,B_v)|
% \le
% 2^{v(v-1)}\max_{\Upsilon_1\ni1,\dots,\Upsilon_v\ni v}
% |\mathscr{E}(C_{\Upsilon_1},\dots,C_{\Upsilon_v})|.
% \end{equation*}
% Notice that each pair $\{C_{\Upsilon_i},C_{\Upsilon_j}\}$ is either disjoint or identical,
% so
\begin{equation}\label{eq:triangle-atoms}
|\mathscr{E}(B_1,\dots,B_v)|
\le
2^{v(v-1)}\underset{\substack{F_i\text{s} \\ \text{either disjoint}\\\text{or identical}}}{\max}
|\mathscr{E}(F_1,\dots,F_v)|.
\end{equation}
To complete the proof,  we state and prove the following claims. 

% \begin{lmm}[Random-labeling disjointification]\label{Random-labeling disjointification}
% Suppose $F\subseteq[n]$ appears in exactly $1\le d\le v$ coordinate slots of
% a rectangle sum and $|F|\ge d$. Let the other coordinate sets $E_{d+1},\dots,E_v$ be disjoint from $F$. Then there exists a disjoint partition $F=F_1\sqcup\cdots\sqcup F_d$ with nonempty $F_i$s
% such that
% \begin{equation*}\label{eq:label}
% |\mathscr{E}(F,\dots,F,E_{d+1},\dots,E_v)| \le d^d\left|\mathscr{E}(F_1,\dots,F_d,E_{d+1},\dots,E_v)\right|.
% \end{equation*}
% % In particular, the RHS is a rectangle sum with pairwise disjoint sets.
% \end{lmm}
\begin{lmm}[Random-labeling disjointification]
\label{Random-labeling disjointification}
Suppose $F\subseteq[n]$ appears in exactly $1\leq d\leq v$
coordinate slots of a rectangle sum. Let the other coordinate sets
$E_{d+1},\ldots,E_v$ be disjoint from $F$. Then there exist pairwise
disjoint sets $F_1,\ldots,F_d\subseteq F$ such that
$F=\bigsqcup_{a=1}^dF_a$ and
\[
\left|
\mathscr E(F,\ldots,F,E_{d+1},\ldots,E_v)
\right|
\leq
d^d
\left|
\mathscr E(F_1,\ldots,F_d,E_{d+1},\ldots,E_v)
\right|.
\]
\end{lmm}
\begin{proof}[Proof of \cref{Random-labeling disjointification}]
% First, notice that for the case $|F|<d$, the zero-diagonal assumption immediately implies that  
% $\mathscr{E}(F,\dots,F,E_{d+1},\dots,E_v)=0$ 
% and there is nothing to prove, so we assume $|F|\ge d$.
Assume $\mathscr{E}(F,\dots,F,E_{d+1},\dots,E_v)\neq 0$, otherwise we have nothing to prove. Independently assign each $x\in F$ a random label $\ell(x)\in[d]$,
and set $F_a:=\{x\in F:\ell(x)=a\}$.
Then $F=\bigsqcup_{a=1}^d F_a$.
Now if $\mathbb{E}$ denotes the expectation w.r.t this random assignment, we have
\begin{align*}
 &\mathbb{E}\Big(\mathscr{E}(F_1,\dots,F_d,E_{d+1},\dots,E_v)\Big)=\\
 &\mathbb{E}\Big(\sum_{(i_1,\dots,i_v)\in F\times\cdots\times F\times E_{d+1}\times\ldots \times E_{v}} f(i_1,\dots,i_v)\times1\{{i_1}\in F_1\}\times\ldots\times 1\{ i_d \in F_d\}\Big)   \\
 & =\sum_{(i_1,\dots,i_v)\in F\times\cdots\times F\times E_{d+1}\times\ldots \times E_{v}} f(i_1,\dots,i_v)\times\mathbb{P}\{{i_1}\in F_1\}\times\ldots\times \mathbb{P}\{ i_d \in F_d\}\\
 & =d^{-d} \mathscr{E}(F,\dots,F,E_{d+1},\dots,E_v).
\end{align*}
Therefore,
\[\mathscr{E}(F,\dots,F,E_{d+1},\dots,E_v) = d^d \mathbb{E}\left(\mathscr{E}(F_1,\dots,F_d,E_{d+1},\dots,E_v)\right),\]
meaning that
\begin{equation}
|\mathscr{E}(F,\dots,F,E_{d+1},\dots,E_v)| \le d^d \max_{F = F_1 \sqcup \cdots \sqcup F_d}\left|\mathscr{E}(F_1,\dots,F_d,E_{d+1},\dots,E_v)\right|.
\end{equation}
%\reihaneh{NOT NEEDED!}
%\reihaneh{Notice that the maximum above is attained at a partition$F=\bigsqcup_{i=1}^d F_i$ for which all the sets $F_i$ are nonempty.Indeed, the maximum is strictly positive, since $\mathscr{E}(F,\dots,F,E_{d+1},\dots,E_v)\neq 0$.On the other hand, if a maximizing partition had $F_i=\varnothing$for some $i$, then$\mathscr{E}(F_1,\dots,F_d,E_{d+1},\dots,E_v)=0$,contradicting the strict positivity of the maximum. Therefore, the maximum is achieved by some partition with nonempty sets.}
\end{proof}

Invoking \cref{Random-labeling disjointification} at most $v$ times and noticing that each set can appear at most $v$ times, we obtain
\begin{align}\label{rndmatoms}
&\underset{\substack{F_i\text{s} \\ \text{either disjoint}\\\text{or identical}}}{\max}
|\mathscr{E}(F_1,\dots,F_v)| \le (v^v)^v \underset{\substack{D_1,\dots,D_v\\ \text{pairwise disjoint }}}{\max}
|\mathscr{E}(D_1,\dots,D_v)|.
\end{align}
% \underset{\substack{{\boldsymbol k} \in [n]^v \\ k_i \text{s distinct}}}{\sum}
% Notice that for the case $|F|<d$, the zero-diagonal assumption immediately implies that  
% $\mathscr{E}(F,\dots,F,E_{d+1},\dots,E_v)=0$ 
%  and there is nothing to check. 
The next technical lemma provides an upper bound on the RHS of~\eqref{rndmatoms}.
\begin{lmm}[Disjoint Sets]\label{Disjoint Sets}
If $D_1,\dots,D_v\subseteq[n]$ are pairwise disjoint, then
\begin{equation}\label{eq:disjoint-pol}
\mathscr{E}(D_1,\dots,D_v)
=
\frac{1}{v!}\sum_{J\subseteq[v]} (-1)^{v-|J|}
\,\mathscr{M}\Big(\bigcup_{j\in J} D_j\Big).
\end{equation}
Consequently,
\begin{equation}\label{eq:disjoint-bound}
|\mathscr{E}(D_1,\dots,D_v)|\le \frac{2^v}{v!}\,\sup_{S\subseteq[n]} |\mathscr{M}(S)|.
\end{equation}
\end{lmm}
\begin{proof}[Proof of \cref{Disjoint Sets}]
Write $S_J:=\bigcup_{j\in J} D_j$.
Expand $\mathscr{M}(S_J)$ and interchange sums:
\[
\frac{1}{v!}\sum_{J\subseteq[v]} (-1)^{v-|J|}\mathscr{M}(S_J)
=
\sum_{\boldsymbol i\in[n]^v} f(\boldsymbol i)\cdot c(\boldsymbol i),
\qquad \boldsymbol i=(i_1,\dots,i_v),
\]
where
\[
c(\boldsymbol i):=\frac{1}{v!}\sum_{J\subseteq[v]} (-1)^{v-|J|}\mathbf 1_{\{\boldsymbol i\in S_J^v\}}.
\]
If $\boldsymbol i$ has a collision then $f(\boldsymbol i)=0$, so assume $i_1,\dots,i_v$ are distinct.
Since $D_1,\dots,D_v$ are disjoint, each $i_r$ belongs to at most one $D_s$.
Define
$I(\boldsymbol i):=\{s\in[v]:\exists\,r\in[v]\text{ with } i_r\in D_s\}$.
Then $\mathbf 1_{\{\boldsymbol i\in S_J^v\}}=1$ if and only if $I(\boldsymbol i)\subseteq J$, so
\[
c(\boldsymbol i)=\frac{1}{v!}\sum_{J:\,I(\boldsymbol i)\subseteq J\subseteq[v]} (-1)^{v-|J|}
=\frac{(-1)^{v-|I(\boldsymbol i)|}}{v!}\sum_{T\subseteq[v]\setminus I(\boldsymbol i)} (-1)^{|T|}.
\]
Using $\sum_{T\subseteq E}(-1)^{|T|}=(1-1)^{|E|}$, we get $c(\boldsymbol i)=0$ unless
$|I(\boldsymbol i)|=v$, in which case $c(\boldsymbol i)=1/v!$.
Thus only tuples $\boldsymbol i$ that hit every $D_1,\dots,D_v$ survive; since the $D_s$ are disjoint
and there are $v$ coordinates, this means exactly one coordinate in each $D_s$.

For each fixed choice of distinct elements $d_s\in D_s$, there are $v!$ permutations of the
coordinates. By symmetry of $f$ all these permutations contribute the same value, and the
factor $1/v!$ cancels the overcounting, leaving precisely $\mathscr{E}(D_1,\dots,D_v)$.
This proves \eqref{eq:disjoint-pol}. The bound \eqref{eq:disjoint-bound} follows
by triangle inequality.
\end{proof}
Now combining \eqref{eq:triangle-atoms}, \eqref{rndmatoms}, and \eqref{eq:disjoint-bound} we get
\[
|\mathscr{E}(B_1,\dots,B_v)|
\le \frac{(2v)^{v^2}}{v!}\,\sup_{S\subseteq[n]} |\mathscr{M}(S)|.
\]
This holds for all $B_1,\dots,B_v\subseteq[n]$, completing the proof of our claim~\eqref{discmultitosingle} with $C'_v=(2v)^{v^2}/{v!}$.

\textbf{Step 2: Proof of $C_v\|W\|_{\square}\le\|W\|_{\square^*}$  for $v$-hypergraphons.}

% Now we show $C_v\|W\|_{\square}\le\|W\|_{\square^*}$ for some $C_v>0$, to be specified later.
For each $n\in\mathbb N$, partition $[0,1]^v$ into the cubes
\[
\mathcal{C}_{\boldsymbol k,n}
:=\prod_{r=1}^v \Big[\frac{k_r-1}{\mathfrak n},\frac{k_r}{\mathfrak n}\Big),
\qquad
\boldsymbol k\in[\mathfrak n]^v \qquad \mathfrak{n}:=2^n.
\]
Define the step function
\[
Y^{(n)}(x_1,\ldots,x_v):=\sum_{\boldsymbol k\in[\mathfrak n]^v }y_{{\boldsymbol k}, n} 1\{\boldsymbol x \in \mathcal{C}_{\boldsymbol k,n}\}
,
\]
where $y_{{\boldsymbol k}, n}:=
\int_{\mathcal{C}_{\boldsymbol k,n}} W(u)\,du/\lambda(\mathcal{C}_{\boldsymbol k,n})$, and $\lambda$ denotes the Lebesgue measure.
Equivalently, $Y^{(n)}=\mathbb E[W \mid \mathcal F_n]$ $\lambda$-a.e., where 
$\mathcal F_n$ is the $\sigma$-algebra generated by this partition. 
Since $(\mathcal F_n)$ increases to the Borel $\sigma$-algebra of $[0,1]^v$, 
the martingale convergence theorem yields
$\|Y^{(n)}-W\|_1 \longrightarrow 0$, and hence $\|Y^{(n)}-W\|_{\square}\longrightarrow 0$.
Also define
\[\tilde Y^{(n)}(x_1,\ldots,x_v):=\underset{\substack{{\boldsymbol k} \in [\mathfrak n]^v \\ k_r \text{s distinct}}}{\sum}y_{{\boldsymbol k}, n}
 1\{\boldsymbol x \in \mathcal{C}_{\boldsymbol k,n}\},\]
 and notice that by construction $Y^{(n)}$ and $\tilde Y^{(n)}$ are symmetric.
For fixed $n$, we can write
\begin{align*}
    &\|\tilde Y^{(n)}\|_{\square}=\sup _{T_1,\ldots,T_v \subseteq[0,1]} \Big|\int_{  T_1\times\ldots\times T_v} \tilde Y^{(n)}\Big(x_1,\ldots, x_v\Big) \prod_{r=1}^v d x_r\Big|\\
    &=\sup _{T_1,\ldots,T_v \subseteq[0,1]} \Big|\underset{\substack{{\boldsymbol k} \in [\mathfrak n]^v \\ k_r \text{s distinct}}}{\sum} y_{{\boldsymbol k}, n}\lambda(T_1\times\ldots\times T_v\cap \mathcal{C}_{{\boldsymbol k},n})\Big|\\
    &= \sup _{T_1,\ldots,T_v \subseteq[0,1]} \Big|\underset{\substack{{\boldsymbol k} \in [\mathfrak n]^v \\ k_r \text{s distinct}}}{\sum} y_{{\boldsymbol k}, n}\prod_{r=1}^v\lambda\Big(T_r\cap\Big[\frac{k_r-1}{\mathfrak n},\frac{k_r}{\mathfrak n}\Big)\Big)\Big|\\
    &= \mathfrak n^{-v}\sup _{T_1,\ldots,T_v \subseteq[0,1]} \Big|\underset{\substack{{\boldsymbol k} \in [\mathfrak n]^v \\ k_r \text{s distinct}}}{\sum} y_{{\boldsymbol k}, n}\prod_{r=1}^v\Bigg[\mathfrak n\lambda\Big(T_r\cap\Big[\frac{k_r-1}{\mathfrak n},\frac{k_r}{\mathfrak n}\Big)\Big)\Bigg]\Big|\\
    &=\mathfrak n^{-v}\sup _{S_1,\ldots,S_v \subseteq[\mathfrak n]} \Big|\underset{\substack{{\boldsymbol k} \in [\mathfrak n]^v \\ k_r \text{s distinct}}}{\sum} y_{{\boldsymbol k}, n}\prod_{r=1}^v\Bigg[1\{k_r\in S_r\}\Bigg]\Big|,
\end{align*}
where in the last line we used \eqref{DtC1} from \cref{DtC}, stated and proved in \cref{AuxiliarySection}. Now using \eqref{discmultitosingle},
\begin{align*}
    &\le C'_v\mathfrak n^{-v}\sup _{S \subseteq[\mathfrak n]} \Big|\underset{\substack{{\boldsymbol k} \in S^v \\ k_r \text{s distinct}}}{\sum} y_{{\boldsymbol k}, n}\Big|=C'_v\mathfrak n^{-v}\sup_{\beta_1,\ldots,\beta_{\mathfrak n}\in \{0,1\}} \Big|\underset{\substack{{\boldsymbol k} \in [\mathfrak n]^v \\ k_r \text{s distinct}}}{\sum} y_{{\boldsymbol k}, n} \prod_{r=1}^v\beta_{k_r} \Big|\\
    &  =C'_v\mathfrak n^{-v}\sup_{\tau_1,\ldots,\tau_n \in [0,1]}\Big| \underset{\substack{{\boldsymbol k} \in [\mathfrak n]^v \\ k_r \text{s distinct}}}{\sum} y_{{\boldsymbol k}, n}\prod_{r=1}^v \tau_{k_r}\Big|,
\end{align*}
where the last equality is justified by \eqref{DtC2}. It is equal to
\begin{align*}
    & C'_v\mathfrak n^{-v}\sup _{T \subseteq[0,1]} \Big|\underset{\substack{{\boldsymbol k} \in [\mathfrak n]^v \\ k_r \text{s distinct}}}{\sum} y_{{\boldsymbol k}, n}\prod_{r=1}^v\Bigg[\mathfrak n\lambda\Big(T\cap\Big[\frac{k_r-1}{\mathfrak n},\frac{k_r}{\mathfrak n}\Big)\Big)\Bigg]\Big|\\
    &=C'_v\sup _{T \subseteq[0,1]} \Big|\underset{\substack{{\boldsymbol k} \in [\mathfrak n]^v \\ k_r \text{s distinct}}}{\sum} y_{{\boldsymbol k}, n}\lambda(T^v\cap \mathcal{C}_{{\boldsymbol k},n})\Big|=C'_v\sup _{T \subseteq[0,1]} \Big|\int_{  T^v} \tilde Y^{(n)}\Big(x_1,\ldots, x_v\Big) \prod_{r=1}^v d x_r\Big|=C'_v\|\tilde Y^{(n)}\|_{\square^*}.
\end{align*}

Recall that $\lim_{n\to\infty}\|W-Y^{(n)}\|_{\square}=0$ and some standard measure-theoretic arguments show $\lim_{n\to \infty}\|\tilde Y^{(n)}-Y^{(n)}\|_1 =0$ (recall that $W, Y^{(n)}\in L^1([0,1]^v)$), therefore $\lim_{n\to \infty}\|\tilde Y^{(n)}-Y^{(n)}\|_{\square}=0$. Now
\begin{align*}\label{}
& \|W\|_{\square}\le \lim_{n\to\infty}\|\tilde Y^{(n)}\|_{\square}+\lim_{n\to\infty}\|W-Y^{(n)}\|_{\square}+\lim_{n\to\infty}\|\tilde Y^{(n)}-Y^{(n)}\|_{\square}\le C'_v\lim_{n\to\infty}\|\tilde Y^{(n)}\|_{\square^*}  = C'_v\| W\|_{\square^*},  
\end{align*}
where in the last line we use $\lim_{n\to\infty}\|\tilde Y^{(n)}-W\|_{\square^*}\le\lim_{n\to\infty}\|\tilde Y^{(n)}-W\|_{\square}=0$. So with $C_v:=1/C'_v$, we are done.
\end{proof}

\begin{proof}[Proof of \cref{bijection}]
Let \({\boldsymbol f}=(f_1,\ldots,f_c)\in\mathcal F_c\). Define a probability measure
\(\nu_{\boldsymbol f}\) on \([0,1]\times[c]\) by
\[
\nu_{\boldsymbol f}(S\times\{r\})
:=
\int_S f_r(x)\,dx,
\qquad S\subseteq[0,1]\ \text{measurable},\ r\in[c].
\]
Then $\nu_{\boldsymbol f}\in\mathcal P$ and
$\nu_{\boldsymbol f}(B=r\mid A=x)=f_r(x)$, proving surjectivity.
If $\boldsymbol f^\nu=\boldsymbol f^{\tilde\nu}$, then for all measurable
$S\subseteq[0,1]$ and $r\in[c]$,
\[
\nu(S\times\{r\})=\int_S f_r^\nu(x)\,dx
=\int_S f_r^{\tilde\nu}(x)\,dx=\tilde\nu(S\times\{r\}),
\]
so $\nu=\tilde\nu$. Hence the map is injective. Also, using the law of iterated expectations for $W\in \mathcal{W}_v$, we obtain $T_{W,\phi}(\nu)=G_{W,\phi}({\boldsymbol f}^\nu)$.

\end{proof}

\begin{proof}[Proof of \cref{functionalcut}]
The inequality \(\le\) follows by taking \(g_a=\mathbf 1_{S_a}\), so the RHS is at least
\(\|W\|_{\square}\).

For the reverse inequality, let \(0\le g_a\le 1\). We write
$g_a(x_a)=\int_0^1 \mathbf 1_{\{g_a(x_a)>t_a\}}\,dt_a$.
Hence, by Fubini,
\[
\int_{[0,1]^v}
W(x_1,\ldots,x_v)\prod_{a=1}^v g_a(x_a)\,dx
=
\int_{[0,1]^v}
\left(
\int_{\prod_{a=1}^v \{g_a>t_a\}}
W(x_1,\ldots,x_v)\,dx
\right)
dt_1\cdots dt_v .
\]
Therefore,
\[
\left|
\int_{[0,1]^v}
W(x_1,\ldots,x_v)\prod_{a=1}^v g_a(x_a)\,dx
\right|
\le
\int_{[0,1]^v}
\|W\|_{\square}\,dt_1\cdots dt_v
=
\|W\|_{\square}.
\]
Taking the supremum over all \(0\le g_1,\ldots,g_v\le1\) gives the result.
\end{proof}
\subsubsection{Proofs for \cref{ER Hypergraph}}
% \begin{proof}[Proof of \cref{DtC}]
%   We only prove the first equality, the latter can be shown similarly.  
% Fix $\tau^{\ell}_1, \ldots,\tau^{\ell}_n \in [0,1]$, for all $\ell \in [v]$. Define independent random variables $\zeta^{\ell}_i\sim \mathrm{Ber}(\tau^{\ell}_i)$ for $i \in [n], \ell \in [v]$. If $\mathbb{E}$ denotes the expectation w.r.t random variables $\{\zeta^{\ell}_1\ldots\zeta^{\ell}_n: \ell\in[v]\}$, we have  
% \begin{align*}
% & \left|\sum_{(i_1,\ldots,i_v)\in [n]^v}D_n(i_1,\ldots,i_v) \prod_{r=1}^v \tau^{r}_{i_r} \right|=\left|\E\Bigg[\sum_{(i_1,\ldots,i_v)\in[n]^v}D_n(i_1,\ldots,i_v) \prod_{r=1}^v\zeta^{r}_{i_r} \Bigg]\right|\\
% &\le \E\left|\Bigg[\sum_{(i_1,\ldots,i_v)\in [n]^v}D_n(i_1,\ldots,i_v) \prod_{r=1}^v\zeta^{r}_{i_r} \Bigg]\right|\le\sup_{\{{\beta_i}^\ell\in \{0,1\} :i\in [n], \ell \in [v]\}} \left|\sum_{(i_1,\ldots,i_v)\in [n]^v}D_n(i_1,\ldots,i_v) \prod_{r=1}^v\beta^r_{i_r} \right|.
% \end{align*}
% These inequalities are valid for all choices of $\tau^{\ell}_1, \ldots,\tau^{\ell}_n \in [0,1]$ for $\ell \in [v]$, so we have 
% \begin{align*}
% &\sup_{\{{\tau_i}^\ell\in [0,1] :i\in [n], \ell \in [v]\}}\left| \sum_{(i_1,\ldots,i_v)\in [n]^v}D_n(i_1,\ldots,i_v)\prod_{r=1}^v \tau^r_{i_r}\right|\le\sup_{\{{\beta_i}^\ell\in \{0,1\} :i\in [n], \ell \in [v]\}} \left|\sum_{(i_1,\ldots,i_v)\in [n]^v}D_n(i_1,\ldots,i_v) \prod_{r=1}^v\beta^r_{i_r} \right|.   
% \end{align*}
% The reverse inequality is obvious, so we are done.
% \end{proof}
\begin{proof}[Proof of \cref{DtC}]
We first prove \eqref{DtC1}. Fix
$(\tau_i^\ell)_{i\in[n],\,\ell\in[v]}\in[0,1]^{nv}$,
and let
\[
\zeta_i^\ell\sim \operatorname{Ber}(\tau_i^\ell),
\qquad i\in[n],\ \ell\in[v],
\]
be mutually independent. Then, for every
\((i_1,\ldots,i_v)\in[n]^v\),
$\mathbb E\left[\prod_{r=1}^v\zeta_{i_r}^r\right]
=
\prod_{r=1}^v\tau_{i_r}^r$.
Therefore,
\begin{align*}
&\left|
\sum_{(i_1,\ldots,i_v)\in[n]^v}
D_n(i_1,\ldots,i_v)
\prod_{r=1}^v\tau_{i_r}^r
\right|=
\left|
\mathbb E\left[
\sum_{(i_1,\ldots,i_v)\in[n]^v}
D_n(i_1,\ldots,i_v)
\prod_{r=1}^v\zeta_{i_r}^r
\right]
\right|\\
&\le
\mathbb E\left|
\sum_{(i_1,\ldots,i_v)\in[n]^v}
D_n(i_1,\ldots,i_v)
\prod_{r=1}^v\zeta_{i_r}^r
\right|\le
\sup_{(\beta_i^\ell)\in\{0,1\}^{nv}}
\left|
\sum_{(i_1,\ldots,i_v)\in[n]^v}
D_n(i_1,\ldots,i_v)
\prod_{r=1}^v\beta_{i_r}^r
\right|.
\end{align*}
Taking the supremum over
\((\tau_i^\ell)\in[0,1]^{nv}\) gives one direction of
\eqref{DtC1}. The reverse inequality follows immediately from
\(\{0,1\}^{nv}\subseteq[0,1]^{nv}\).

We now prove \eqref{DtC2}. Fix
$(\tau_1,\ldots,\tau_n)\in[0,1]^n$,
and let
\[
\zeta_i\sim\operatorname{Ber}(\tau_i),
\qquad i\in[n],
\]
be independent. Since \(D_n\) is zero-diagonal, every nonzero
summand has \(i_1,\ldots,i_v\) all distinct. Consequently,
$\mathbb E\left[\prod_{r=1}^v\zeta_{i_r}\right]
=
\prod_{r=1}^v\tau_{i_r}$
for every nonzero summand. Hence,
\begin{align*}
&\left|
\sum_{(i_1,\ldots,i_v)\in[n]^v}
D_n(i_1,\ldots,i_v)
\prod_{r=1}^v\tau_{i_r}
\right|=
\left|
\mathbb E\left[
\sum_{(i_1,\ldots,i_v)\in[n]^v}
D_n(i_1,\ldots,i_v)
\prod_{r=1}^v\zeta_{i_r}
\right]
\right|\\
&\le
\mathbb E\Big|
\sum_{(i_1,\ldots,i_v)\in[n]^v}
D_n(i_1,\ldots,i_v)
\prod_{r=1}^v\zeta_{i_r}
\Big|\le
\sup_{(\beta_1,\ldots,\beta_n)\in\{0,1\}^n}
\Big|
\sum_{(i_1,\ldots,i_v)\in[n]^v}
D_n(i_1,\ldots,i_v)
\prod_{r=1}^v\beta_{i_r}
\Big|.
\end{align*}
Taking the supremum over
\((\tau_1,\ldots,\tau_n)\in[0,1]^n\) gives one direction of
\eqref{DtC2}. The reverse inequality follows from
\(\{0,1\}^n\subseteq[0,1]^n\).
\end{proof}
\subsubsection{Proofs for \cref{SIERgraph}}
\begin{proof}[Proof of \cref{lem:uniform-core-reduction-ER}]
Write $p:=p_n$. We prove the result by removing the
vertices of $V(H)\setminus V(\widetilde H)$ one at a time in the leaf-removal procedure as described after \cref{corH}. Let $K,K^{\prime}$ be two consecutive intermediate graphs along this process. Then
$K,K^{\prime}$ both have $2$-core $\widetilde H$, and there exists $u\in V(K)$ such that $V(K)=V(K^{\prime})\cup \{u\}$. Since $u$ is a leaf in $K$ (by construction), there exists $u'\in V(K^{\prime})$ such that $u'$ is the unique neighbor of $u$ in $K$.
Also,
$|V(K^{\prime})|=v_K-1$ and $|E(K^{\prime})|=e_K-1$.
We first show that
\begin{align}
\sup_{S\subseteq[n]}n^{-v_K}
\left|
p^{-e_K}\mathcal N_K(S;\mathcal G_n)
-(|S|-v_K+1)p^{-(e_K-1)}
\mathcal N_{K^{\prime}}(S;\mathcal G_n)
\right|
=o_\mathbb P(1).
\label{eq:one-leaf-independent}
\end{align}

For $d\in S\subseteq[n]$, define
\[\mathcal R_{d,S}
:=
\left\{
\varphi:V(K^{\prime})\hookrightarrow S:
\varphi(u')=d
\right\},\quad
\mathcal N_{K^{\prime}}(S;\mathcal G_n,d)=\mathcal N(S;d)
:=
\sum_{\varphi\in\mathcal R_{d,S}}
\prod_{\{a,b\}\in E(K^{\prime})}
\mathcal G_n(\varphi(a),\varphi(b)).
\]
We regard
$\mathcal N(S;.)=\bigl(\mathcal N(S;d)\bigr)_{d=1}^n$
as a vector in $\mathbb R^n$. Now we use \cref{eq:rooted-L2} and \cref{lem:ER-centered-opnorm}.
% which give
% \begin{align}
% \sum_{d=1}^n U_{u,[n]}(d)^2
% =
% O_\mathbb P\left(
% n^{2v_K-3}p^{2e_K-2}
% \right)
% \label{eq:rooted-L2}
% \end{align}
% and
% \[
% \|D\|_{\mathrm{op}}
% =
% O_\mathbb P\left(\sqrt{np\log n}\right),
% \]
% where
% $D_{ij}
% :=
% \bigl(\mathcal G_n(i,j)-p\bigr)\mathbf 1\{i\neq j\}$.
\begin{lmm}\label{eq:rooted-L2} Under the setting of \cref{lem:uniform-core-reduction-ER}, let $K, K'$ be two consecutive  intermediate graphs in the leaf-removal procedure, such that $V(K)=V(K')\cup \{u\}$, where  $u$ is the chosen leaf as in the proof of \cref{lem:uniform-core-reduction-ER} above. Then, with $\mathcal N(\cdot;d)$ as defined above, we have
\begin{equation*}
\sum_{d=1}^n \mathcal N([n];d)^2
=
O_{\mathbb P}\left(n^{2v_K-3}p^{2e_K-2}\right).
\end{equation*}    
\end{lmm}
\begin{lmm}
\label{lem:ER-centered-opnorm}
Define the symmetric matrix \(D=(D_{ij})_{1\le i,j\le n}\) by $D_{ij}:=\big(\mathcal G_n(i,j)-p_n\big)\mathbf 1_{\{i\neq j\}}$.
% \[
% D_{ij}:=
% \begin{cases}
% \mathcal G_n(i,j)-p_n, & i\ne j,\\
% 0, & i=j.
% \end{cases}
% \]
If \(n p_n\gg \log n\), then
$\|D\|_{\mathrm{op}}
=
O_{\mathbb P}\left(\sqrt{np_n}\right)$.    
\end{lmm}
Now fix $S\subseteq[n]$. Each injective embedding of $K$ into $S$ is obtained uniquely by first choosing an injective map $\varphi:V(K^{\prime})\hookrightarrow S$ and then choosing the image $i_u \in S\setminus\varphi(V(K^{\prime}))$. Since
$u$ is adjacent only to $u'$, using the definition of $\mathcal N_K(S;\mathcal G_n)$ (see \eqref{defon}), we have
\begin{align}
\mathcal N_K(S;\mathcal G_n)
=
\sum_{\varphi:V(K^{\prime})\hookrightarrow S}
\left[
\prod_{\{a,b\}\in E(K^{\prime})}
\mathcal G_n(\varphi(a),\varphi(b))
\right]
\sum_{i_u\in S\setminus\varphi(V(K^{\prime}))}
\mathcal G_n(i_u,\varphi(u')).
\label{idk}
\end{align}
Using
$\mathcal G_n(i_u,\varphi(u'))
=
p+D_{i_u,\varphi(u')}$
and
\[
\sum_{i_u\in S\setminus\varphi(V(K^{\prime}))}
D_{i_u,\varphi(u')}
=
(D\mathbf 1_S)(\varphi(u'))
-
\sum_{z\in V(K^{\prime})}
D_{\varphi(z),\varphi(u')},
\]
we obtain
\begin{align*}
\sum_{i_u\in S\setminus\varphi(V(K^{\prime}))}
\mathcal G_n(i_u,\varphi(u'))
&=
p(|S|-v_K+1)
+(D\mathbf 1_S)(\varphi(u'))-
\sum_{z\in V(K^{\prime})}
D_{\varphi(z),\varphi(u')}.
\end{align*}
Substituting this identity into \eqref{idk} gives
\begin{align}
\mathcal N_K(S;\mathcal G_n)
&=
p(|S|-v_K+1)\mathcal N_{K^{\prime}}(S;\mathcal G_n)
+E^{(1)}_{K,S}-E^{(2)}_{K,S},
\label{est4}
\end{align}
where
\begin{align*}
E^{(1)}_{K,S}
&:=
\sum_{\varphi:V(K^{\prime})\hookrightarrow S}
\left[
\prod_{\{a,b\}\in E(K^{\prime})}
\mathcal G_n(\varphi(a),\varphi(b))
\right]
(D\mathbf 1_S)(\varphi(u')),
\end{align*}
and
\begin{align*}
E^{(2)}_{K,S}
&:=
\sum_{\varphi:V(K^{\prime})\hookrightarrow S}
\left[
\prod_{\{a,b\}\in E(K^{\prime})}
\mathcal G_n(\varphi(a),\varphi(b))
\right]
\sum_{z\in V(K^{\prime})}
D_{\varphi(z),\varphi(u')}.
\end{align*}

We first control $E^{(1)}_{K,S}$. By the definition of
$\mathcal N(S;d)$,
\[
E^{(1)}_{K,S}
=
\sum_{d=1}^n
\mathcal N(S;d)(D\mathbf 1_S)(d).
\]
Since all summands defining $\mathcal N(S;d)$ are nonnegative,
\[
0\leq \mathcal N(S;d)\leq \mathcal N([n];d),
\qquad d\in[n].
\]
Therefore, by the Cauchy--Schwarz inequality,
\begin{align*}
\sup_{S\subseteq[n]}|E^{(1)}_{K,S}|
&\leq
\sup_{S\subseteq[n]}
\|\mathcal N(S;.)\|_2\,
\|D\mathbf 1_S\|_2\leq
\|\mathcal N([n];.)\|_2\,
\|D\|_{\mathrm{op}}\sqrt n.
\end{align*}
Using \cref{eq:rooted-L2} and \cref{lem:ER-centered-opnorm}, we obtain $\sup_{S\subseteq[n]}|E^{(1)}_{K,S}|
=
O_\mathbb P (
n^{v_K-\frac12}
p^{e_K-\frac12}
).$
Consequently,
\begin{align*}
p^{-e_K}\sup_{S\subseteq[n]}|E^{(1)}_{K,S}|
&=
O_\mathbb P\left(
n^{v_K}\sqrt{\frac{1}{np}}
\right)
=o_\mathbb P(n^{v_K}),
\end{align*}
since
$np\gg\log n$.

We next control $E^{(2)}_{K,S}$. Since
$D_{\varphi(u'),\varphi(u')}=0$,
the term corresponding to $z=u'$ vanishes, and hence
\begin{align*}
E^{(2)}_{K,S}
&=
\sum_{\varphi:V(K^{\prime})\hookrightarrow S}
\left[
\prod_{\{a,b\}\in E(K^{\prime})}
\mathcal G_n(\varphi(a),\varphi(b))
\right]
\sum_{z\in V(K^{\prime})\setminus\{u'\}}
D_{\varphi(z),\varphi(u')}.
\end{align*}
Since
$|D_{ij}|\leq \mathcal G_n(i,j)+p$,
we have
\begin{align*}
|E^{(2)}_{K,S}|
&\leq
\sum_{z\in V(K^{\prime})\setminus\{u'\}}
\sum_{\varphi:V(K^{\prime})\hookrightarrow S}
\left[
\prod_{\{a,b\}\in E(K^{\prime})}
\mathcal G_n(\varphi(a),\varphi(b))
\right]\times
\left(
\mathcal G_n(\varphi(z),\varphi(u'))+p
\right).
\end{align*}
The displayed upper bound is monotone under inclusion of $S$.
Therefore,
\begin{align*}
\sup_{S\subseteq[n]}|E^{(2)}_{K,S}|
&\leq
\sum_{z\in V(K^{\prime})\setminus\{u'\}}
\sum_{\varphi:V(K^{\prime})\hookrightarrow[n]}
\left[
\prod_{\{a,b\}\in E(K^{\prime})}
\mathcal G_n(\varphi(a),\varphi(b))
\right]\times
\left(
\mathcal G_n(\varphi(z),\varphi(u'))+p
\right).
\end{align*}

For each fixed $z$ and $\varphi$, consider the expectation of the
corresponding summand. If $\{z,u'\}\in E(K^{\prime})$, then the additional
factor
$\mathcal G_n(\varphi(z),\varphi(u'))$ duplicates an edge indicator
already present in the product, and hence the expectation of that
term is $p^{e_K-1}$. If $\{z,u'\}\notin E(K^{\prime})$, its expectation is
$p^{e_K}$. The term containing the added constant $p$ also has
expectation $p^{e_K}$. Thus, for all sufficiently large $n$,
the expectation of each summand is at most $2p^{e_K-1}$.
Since $K$ is fixed and there are at most $n^{v_K-1}$ injective maps
$\varphi:V(K^{\prime})\hookrightarrow[n]$, it follows that
\[
\mathbb E\left[
\sup_{S\subseteq[n]}|E^{(2)}_{K,S}|
\right]
\leq
C_K n^{v_K-1}p^{e_K-1}\Rightarrow p^{-e_K}
\sup_{S\subseteq[n]}|E^{(2)}_{K,S}|
=
O_\mathbb P\left(\frac{n^{v_K-1}}{p}\right)
=o_{\mathbb P}(n^{v_K})
\]
because $np\to\infty$. 
% Therefore, by Markov's inequality,
% \[
% \sup_{S\subseteq[n]}|E^{(2)}_{K,S}|
% =
% O_\mathbb P\left(
% n^{v_K-1}p^{e_K-1}
% \right).
% \]
% % % % % % % Consequently,
%$p^{-e_K}
%\sup_{S\subseteq[n]}|E^{(2)}_{K,S}|
%=
%% O_\mathbb P\left(\frac{n^{v_K-1}}{p}\right)
% =o_{\mathbb P}(n^{v_K})$,
% because $np\to\infty$.
Recalling \eqref{est4}, this proves
\eqref{eq:one-leaf-independent}.
\\

Finally, with $\omega=v-\tilde{v}$  as in \cref{def:notation},
% choose an order
% $u_1,\ldots,u_\omega$
% in which the vertices of $V(H)\setminus V(\widetilde H)$ are removed
% as leaves. 
let $\{K_j: 0\le j\le \omega\}$ be the sequence of graphs produced by the leaf-removal process, such that  $K_0:=H$ and $K_\omega=\widetilde H$.
% \[
% K_0:=H,\qquad
% K_j:=K_{j-1}-u_j,\qquad 1\leq j\leq\omega.
% \]
Then we have $v_{K_j}=v-j$, and
$e_{K_j}=e-j$. For $0\leq j\leq\omega$ and $S\subseteq[n]$, define
$A_j(S):=
p^{-e_{K_j}}\mathcal N_{K_j}(S;\mathcal G_n).$
By \eqref{eq:one-leaf-independent}, for each
$0\leq j\leq\omega-1$,
\[
A_j(S)
=
(|S|-v+j+1)A_{j+1}(S)+R_{j,n}(S), \quad \mbox{where}\quad \sup_{S\subseteq[n]}|R_{j,n}(S)|
=
o_{\mathbb P}(n^{v-j}).
\]

Iterating these identities gives
\begin{align*}
A_0(S)
&=
\left[
\prod_{j=0}^{\omega-1}(|S|-v+j+1)
\right]A_\omega(S)+
\sum_{j=0}^{\omega-1}
\left[
\prod_{k=0}^{j-1}(|S|-v+k+1)
\right]R_{j,n}(S),
\end{align*}
where the product is understood to be $1$ when $j=0$.

Since $\omega$ is fixed and
$\bigl||S|-v+k+1\bigr|\leq n+v$
uniformly over $S\subseteq[n]$, we have, for each fixed $j$,
\begin{align*}
\sup_{S\subseteq[n]}
\left|
\prod_{k=0}^{j-1}(|S|-v+k+1)R_{j,n}(S)
\right|
&\leq
(n+v)^j
\sup_{S\subseteq[n]}|R_{j,n}(S)|=
o_{\mathbb P}(n^v).
\end{align*}
Since there are only finitely many such error terms, their sum is
$o_P(n^v)$. Therefore, uniformly over $S\subseteq[n]$,
\[
p^{-e}\mathcal N_H(S;\mathcal G_n)
=
\prod_{j=0}^{\omega-1}(|S|-v+j+1)
p^{-\widetilde e}
\mathcal N_{\widetilde H}(S;\mathcal G_n)
+o_{\mathbb P}(n^v).
\]
Finally,
$\prod_{j=0}^{\omega-1}(|S|-v+j+1)
=
(|S|-\widetilde v)_\omega.$
Hence,
\[
\sup_{S\subseteq[n]}
\frac1{n^v}
\left|
p^{-e}\mathcal N_H(S;\mathcal G_n)
-
(|S|-\widetilde v)_\omega
p^{-\widetilde e}
\mathcal N_{\widetilde H}(S;\mathcal G_n)
\right|
=o_{\mathbb P}(1),
\]
which proves the result.
\end{proof}

\begin{proof}[Proof of \cref{lem:uniform-core-discrepancy-ER}]
Write $p:=p_n$. We first consider subsets $S\subseteq [n]$ of linear size. Fix
$\rho>0$. We claim that, for every $\varepsilon>0$, there exist
constants $c_{\varepsilon,\rho},C_{\varepsilon,\rho}>0$ such that,
uniformly over all $S\subseteq[n]$ satisfying $|S|\geq \rho n$,
\begin{align}
&\mathbb{P}\left(
n^{-\widetilde v}
\left|
p^{-\widetilde e}
\mathcal{N}_{\widetilde H}(S;\mathcal{G}_n)
-(|S|)_{\widetilde v}
\right|>\varepsilon
\right) \leq
C_{\varepsilon,\rho}
\exp\left\{
-c_{\varepsilon,\rho}
|S|^2p^{\widetilde\Delta}\log(1/p)
\right\}.
\label{est5}
\end{align}

The upper-tail estimate follows from
\cref{lem:hom-to-inj-upper-tail}.
\begin{lmm}[Upper-tail control]\label{lem:hom-to-inj-upper-tail} Consider $\widetilde{H}$ defined in \cref{corH}. Assume that
  $  (n\log n)^{-1/\widetilde\Delta}\ll p_n\ll 1$.
Then, for every $\rho,\varepsilon>0$, there exist constants
$c_{\varepsilon,\rho},C_{\varepsilon,\rho}>0$ such that, for all $n$,
uniformly over all $S\subseteq[n]$ satisfying $|S|\ge \rho n$, we have
\[
\mathbb P\left(
    \mathcal N_{\widetilde{H}}(S;\mathcal G_n)>
    (1+\varepsilon)p_n^{\widetilde e}(|S|)_{\widetilde v}
\right)
\le
C_{\varepsilon,\rho}\exp\left\{
    -c_{\varepsilon,\rho, \widetilde H}
    |S|^2p_n^{\widetilde\Delta}\log(1/p_n)
\right\}.
\]
\end{lmm}
% Notice that since $H$ is connected and is not a tree, its 2-core $\widetilde{H}$ is
% nonempty, connected and has minimum degree at least $2$. Hence \cref{lem:hom-to-inj-upper-tail} applies with $F=\widetilde H$.

Without loss of generality assume
$\varepsilon\in (0,1/2)$. Since $(|S|)_{\widetilde v}\le n^{\widetilde v}$, we have
\begin{align*}
&\left\{
n^{-\widetilde v}
\left(
p^{-\widetilde e}\mathcal N_{\widetilde H}(S;\mathcal G_n)
-(|S|)_{\widetilde v}
\right)
>\varepsilon
\right\}\subseteq
\left\{
\mathcal N_{\widetilde H}(S;\mathcal G_n)
>
(1+\varepsilon)p^{\widetilde e}(|S|)_{\widetilde v}
\right\},
\end{align*}
and
\begin{align*}
&\left\{
n^{-\widetilde v}
\left(
(|S|)_{\widetilde v}
-p^{-\widetilde e}\mathcal N_{\widetilde H}(S;\mathcal G_n)
\right)
>\varepsilon
\right\}\subseteq
\left\{
\mathcal N_{\widetilde H}(S;\mathcal G_n)
<
(1-\varepsilon)p^{\widetilde e}(|S|)_{\widetilde v}
\right\}.
\end{align*}

For the lower-tail event, we use the standard subgraph-count
lower-tail estimate of \cite[Theorem 3]{jansonwarnke2016lower},
% equivalently, it follows from Janson's inequality
% \cite[Theorem 2.14]{jansonluczakrucinski2000}, together with the
% usual overlap estimate. 
applied to the graph $\mathcal G_n$ induced on $S$,
it gives
\[
\mathbb P\left(
\mathcal N_{\widetilde H}(S;\mathcal G_n)
<
(1-\varepsilon)p^{\widetilde e}(|S|)_{\widetilde v}
\right)
\le
C_\varepsilon\exp\left\{
-c_{\varepsilon}
\min_{\substack{J\subseteq\widetilde H\\ e_J\ge1}}
|S|^{v_J}p^{e_J}
\right\},
\]
where $v_J$ and $e_J$ denote the number of vertices and edges,
respectively, of $J$, and $C_\varepsilon>0$ is a constant depends on $\varepsilon$.
% For the lower tail, we use the standard subgraph-count lower-tail
% estimate of \cite[Theorem 3]{jansonwarnke2016lower}; equivalently, it
% follows from Janson's inequality
% \cite[Theorem 2.14]{jansonluczakrucinski2000}, together with the usual
% overlap estimate. Applied to the graph $\mathcal{G}_n$ induced on
% $S$, it gives
% \begin{align*}
% &\mathbb{P}\left(
% \mathcal{N}_{\widetilde H}(S;\mathcal{G}_n)
% <
% (1-\varepsilon')
% p^{\widetilde e}(|S|)_{\widetilde v}
% \right)\leq
% \exp\left\{
% -c_{\varepsilon'}
% \min_{\substack{J\subseteq\widetilde H\\e_J>0}}
% |S|^{v_J}p^{e_J}
% \right\},
% \end{align*}
% where $v_J$ and $e_J$ denote the number of vertices and edges,
% respectively, of $J$.
We now claim that, for $|S|\geq \rho n$,
\begin{align}
\min_{\substack{J\subseteq\widetilde H\\e_J\ge1}}
|S|^{v_J}p^{e_J}
\gg
|S|^2p^{\widetilde\Delta}\log(1/p).
\label{eq:lower-tail-overlap}
\end{align}
Indeed, fix a nonempty subgraph $J\subseteq\widetilde H$. It is enough
to show
\[
|S|^{v_J-2}p^{e_J-\widetilde\Delta}
\gg \log(1/p).
\]
Since $e_J\ge 1$, we have $v_J\geq2$.

If $e_J<\widetilde\Delta$, then
\[
|S|^{v_J-2}p^{e_J-\widetilde\Delta}
\geq
p^{-(\widetilde\Delta-e_J)}
\gg \log(1/p).
\]

If $e_J=\widetilde\Delta$, then $v_J\geq3$, and hence using $\log(1/p)=O(\log n)$ we have 
\[
|S|^{v_J-2}p^{e_J-\widetilde\Delta}
=
|S|^{v_J-2}
\geq c_\rho n
\gg \log(1/p).
\]
Finally, suppose $e_J>\widetilde\Delta$. Recall that
$p(n\log n)^{1/\widetilde\Delta}\longrightarrow\infty$.
Since $|S|\geq\rho n$, we have
\begin{align*}
|S|^{v_J-2}p^{e_J-\widetilde\Delta}
\gg
 n^{v_J-2}
\left(
(n\log n)^{-1/\widetilde\Delta}
\right)^{e_J-\widetilde\Delta}= 
n^{v_J-1-e_J/\widetilde\Delta}
(\log n)^{-(e_J-\widetilde\Delta)/\widetilde\Delta}.
\end{align*}
Since $J\subseteq\widetilde H$ and $\widetilde H$ has maximum degree
$\widetilde\Delta$, we have $e_J\leq {\widetilde\Delta v_J}/{2}$ implying that $v_J-1-{e_J}/{\widetilde\Delta}
\geq
{v_J}/{2}-1\ge 1/2$.
%\begin{align*}
 %   \Rightarrow 
%\end{align*}
% Moreover, since $e_J>\widetilde\Delta\geq2$, we have $v_J\geq3$.
% Consequently,
% $v_J-1-{e_J}/{\widetilde\Delta}\geq 1/2$.
% Thus,
% \[
% |S|^{v_J-2}p^{e_J-\widetilde\Delta}
% \gg
% c_\rho
% n^{1/2}
% (\log n)^{-(e_J-\widetilde\Delta)/\widetilde\Delta}
% \gg \log n.
% \]
% Since
% $p\gg(n\log n)^{-1/\widetilde\Delta}$,
% we also have $\log(1/p)=O(\log n)$. 
Therefore,
$|S|^{v_J-2}p^{e_J-\widetilde\Delta}
\gg \log(1/p)$. Since $\widetilde H$ is fixed, there are only finitely many subgraphs
$J\subseteq\widetilde H$, and hence
\eqref{eq:lower-tail-overlap} follows.
Therefore,
\begin{align*}
&\mathbb P\left(
n^{-\widetilde v}
\left(
(|S|)_{\widetilde v}
-p^{-\widetilde e}\mathcal N_{\widetilde H}(S;\mathcal G_n)
\right)
>\varepsilon
\right)\le
\mathbb P\left(
\mathcal N_{\widetilde H}(S;\mathcal G_n)
<
(1-\varepsilon)p^{\widetilde e}(|S|)_{\widetilde v}
\right)\le
C_{\varepsilon} \exp\left\{
-c_{\varepsilon,\rho}
|S|^2p^{\widetilde\Delta}\log(1/p)
\right\}.
\end{align*}
Together with the upper-tail estimate from \cref{lem:hom-to-inj-upper-tail}, this proves
\eqref{est5}.
% Therefore,
% \begin{align*}
% &\mathbb{P}\left(
% \mathcal{N}_{\widetilde H}(S;\mathcal{G}_n)
% <
% (1-\varepsilon)
% p^{\widetilde e}(|S|)_{\widetilde v}
% \right)\leq
% \exp\left\{
% -c_{\varepsilon}
% |S|^2p^{\widetilde\Delta}\log(1/p)
% \right\}.
% \end{align*}
% Together with the upper-tail estimate, this proves \eqref{est5},
% after adjusting the constants.

Taking a union bound over all $S\subseteq[n]$ with $|S|\geq\rho n$,
we obtain
\begin{align}\label{eq:large-sets}
&\mathbb{P}\left(
\sup_{\substack{S\subseteq[n]\\|S|\geq\rho n}}
n^{-\widetilde v}
\left|
p^{-\widetilde e}
\mathcal{N}_{\widetilde H}(S;\mathcal{G}_n)
-(|S|)_{\widetilde v}
\right|
>\varepsilon
\right)\leq
2^n C_{\varepsilon,\rho}
\exp\left\{
-c_{\varepsilon,\rho}n^2
p^{\widetilde\Delta}\log(1/p)
\right\}\to 0.
\end{align}
% We now verify that
% $np^{\widetilde\Delta}\log(1/p)\longrightarrow\infty$.
% If $p\leq n^{-1/(2\widetilde\Delta)}$, then
% $\log(1/p)\geq (2\widetilde\Delta)^{-1}\log n$,
% and hence
% \[
% np^{\widetilde\Delta}\log(1/p)
% \geq
% \frac{1}{2\widetilde\Delta}
% np^{\widetilde\Delta}\log n
% \longrightarrow\infty.
% \]
% If $p>n^{-1/(2\widetilde\Delta)}$, then
% \[
% np^{\widetilde\Delta}\log(1/p)
% \geq
% n^{1/2}\log(1/p)
% \longrightarrow\infty,
% \]
% where the last conclusion follows from $p\ll1$. Therefore, the
% right-hand side above converges to zero, and hence, for every fixed
% $\rho>0$,
% \begin{align}
% \sup_{\substack{S\subseteq[n]\\|S|\geq\rho n}}
% n^{-\widetilde v}
% \left|
% p^{-\widetilde e}
% \mathcal{N}_{\widetilde H}(S;\mathcal{G}_n)
% -(|S|)_{\widetilde v}
% \right|
% =o_P(1).
% \label{eq:large-sets}
% \end{align}

It remains to handle small sets. Let $S\subseteq[n]$ satisfy
$|S|<\rho n$, and choose $T\subseteq[n]$ such that $S\subseteq T$ and
$|T|=\lceil\rho n\rceil$. By monotonicity,
$\mathcal{N}_{\widetilde H}(S;\mathcal{G}_n)
\leq
\mathcal{N}_{\widetilde H}(T;\mathcal{G}_n)$.
Thus,
\begin{align*}
&n^{-\widetilde v}
\left(
p^{-\widetilde e}
\mathcal{N}_{\widetilde H}(S;\mathcal{G}_n)
-(|S|)_{\widetilde v}
\right)\leq
n^{-\widetilde v}
\left(
p^{-\widetilde e}
\mathcal{N}_{\widetilde H}(T;\mathcal{G}_n)
-(|T|)_{\widetilde v}
\right)+
n^{-\widetilde v}
\left(
(|T|)_{\widetilde v}
-
(|S|)_{\widetilde v}
\right)\\
&\qquad\leq
\sup_{\substack{R\subseteq[n]\\|R|\geq\rho n}}
n^{-\widetilde v}
\left|
p^{-\widetilde e}
\mathcal{N}_{\widetilde H}(R;\mathcal{G}_n)
-(|R|)_{\widetilde v}
\right|
+
C\rho^{\widetilde v},
\end{align*}
for all $n$. Moreover,
\[n^{-\widetilde v}
\left(
(|S|)_{\widetilde v}-p^{-\widetilde e}
\mathcal{N}_{\widetilde H}(S;\mathcal{G}_n)
\right)\le n^{-\widetilde v}
(|S|)_{\widetilde v}\le \rho^{\widetilde v}\]

Taking the supremum over all $S$ with $|S|<\rho n$ and combining this
with \eqref{eq:large-sets}, we obtain
\[
\sup_{S\subseteq[n]}
n^{-\widetilde v}
\left|
p^{-\widetilde e}
\mathcal{N}_{\widetilde H}(S;\mathcal{G}_n)
-(|S|)_{\widetilde v}
\right|
\leq
o_{\mathbb P}(1)+C\rho^{\widetilde v}.
\]
Since $\rho>0$ is arbitrary, letting $\rho\downarrow0$ proves the
claim.
\end{proof}

\begin{proof}[Proof of \cref{eq:rooted-L2}]
Note that $K$ has $2$-core
$\widetilde H$, and $\widetilde H$ has maximum degree
$\widetilde\Delta$. With $p=p_n$, we have
\begin{align*}
\E\Big[\sum_{d=1}^n\mathcal N([n];d)^2\Big]
&=
\sum_{d=1}^n
\sum_{\varphi_1,\varphi_2\in\mathcal R_{d,[n]}}
\E\left[
\prod_{\{i,j\}\in
E(\varphi_1(K^{\prime}))\cup E(\varphi_2(K^{\prime}))}
\mathcal G_n(i,j)
\right]=
\sum_{d=1}^n
\sum_{\varphi_1,\varphi_2\in\mathcal R_{d,[n]}}
p^{|E(\varphi_1(K^{\prime}))\cup E(\varphi_2(K^{\prime}))|}.
\end{align*}

We classify pairs $(\varphi_1,\varphi_2)$ according to their overlap. Let
$\varphi_1(K')$ and $\varphi_2(K')$ denote the labeled copies of $K'$
produced by the maps $\varphi_1,\varphi_2:V(K')\to[n]$. Let $K_{\cap}$ be the graph with
\[
V(K_{\cap})
=
V(\varphi_1(K'))
\cap
V(\varphi_2(K')),
\qquad
E(K_{\cap})
=
E(\varphi_1(K'))
\cap
E(\varphi_2(K')).
\]
We set $r
:=|V(K_{\cap})|$, $s
:=|E(K_{\cap})|$.
Note that $r \ge 1$ since $\phi_1,\phi_2$ both map $u'$ to $d$.

Since $|V(K^{\prime})|=v_K-1$ and $|E(K^{\prime})|=e_K-1$, the union
graph $\varphi_1(K')\cup\varphi_2(K')$ has
$2(v_K-1)-r$ vertices and
$2(e_K-1)-s$ edges. Thus, for any fixed $d,r$, and $s$, the
number of possible pairs
$\varphi_1,\varphi_2\in\mathcal R_{d,[n]}$ is at most
$O\left(n^{2(v_K-1)-r-1}\right)
=
O\left(n^{2v_K-3-r}\right)$,
and the corresponding expectation is exactly
$p^{2(e_K-1)-s}
=
p^{2e_K-2-s}$. Since $K'$ is fixed, there are only finitely many possible overlap types. For any overlap type having $r$ common vertices and $s$ common edges and summing over $d$, its contribution is $O\left(
n^{2(v_K-1)-r}p^{2e_K-2-s}
\right) $.

\begin{comment}
This yields
\begin{align}\label{eq:bound_1}
    \E\Big[\sum_{d=1}^n\mathcal N([n];d)^2\Big] \lesssim \sum_{r=1}^{v_K-1} \sum_{s \in S_r}\left(
n^{2(v_K-1)-r}p^{2e_K-2-s}
\right) 
\end{align}
\end{comment}
%Summing over $d\in[n]$ gives the total contribution, for fixed $r$ and $s$, as $O\left( n^{2(v_K-1)-r}p^{2e_K-2-s} \right)$.

\textbf{Case 1: $r=1$.} $\varphi_1(K')$ and $\varphi_2(K')$ intersect only at
the vertex $d$. In particular, they have no common edge, so $s=0$.
Hence the total contribution of all pairs $(\varphi_1,\varphi_2)$
with $r=1$ is
$O\left(
n^{2v_K-3}p^{2e_K-2}
\right)$.

\textbf{Case 2: $r\ge 2$.}
Let
\[
\zeta:=|V(K_\cap)\cap V(\varphi_1(\widetilde H))|.
\]
Recall that $K'$ is obtained from $\widetilde H$ by successively
attaching leaves. Thus, among the edges of $K_\cap$, those contained
in $\varphi_1(\widetilde H)$ are at most $\widetilde\Delta \zeta/2$, while
the remaining edges are at most $r-\zeta$. Therefore,
\begin{align*}
s
\le \frac{\widetilde\Delta \zeta}{2}+r-\zeta
\le \frac{\widetilde\Delta r}{2},
\end{align*}
where the last inequality uses $\widetilde\Delta\ge 2$. If $r=2$, we have $s\le 1\le \widetilde\Delta(r-1)-1$. If $r\ge3$, then
\(
s\le {\widetilde\Delta r}/{2}
\le \widetilde\Delta(r-1)-1,
\)
again since $\widetilde\Delta\ge2$. Hence, for any $r \ge 2$,
\begin{align}\label{eq:s_r}
s\le \widetilde\Delta(r-1)-1.
\end{align}

Now we have
\begin{align*}
\E\left[
\sum_{d=1}^n \mathcal N([n];d)^2
\right] & = O\left(
n^{2v_K-3}p^{2e_K-2}
\right)+ O\left(
n^{2(v_K-1)-r}p^{2e_K-2-s}
\right).
\end{align*}
Using~\eqref{eq:s_r}, second summand is
$$O\left(
n^{2(v_K-1)-r}p^{2e_K-2- (\widetilde\Delta(r-1)-1)}
\right) =  O\left( n^{2v_K-3}p^{2e_K-2} \times \Big(n^{1-r} p^{1- \widetilde\Delta(r-1)}\Big) \right)= o(n^{2v_K-3}p^{2e_K-2})$$ using the sparsity of $p$.
\begin{comment}
Since $0<p<1$, every pair with $r\geq2$ contributes at most
\begin{align*}
O\left(
n^{2v_K-2-r}p^{2e_K-2-s}
\right)
&=
O\left(
n^{2v_K-3}p^{2e_K-2}
\cdot n^{1-r}p^{-s}
\right)\leq
O\left(
n^{2v_K-3}p^{2e_K-2}
\cdot n^{1-r}p^{1-\widetilde\Delta(r-1)}
\right).
\end{align*}
Moreover,
\begin{align*}
n^{r-1}p^{\widetilde\Delta(r-1)-1}
&\gg
n^{r-1}
\left(
(n\log n)^{-1/\widetilde\Delta}
\right)^{\widetilde\Delta(r-1)-1}=
n^{1/\widetilde\Delta}
(\log n)^{-(r-1)+1/\widetilde\Delta}
\longrightarrow\infty.
\end{align*}
Thus
$n^{1-r}p^{1-\widetilde\Delta(r-1)}
=
o(1)$,
and hence, for every possible pair $(r,s)$ with $r\geq2$, the
corresponding contribution is
$o\left(
n^{2v_K-3}p^{2e_K-2}
\right)$. Since $K$ is fixed, there are only finitely many possible pairs $(r,s)$.
Combining the cases $r=1$ and $r\geq2$, we obtain
\[
\E\left[
\sum_{d=1}^n \mathcal N([n];d)^2
\right]
\leq
C_K n^{2v_K-3}p^{2e_K-2}.
\]
\end{comment}
Finally, Markov's inequality yields the conclusion.
\end{proof}
\begin{proof}[Proof of \cref{lem:ER-centered-opnorm}]
% It is easily proved by the matrix Bernstein inequality, see for example
% \cite[Theorem~1.4]{tropp2012userfriendly}.
This follows from \cite[Proposition~4.1]{le2018concentration}. Indeed,
\(D=A-\mathbb EA\), where \(A\) is the adjacency matrix of
\(\mathcal G_n\), and the maximal expected degree satisfies
\(d=(n-1)p_n\le np_n\). Since \(np_n\gg\log n\),
\[
\mathbb E\|D\|_{\mathrm{op}}
\lesssim
\sqrt{np_n},
\]
and the result follows by Markov's inequality.
\end{proof}

\begin{proof}[Proof of \cref{lem:hom-to-inj-upper-tail}]
Write $p:=p_n$. Recall that
\[
\mathcal N_{\widetilde H}(S;\mathcal G_n)
=
\sum_{\phi:V(\widetilde H)\hookrightarrow S}
\prod_{\{a,b\}\in E(\widetilde H)}
\mathcal G_n\bigl(\phi(a),\phi(b)\bigr),
\]
where the sum is over injective maps. In particular,
$\mathbb E\mathcal N_{\widetilde H}(S;\mathcal G_n)
=
(|S|)_{\widetilde v}p^{\widetilde e}$.
When an upper-tail result cited below is stated in terms of
unlabelled copies, it applies directly to the labelled count
$\mathcal N_{\widetilde H}(S;\mathcal G_n)$. Indeed, for the fixed graph
$\widetilde H$, the number of injective maps
$\phi:V(\widetilde H)\hookrightarrow S$ whose image equals a given
unlabelled copy is the same for every such copy and depends only on
$\widetilde H$. Consequently, passing from the unlabelled count to
$\mathcal N_{\widetilde H}(S;\mathcal G_n)$ multiplies both the random
count and its corresponding deterministic reference value by the same
fixed positive constant. Therefore, the relative upper-tail events are
unchanged.

We consider separately the cases where $\widetilde H$ is regular and
irregular.

\medskip

\noindent
\textbf{Case 1: $\widetilde H$ is regular.}
Since $\widetilde H$ is connected and $\widetilde\Delta\ge 2$, we have
$\widetilde v\ge 3$. Since $|S|\ge \rho n$ and
$p\gg (n\log n)^{-1/\widetilde\Delta}$, we have
\[
|S|p^{\widetilde\Delta/2}
\gg
\frac{\sqrt n}{\sqrt{\log n}}
\gg
(\log |S|)^{1/(\widetilde v-2)},
\]
uniformly over $|S|\in[\rho n,n]$. Thus, $p$ lies in the regime covered
by
\cite[Theorem 1.2 and Remark 1.6]
{BasakBasu2023UpperTailRegularLocalized}.

Moreover, since $\widetilde H$ is fixed,
\[
(|S|)_{\widetilde v}
=
|S|^{\widetilde v}(1+o(1))
\]
uniformly over $|S|\in[\rho n,n]$. Hence
\[
(1+\varepsilon)(|S|)_{\widetilde v}
\ge
\left(1+\frac{\varepsilon}{2}\right)|S|^{\widetilde v}
\]
throughout the asymptotic regime. It follows from
\cite[Theorem 1.2 and Remark 1.6]
{BasakBasu2023UpperTailRegularLocalized} that
\begin{align*}
&\mathbb P\left(
\mathcal N_{\widetilde H}(S;\mathcal G_n)
>
(1+\varepsilon)p^{\widetilde e}(|S|)_{\widetilde v}
\right)\le
\mathbb P\left(
\mathcal N_{\widetilde H}(S;\mathcal G_n)
>
\left(1+\frac{\varepsilon}{2}\right)
|S|^{\widetilde v}p^{\widetilde e}
\right)\\
&\qquad\le
\exp\left\{
-c_{\varepsilon,\rho}
|S|^2p^{\widetilde\Delta}\log(1/p)
\right\}.
\end{align*}
The last estimate is uniform over $|S|\in[\rho n,n]$.

\medskip

\noindent
\textbf{Case 2: $\widetilde H$ is irregular.}
Let $p_{\widetilde H}^\star(|S|)$ denote the threshold denoted by $p_H$ in
\cite{2026uppertails}, with $H=\widetilde H$ and the graph size there
equal to $|S|$. By its definition,
\[
p_{\widetilde H}^\star(|S|)
=
\max\left\{
|S|^{-1/m(\widetilde H)},
|S|^{-1/\widetilde\Delta-\varepsilon_{\widetilde H}}
\right\},
\qquad
m(\widetilde H)
=
\max\left\{
\frac{e_J}{v_J}:
\varnothing\subsetneq J\subseteq \widetilde H
\right\},
\]
for some $\varepsilon_{\widetilde H}>0$.
Since $m(\widetilde H)\le \widetilde\Delta/2$, there exists
$\eta_{\widetilde H}>0$, depending only on $\widetilde H$, such that
\[
p_{\widetilde H}^\star(|S|)
\le
|S|^{-1/\widetilde\Delta-\eta_{\widetilde H}}.
\]

We first verify that the density assumption in
\cite[Theorem 1.3]{2026uppertails} is satisfied uniformly over
$|S|\in[\rho n,n]$. For every fixed $K>0$,
\[
\frac{p}
{p_{\widetilde H}^\star(|S|)(\log |S|)^K}
\ge
\frac{p\,|S|^{1/\widetilde\Delta+\eta_{\widetilde H}}}
{(\log |S|)^K}.
\]
Since $|S|\in[\rho n,n]$ and
$p\gg(n\log n)^{-1/\widetilde\Delta}$,
the right-hand side tends to infinity uniformly over
$|S|\in[\rho n,n]$. Thus
\[
p\gg
p_{\widetilde H}^\star(|S|)(\log |S|)^K
\]
uniformly over $|S|\in[\rho n,n]$.

We next compare the two logarithmic terms. We have
\[
\log\frac{p}{p_{\widetilde H}^\star(|S|)}
\ge
\Big(
\frac1{\widetilde\Delta}+\eta_{\widetilde H}
\Big)\log |S|
-
\log\frac1p.
\]
Moreover, since
$p\gg(n\log n)^{-1/\widetilde\Delta}$ and
$|S|\in[\rho n,n]$,
\[
\log\frac1p
\le
\left(
\frac1{\widetilde\Delta}+o(1)
\right)\log |S|
\]
uniformly over $|S|\in[\rho n,n]$. Consequently, for some constant
$c_{\rho,\widetilde H}>0$,
\[
\min\Big\{
\log\frac{p}{p_{\widetilde H}^\star(|S|)},
\log\frac1p
\Big\}
\ge
c_{\rho,\widetilde H}\log\frac1p
\]
throughout the asymptotic regime, uniformly over
$|S|\in[\rho n,n]$.

Applying \cite[Theorem 1.3]{2026uppertails} together with
\cite[Proposition 4.3]{2026uppertails} gives
\begin{align*}
&-\log\mathbb P\left(
\mathcal N_{\widetilde H}(S;\mathcal G_n)>
(1+\varepsilon)p^{\widetilde e}(|S|)_{\widetilde v}
\right)\ge
c_{\varepsilon,\rho,\widetilde H}\,
f\,
\min\left\{
\log\frac{p}{p_{\widetilde H}^\star(|S|)},
\log\frac1p
\right\},
\end{align*}
where
$f=f(|S|,p)\ge |S|^2p^{\widetilde\Delta}$
as in \cite[Page 4]{2026uppertails}. This gives
\[
\mathbb P\left(
\mathcal N_{\widetilde H}(S;\mathcal G_n)>
(1+\varepsilon)p^{\widetilde e}(|S|)_{\widetilde v}
\right)
\le
\exp\left\{
-c_{\varepsilon,\rho,\widetilde H}
|S|^2p^{\widetilde\Delta}\log(1/p)
\right\}.
\]
As in the regular case, this estimate is uniform over
$|S|\in[\rho n,n]$.

Finally, by enlarging $C_{\varepsilon,\rho}$ if necessary, the finitely
many remaining values of $n$ are absorbed into the prefactor. Therefore,
for all $n$, uniformly over $S\subseteq[n]$ with $|S|\ge\rho n$,
\[
\mathbb P\left(
\mathcal N_{\widetilde H}(S;\mathcal G_n)>
(1+\varepsilon)p^{\widetilde e}(|S|)_{\widetilde v}
\right)
\le
C_{\varepsilon,\rho}
\exp\left\{
-c_{\varepsilon,\rho,\widetilde H}
|S|^2p^{\widetilde\Delta}\log(1/p)
\right\}.
\]
This proves the result.
\end{proof}

\begin{comment}
\begin{remark}
The assumption
$p\gg (n\log n)^{-1/\widetilde\Delta}$
 is stronger than what is needed for \cref{lem:hom-to-inj-upper-tail}
itself. %Indeed, if $\widetilde H$ is regular, a careful check reveals that the proof only requires
% $p\gg
% n^{-2/\widetilde\Delta}
% (\log n)^{{2}/{(\widetilde\Delta(\widetilde v-2)})}$.
% If $\widetilde H$ is irregular, recall that
% $p_{\widetilde H}^{\star}(n)
% \le
% n^{-1/\widetilde\Delta-\eta_{\widetilde H}}$
% for some $\eta_{\widetilde H}>0$. Hence the same proof applies, for
% example, under
% $p\gg n^{-1/\widetilde\Delta-\delta}$
% for any fixed $0<\delta<\eta_{\widetilde H}$.
We retain the stronger assumption
$p\gg (n\log n)^{-1/\widetilde\Delta}$ since it is needed in the
subsequent application of the lemma (in particular, \cref{lem:}).
\end{remark}
\end{comment}

\subsubsection{Proofs for \cref{OPT}}
\begin{proof}[Proof of \cref{smdef}]
%We first assume $\theta\leq 2$. A direct calculation gives 
Note that,
\[
\mv_{\theta,p}''(x)
=
2\theta-\frac{1}{x(1-x)},
\qquad x\in(0,1).
\]
% Since $x(1-x)\leq 1/4$, if $\theta< 2$, then
% $\mv_{\theta,p}''(x)<0$.
% Hence
% $\mv_{\theta,p}$ is strictly concave on $[0,1]$. Therefore, the
% maximizer is unique.
% Suppose first that $\theta\leq 2$. A direct calculation gives
% \[
% \mv_{\theta,p}''(x)
% =
% 2\theta-\frac{1}{x(1-x)},
% \qquad x\in(0,1).
% \]
% If $\theta<2$, then $\mv_{\theta,p}''(x)<0$ for every
% $x\in(0,1)$, and hence $\mv_{\theta,p}$ is strictly concave.

% If $\theta=2$, then
% \[
% \mv_{2,p}''(x)
% =
% 4-\frac{1}{x(1-x)}
% =
% -\frac{(2x-1)^2}{x(1-x)}
% \leq 0,
% \]
If $\theta \le 2$, $\mv_{\theta,p}''(x)\le0$, with equality only when $(x,\theta)=(1/2, 2)$. %Therefore, for every
%$0<x<y<1$,
%\[ \mv_{\theta,p}'(y)-\mv_{\theta,p}'(x)=\int_x^y \mv_{\theta,p}''(t)\,dt<0.\]
Hence %$\mv_{\theta,p}'$ is strictly decreasing, and
$\mv_{\theta,p}$ is strictly concave, so the maximizer is unique. %Consequently,$\mv_{\theta,p}$ has a unique maximizer whenever $\theta\leq 2$.
Assume now that $\theta>2$. %We may rewrite \[\mv_{\theta,p}(x)=\theta x^2+\mathfrak h(x)+x\log\frac{p}{1-p}+\log(1-p),\]where $\mathfrak h(x):=-x\log x-(1-x)\log(1-x)$.
Set $m:=2x-1$ and define
\[
F(m):=\mv_{\theta,p}\left(\frac{1+m}{2}\right)-\log(1-p),
\qquad m\in[-1,1].
\]
Then
$F(m)-F(-m)=2am$,
where
$a:=\frac12\left(\theta+\log\frac{p}{1-p}\right)$. Since $(\theta,p)\in \Omega$ and $\theta>2$, we have $a \neq 0$.

%We first note that $a\neq 0$. Indeed, $a=0$ is equivalent to $\theta=\log\frac{1-p}{p}$.Since $\theta>2$, this would imply$p<\frac{1}{1+e^2}$ and hence $(\theta,p)\notin\Omega$, a contradiction. 
% For such $p$, the pair $(\theta,p)$ with
% $\theta=\log((1-p)/p)$ is excluded from $\Omega$. Thus, $a\neq 0$.Indeed, $a=0$ is equivalent to
% \[
% \theta=\log\frac{1-p}{p}.
% \]
% Since $\theta>2$, this implies
% $p<1/(1+e^2)$, and hence $(\theta,p)\notin\Omega$,
% a contradiction.
Assume first $a>0$. Then
$F(m)-F(-m)=2am>0$
for every $m >0$, so every global maximizer of $F$ lies in
$[0,1]$. Moreover,
\[
F'(m)
=
\frac{\theta m}{2}+\frac{1}{2}\log\left(\frac{1-m}{1+m}\right)+a,
\qquad
F''(m)
=
\frac{\theta}{2}+\frac{1}{m^2-1}.
\]
Hence, $F'''(m) <0$ for $m>0$. Hence $F''$ is strictly decreasing with $F''(0)>0$ and $F''(m) \rightarrow -\infty$ as $m \rightarrow 1^-$. Consequently $F'$ first increases and then decreases. Since $F^\prime(0)=a>0$ and $\lim_{m\to1^-}F^\prime(m)=-\infty$, $F'$ has exactly one zero in $(0,1)$, implying $F$ has a unique maximizer.

%Since $a>0$, we have $F^\prime(0)=a>0$ and $\lim_{m\to1}F^\prime(m)=-\infty$.  Therefore, $F^\prime$ has an odd number of zeros in $(0,1)$. Suppose $F^\prime$ has at least three zeros in $(0,1)$, then $F^{\prime\prime}$ must have at least two zeros in $(0,1)$.  Clearly, its only root is $\sqrt{1-{2}/{\theta}}\in(0,1)$, so $F^\prime$ has exactly one zero in $(0,1)$. Thus $F$ has a unique maximizer on $[-1,1]$.

% Define
% $\psi(m):=\frac{\theta m}{2}+\frac{1}{2}\log\left(\frac{1-m}{1+m}\right)$,
% so that
% $F'(m)=\psi(m)+a$.

% Since $\theta>2$, the equation $F''(m)=0$ has a unique solution in $(0,1)$, denoted by
% $m_0:=\sqrt{1-\frac{2}{\theta}}\in(0,1)$.
% Therefore, $\psi$ is strictly increasing on $(0,m_0)$ and strictly
% decreasing on $(m_0,1)$. Also,
% \[
% \psi(0)=0,
% \qquad
% \lim_{m\uparrow1}\psi(m)=-\infty.
% \]
% Since $a>0$, we have
% $F'(0)=\psi(0)+a=a>0$.
% Moreover, $\psi(m)+a>0$ on $[0,m_0]$, while
% $\lim_{m\uparrow1}\bigl(\psi(m)+a\bigr)=-\infty$.
% Since $\psi$ is strictly decreasing on $(m_0,1)$, there exists a unique
% $m_*\in(m_0,1)$ such that
% $\psi(m_*)+a=0$.
% Consequently,
% \[
% F'(m)>0 \quad\text{for }m\in[0,m_*),
% \qquad
% F'(m)<0 \quad\text{for }m\in(m_*,1).
% \]
% Thus, $F$ is strictly increasing on $[0,m_*]$ and strictly decreasing
% on $[m_*,1]$. Therefore, $m_*$ is the unique maximizer of $F$ on
% $[0,1]$. Furthermore, for every $m\in[-1,0)$,
% $F(m)<F(-m)\leq F(m_*)$.
% Hence, $m_*$ is also the unique maximizer of $F$ on $[-1,1]$.

If $a<0$, we apply the same argument to the function $m\mapsto F(-m)$,
whose corresponding coefficient $-a>0$. Therefore, $F$, and hence $\mv_{\theta,p}$, again
has a unique maximizer on $[-1,1]$.

%It follows that $\mv_{\theta,p}$ has a unique maximizer on $[0,1]$, which we denote by $s_{\theta,p}$. Consequently, the map $(\theta,p)\longmapsto s_{\theta,p}$ is well-defined on $\Omega$.
\end{proof}
\begin{proof}[Proof of \cref{lem:rootlim}]
We first prove a general fact. Let \((\theta_k,p_k)\in\Omega\) satisfy
\((\theta_k,p_k)\to(\theta_\infty,p_\infty)\in[0,\infty)\times(0,1)\), and suppose
\(s_{\theta_k,p_k}\to s^*\) along a subsequence. Since \(s_{\theta_k,p_k}\)
maximizes \(\mv_{\theta_k,p_k}\) and \(\mv_{\theta_k,p_k}\to
\mv_{\theta_\infty,p_\infty}\) uniformly on \([0,1]\), for every \(x\in[0,1]\),
$\mv_{\theta_\infty,p_\infty}(x)
\le
\mv_{\theta_\infty,p_\infty}(s^*)$.
Hence \(s^*\) is a global optimizer of \(\mv_{\theta_\infty,p_\infty}\).

\begin{itemize}
\item[(i)]
Let \((\theta_k,p_k)\to(\theta_\infty,p_\infty)\in\Omega\). By \cref{c2-1},
\(\mv_{\theta_\infty,p_\infty}\) has a unique optimizer, namely
\(s_{\theta_\infty,p_\infty}\). Therefore every subsequential limit of
\(s_{\theta_k,p_k}\) equals \(s_{\theta_\infty,p_\infty}\), so
\(s_{\theta_k,p_k}\to s_{\theta_\infty,p_\infty}\).

\item[(ii)]
Fix \(0<p<1/(1+e^2)\), and set \(\theta_*:=\log\frac{1-p}{p}\). If
\(\theta_k\to\theta_*\) with \(\theta_k>\theta_*\), then every subsequential
limit of \(s_{\theta_k,p}\) is a global optimizer of \(\mv_{\theta_*,p}\).
By \cref{c2-1}, every global optimizer of
\(\mv_{\theta_*,p}\) is one of $\frac{1 \pm t_{\theta_*,p}}2$,
%\[\frac{1-t_{\theta_*,p}}2,\qquad\frac{1+t_{\theta_*,p}}2,\]
where \(t_{\theta_*,p}>0\) is the unique positive solution of $x=\tanh\Bigg(\frac{x}{2}\log\frac{1-p}{p}\Bigg)$.
Moreover, for \(\theta>\theta_*\), the unique optimizer satisfies
\(s_{\theta,p}>1/2\). Hence the right limit must be
\((1+t_{\theta_*,p})/2\), meaning that
\[
\lim_{\theta\downarrow\theta_*}s_{\theta,p}
=
\frac{1+t_{\theta_*,p}}{2}.
\]
Similarly,
$\lim_{\theta\uparrow\theta_*}s_{\theta,p}
=
\frac{1-t_{\theta_*,p}}{2}$.
The two limits are different, so \(s\) does not admit a continuous extension on \(\overline\Omega\).
\item[(iii)]
% We again use the preliminary compactness argument.
First let
\(p_k\downarrow 1/(1+e^2)\), and set
\(\theta_k=\log\frac{1-p_k}{p_k}\). Then
\((\theta_k,p_k)\to(2,1/(1+e^2))\). Hence every subsequential limit of
\(s_{\theta_k,p_k}\) is a global optimizer of
\(\mv_{2,1/(1+e^2)}\). By \cref{c2-1}, this limiting variational problem has
the unique global optimizer \(1/2\). Therefore
\[
\lim_{p\downarrow 1/(1+e^2)}
s_{\log\frac{1-p}{p},p}
=
\frac12 .
\]

Similarly, if \(p_k\uparrow 1/2\) and
\(\theta_k=\log\frac{1-p_k}{p_k}\), then
\((\theta_k,p_k)\to(0,1/2)\). Every subsequential limit of
\(s_{\theta_k,p_k}\) is a global optimizer of \(\mv_{0,1/2}\). Since
\(\mv_{0,1/2}\) has the unique global optimizer \(1/2\), we get
\[
\lim_{p\uparrow 1/2}
s_{\log\frac{1-p}{p},p}
=
\frac12 .
\]
\item[(iv)] For fixed \(p\in(0,1)\), if \(\theta_k\downarrow0\), then every
subsequential limit of \(s_{\theta_k,p}\) is a global optimizer of
\(\mv_{0,p}\). Since this optimizer is unique and equals \(s_{0,p}\), we have
$\lim_{\theta\downarrow0}s_{\theta,p}=s_{0,p}$.
Moreover, a direct calculation shows $s_{0,p}=p$, since $\mv_{0,p}(x)
=
-x\log\frac{x}{p}
-(1-x)\log\frac{1-x}{1-p}$ which is uniquely maximized at $x=p$.

Now write
$I_p(x):=
x\log\frac{x}{p}
+
(1-x)\log\frac{1-x}{1-p}.$
Then \(\mv_{\theta,p}(x)=\theta x^2-I_p(x)\) and
$I_p(1)=-\log(p)$.

Since \(s_{\theta,p}\) maximizes \(\mv_{\theta,p}\), we have
$\mv_{\theta,p}(s_{\theta,p})\ge \mv_{\theta,p}(1)$.
Hence
\[
\theta s_{\theta,p}^2-I_p(s_{\theta,p})
\ge
\theta-I_p(1)
=
\theta+\log(p) .
\]
Rearranging gives
\[
\theta(1-s_{\theta,p}^2)
\le
I_p(1)-I_p(s_{\theta,p})
\le
I_p(1)
=
-\log(p) .
\]
Therefore $0\le 1-s_{\theta,p}
\le
1-s_{\theta,p}^2
\le
-\log(p)/\theta$.
Letting \(\theta\to\infty\), we obtain
$\lim_{\theta\to\infty}s_{\theta,p}=1$.
% ```

% \item[(iii)]
% The proof is the same. If \(p\downarrow 1/(1+e^2)\) and
% \(\theta=\log\frac{1-p}{p}\), then \((\theta,p)\to(2,1/(1+e^2))\), and every
% subsequential limit of \(h(\theta,p)\) is a global optimizer of
% \(\mv_{2,1/(1+e^2)}\). Since this optimizer is unique, the limit exists.
% Similarly, if \(p\uparrow1/2\), then \((\theta,p)\to(0,1/2)\), and every
% subsequential limit is the unique global optimizer of \(\mv_{0,1/2}\). Finally,
% for fixed \(p\in(0,1)\), if \(\theta\downarrow0\), every subsequential limit of
% \(h(\theta,p)\) is the unique global optimizer of \(\mv_{0,p}\); hence
% \[
% \lim_{\theta\downarrow0}h(\theta,p)=s_{0,p}.
% \]
\item[(v)] Fix $p\in(0,1)$ and let $\theta_1<\theta_2$. We first note that
$s_{\theta_1,p}\neq s_{\theta_2,p}$. Indeed, suppose that
$s_{\theta_1,p}=s_{\theta_2,p}=:s$. Since any maximizer of
$\mv_{\theta,p}$ lies in $(0,1)$, the first-order condition gives
\[
2\theta_i s
=
\log\frac{s(1-p)}{p(1-s)},
\qquad i=1,2.
\]
Therefore,
$2(\theta_2-\theta_1)s=0$,
which is a contradiction since $\theta_1<\theta_2$ and $s>0$.
Now since \(s_{\theta_i,p}\) uniquely maximizes \(\mv_{\theta_i,p}\), we have $\mv_{\theta_1,p}(s_{\theta_1,p})>\mv_{\theta_1,p}(s_{\theta_2,p})$ and $\mv_{\theta_2,p}(s_{\theta_2,p})>\mv_{\theta_2,p}(s_{\theta_1,p})$.
Adding and cancelling common terms gives
$\theta_1s_{\theta_1,p}^2+\theta_2s_{\theta_2,p}^2>
\theta_1s_{\theta_2,p}^2+\theta_2s_{\theta_1,p}^2$,
or equivalently \((\theta_2-\theta_1)(s_{\theta_2,p}^2-s_{\theta_1,p}^2)>0\).
Hence the map $\theta\mapsto s_{\theta,p}$ is strictly increasing on $J$.
\end{itemize}

\end{proof}
\begin{proof}[Proof of \cref{Wc}]
Define $J_p(t)
:=
\sqrt t\log\frac{\sqrt t}{p}
+
(1-\sqrt t)\log\frac{1-\sqrt t}{1-p}$, $t\in[0,1]$. A direct calculation gives
\[
J_p''(t)
=
\frac{1}{4t^{3/2}}
\left[
\frac{1}{1-\sqrt t}
-
\log\frac{\sqrt t(1-p)}{p(1-\sqrt t)}
\right],
\qquad t\in(0,1).
\]
In particular,
$J_p''\left(\frac14\right)
=
2\left(2-\log\frac{1-p}{p}\right)<0$,
where the inequality follows from
$p<(1+e^2)^{-1}$. Hence $J_p$ is strictly concave on some open
interval $U$ containing $1/4$.

Choose $\varepsilon>0$ sufficiently small so that $a,b\in U$, where
$a:=1/4-\varepsilon$, and
$b:=1/4+\varepsilon$. Then
${(a+b)}/{2}=1/4$
and strict concavity gives
\begin{equation}\label{eq:strict-concavity-Jp}
\frac12J_p(a)+\frac12J_p(b)
<
J_p\left(\frac14\right).
\end{equation}

Define
$f_{a,b}(u)
:=
\sqrt a\,\mathbf 1_{[0,1/2]}(u)
+
\sqrt b\,\mathbf 1_{(1/2,1]}(u)$.
Since $W_\circ=2$ on the two diagonal blocks and vanishes on the
off-diagonal blocks,
\[
G_{W_\circ,\mathbf 1_1}(f_{a,b})
=
2\left(\int_0^{1/2}\sqrt a\,du\right)^2
+
2\left(\int_{1/2}^1\sqrt b\,du\right)^2
=
\frac{a+b}{2}
=
\frac14.
\]
Thus, $f_{a,b}$ is feasible. Moreover, its objective value is
\[
\frac12J_p(a)+\frac12J_p(b)
<
J_p\left(\frac14\right),
\]
by \eqref{eq:strict-concavity-Jp}. On the other hand, if $f\equiv C$ is feasible, then
\[
\frac14
=
G_{W_\circ,\mathbf 1_1}(f)
=
C^2\int_{[0,1]^2}W_\circ(x,y)\,dx\,dy
=
C^2.
\]
Since $C\in[0,1]$, necessarily $C=1/2$. The objective value of this
constant function is precisely
$J_p\left(\frac14\right)$.
Therefore, the only constant feasible function has strictly larger
objective value than the nonconstant feasible function $f_{a,b}$.
Consequently, no optimizer is constant. The strict-concavity argument above is the $\gamma=2$ specialization of
\cite[Lemma A.1]{lubetzky2015replica}, and the subsequent two-block
construction %is adapted from 
\cite[Example 1]{SomabhaBhattacharya2020}.
\end{proof}

\bibliographystyle{apalike}
\bibliography{template, Newtemplate}

\end{document}